\documentclass[11pt]{article}
\usepackage{amsfonts,amsmath,amssymb}
\usepackage{url}
\usepackage{epsfig,latexsym,graphicx}
\usepackage{esint}
\usepackage{graphicx,latexsym,amssymb,amsmath,amsfonts}
\newtheorem{maintheorem}{Theorem}

\newtheorem{theorem}{Theorem}[section]
\newtheorem{lemma}[theorem]{Lemma}

\newtheorem{proposition}[theorem]{Proposition}
\newtheorem{observation}[theorem]{Observation}

\newtheorem{definition}[theorem]{Definition}

\newenvironment{proof}{\noindent{\bf Proof.\,\ }}{\hfill\mbox{$\Box$}\smallskip}

\renewcommand{\P}{\mathbb{P}}

\newcommand{\A}[1]{\mathcal{A}^{(#1)}}
\newcommand{\tk}{\tilde{K}_X}
\newcommand{\fs}{\mathfrak{S}}
\newcommand{\cd}{\mathcal{D}}
\newcommand{\ch}{\mathcal{H}}
\newcommand{\ce}{\mathcal{E}}
\newcommand{\cg}{\mathcal{G}}
\newcommand{\cj}{\mathcal{J}}
\newcommand{\cm}{\mathcal{M}}
\newcommand{\Z}{\mathbb{Z}}
\newcommand{\N}{\mathbb{N}}
\newcommand{\R}{\mathbb{R}}
\newcommand{\cf}{\mathcal{F}}

\begin{document}
\title{Lipschitz embeddings of random sequences}

\author{Riddhipratim Basu and Allan Sly}

\maketitle

\begin{abstract}

We develop a new multi-scale framework flexible enough to solve a number of problems involving embedding random sequences into random sequences.
Grimmett, Liggett and Richthammer~\cite{GLR:10} asked whether there exists an increasing $M$-Lipschitz embedding from one i.i.d. Bernoulli sequences into an independent copy with positive probability.  We give a positive answer for large enough $M$.
A closely related problem is to show that two independent Poisson processes on $\mathbb{R}$ are roughly  isometric (or quasi-isometric).  Our approach also applies in this case answering a conjecture of Szegedy and of Peled~\cite{Peled:10}.  Our theorem also gives a new proof to Winkler's compatible sequences problem.  %Our approach is more abstract than previous work in this area and does not explicitly consider the geometry of the problems.
\end{abstract}

\section{Introduction}\label{s:intro}
With his compatible sequences and clairvoyant demon scheduling problems Winkler introduced a fascinating class of dependent or ``co-ordinate'' percolation type problems (see e.g.~\cite{BBS:00,CTP:00,Grimmett:09,Winkler:00}).  These, and other problems in this class, can be interpreted either as embedding one sequence into another according to certain rules or as oriented percolation problems in $\Z^2$ where the sites are open or closed according to random variables on the co-ordinate axes.
A natural question of this class posed by Grimmett, Liggett and Richthammer~\cite{GLR:10} asks whether there exists a Lipshitz embedding of one Bernoulli sequence into another.  The following theorem answers the main question of~\cite{GLR:10}.

\begin{maintheorem}\label{t:embed}
Let $\{X_i\}_{i\in\Z}$ and $\{Y_i\}_{i\in\Z}$ be independent sequences of independent identically distributed  $\mbox{Ber}(\frac12)$ random variables.  For sufficiently large $M$ there almost surely there exists a strictly increasing function $\phi:\mathbb{Z}\to\mathbb{Z}$ such that $X_i=Y_{\phi(i)}$ and $1\leq \phi(i)-\phi(i-1)\leq M$ for all $i$.
\end{maintheorem}

The original question of~\cite{GLR:10} was slightly different asking for a positive probability on the natural numbers with the condition $\phi(0)=0$, which is implied by our theorem (and is equivalent by ergodic theory considerations).  Among other results they showed that Theorem~\ref{t:embed} fails in the case of $M=2$.  In a series of subsequent works Grimmett, Holroyd and their collaborators~\cite{DDGHS:10,Grimmett:09,GriHol:10a,GriHol:10,GriHol:12,HolMar:12} investigated a range of related problems including when one can embed $\mathbb{Z}^d$ into site percolation in $\mathbb{Z}^D$ and showed that this was possible almost surely for $M=2$ when $D>d$ and the the site percolation parameter was sufficiently large but almost surely impossible for any $M$ when $D \leq d$.  Recently Holroyd and Martin showed that a comb can be embedded in $\mathbb{Z}^2$.  Another important series of work in this area involves embedding words into higher dimensional percolation clusters~\cite{BenKes:95,dLSS:11,KSV:98,KSZ:01}.  Despite this impressive progress the question of embedding  one random sequence into another remained open.  The difficulty lies in the presence of long strings of ones and zeros on all scales in both sequences which must be paired together.

In a  similar vein is the question of a rough, (or  quasi-), isometry of two independent Poisson processes.  Informally, two metric spaces are roughly isometric if their metrics are equivalent up to multiplicative and additive constants.  The formal definition, introduced by Gromov~\cite{Gromov:81} in the case of groups and more generally by Kanai~\cite{Kanai:85}, is as follows.
\begin{definition}
We say two metric spaces $X$ and $Y$ are roughly isometric with parameters $(M,D,C)$ if there exists a mapping $T:X\to Y$ such that for any $x_1, x_2 \in X$,
\[
\frac1M d_X(x_1,x_2)-D \leq d_Y(T(x_1),T(x_2)) \leq M d_X(x_1,x_2) + D,
\]
and for all $y\in Y$ there exists $x\in X$ such that $d_Y(T(x),y)\leq C$.
\end{definition}

Originally Ab\'{e}rt~\cite{Abert:08} asked whether two independent infinite components of bond percolation on a Cayley graph are roughly isometric.  Szegedy asked the problem when these sets are independent Poisson process in $\R$  (see~\cite{Peled:10} for a fuller description of the history of the problem).
The most important progress on this question is by Peled~\cite{Peled:10} who showed that Poisson processes on $[0,n]$ are roughly isometric with parameter $M=\sqrt{\log n}$.  The question of whether two independent Poisson processes on $\R$ are roughly isometric for fixed $(M,D,C)$ was the main open question of~\cite{Peled:10}.  We prove that this is indeed the case.

\begin{maintheorem}\label{t:RIPoisson}
Let $X$ and $Y$ be independent Poisson processes on $\R$ viewed as metric spaces.  There exists $(M,D,C)$ such that almost surely $X$ and $Y$ are $(M,D,C)$-roughly isometric.
\end{maintheorem}

Again the challenge is to find a good matching on all scales, in this case to  the long gaps in the each point processes with ones of proportional length in the other.  The isometries we find are also weakly increasing answering a further question of Peled~\cite{Peled:10}.

Our final result is the compatible sequence problem of Winkler.  Given two independent sequence $\{X_i\}_{i\in\N}$ and $\{Y_i\}_{i\in\N}$ of independent identically distributed  $\mbox{Ber}(q)$ random variables we say they are compatible if after removing some zeros from both sequences, there is no index with a 1 in both sequence.  Equivalently there exist increasing subsequences $i_1,i_2,\ldots,$ (respectively $i_1'\ldots$) such that if $X_j=1$ then $j=i_k$ for some $k$ (resp. if $Y_j=1$ then $j=i_k'$) so that for all $k$, we have $X_{i_k}Y_{i_k}=0$.  We give a new proof of the following result of G\'{a}cs~\cite{Gacs:04}.

\begin{maintheorem}\label{t:compatible}
For sufficiently small $q>0$ two independent $\mbox{Ber}(q)$ sequences $\{X_i\}_{i\in\N}$ and $\{Y_i\}_{i\in\N}$ are compatible with positive probability.
\end{maintheorem}

Our proof is different and we believe more transparent enabling us to state a concise induction step (see Theorem~\ref{induction}).  Other recent progress was made on this problem by Kesten et. al.~\cite{KLSV:12} constructing sequences which are compatible with a random sequence with positive probability.

Each of these results follows from an abstract theorem in the next section which applies to range of different models.  Like essentially all results in this area, our approach is multi-scale using renormalization.  The novelty of our approach is that, as far as possible, we ignore the anatomy of what makes different configurations difficult to embed and instead consider simply the probability that they can be embedded into a random block proving recursive power-law estimates for these quantities.

{\bf Independent Results}  Two other researchers have also solved some of these problems independently.  Vladas Sidoravicius~\cite{Sidoravicius:12} solved the same set of problems and described his approach to us which is quite different from ours.
His work is based on a different multi-scale approach, proving that for
certain choices of parameters $p_1$ and $p_2$ one can see random binary
sequence sampled with parameter $p_1$ in the scenery determined by another
binary sequence sampled with parameter $p_2$, with positive probability.
This generalizes the main theorem of~\cite{KSV:12}. This result, and a slight modification of it then implies Theorems~\ref{t:embed},~\ref{t:RIPoisson} and~\ref{t:compatible}.

Shortly before first uploading this paper Peter G\'{a}cs sent us a draft of his paper~\cite{Gacs:12} solving Theorem~\ref{t:embed}.  His approach extends his work on the scheduling problem~\cite{Gacs:11}.  The proof is geometric taking a percolation type view and involves a complex multi-scale system of structures called ``mazeries'' involving ``walls'', ``holes'' and ``traps''.

Our work was done completely independently of both.

\subsection{General Theorem}\label{s:abstract}

To apply to a range of problems we need to consider larger alphabets of symbols.  Let $\mathcal{C}^{\mathbb{X}}=\{C_1,C_2,\ldots\}$ and $\mathcal{C}^{\mathbb{Y}}=\{C'_1,C'_2,\ldots\}$ be a pair of countable alphabets and let $\mu^{\mathbb{X}}$ and $\mu^{\mathbb{Y}}$ be probability measures on $\mathcal{C}^{\mathbb{X}}$ and $\mathcal{C}^{\mathbb{Y}}$ respectively.

We will suppose also that we have a relation $\mathcal{R}\subseteq \mathcal{C}^{\mathbb{X}}\times \mathcal{C}^{\mathbb{Y}}$. If $(C_i,C'_k)\in \mathcal{R}$, we denote this by $C_i \hookrightarrow C'_k$. Let $G_0^{\mathbb{X}}\subseteq \mathcal{C}^{\mathbb{X}}$ and $G_0^{\mathbb{Y}}\subseteq \mathcal{C}^{\mathbb{Y}}$ be two given subsets such that $C_i\in G_0^{\mathbb{X}}$ and $C'_k\in G_0^{\mathbb{Y}}$ implies $C_i\hookrightarrow C'_k$. Symbols in $G_0^{\mathbb{X}}$ and $G_0^{\mathbb{Y}}$ will be referred  to as ``good''.

\subsubsection{Definitions}
Now let $\mathbb{X}=(X_1,X_2,\ldots)$ and $\mathbb{Y}=(Y_1,Y_2,\ldots)$ be two sequences of symbols coming from the alphabets $\mathcal{C}^{\mathbb{X}}$ and $\mathcal{C}^{\mathbb{Y}}$ respectively. We will refer to such sequences as an $\mathbb{X}$-sequence and a $\mathbb{Y}$-sequence respectively. For $1\leq i_1 < i_2$, we call the subsequence $(X_{i_1},X_{i_1+1},\ldots , X_{i_2})$ the ``$[i_1,i_2]$-segment" of $\mathbb{X}$ and denote it by $\mathbb{X}^{[i_1,i_2]}$. We call $\mathbb{X}^{[i_1,i_2]}$   a ``good" segment if $X_i\in G_0^{\mathbb{X}}$ for $i_1\leq i\leq i_2$ and similarly for $\mathbb{Y}$.

Let $R$ be a fixed constant. Let $R_0=2R$, $R_0^{-}=1=$, $R_0^{+}=3R^2$.

\begin{definition}
\label{mapstodefinition}
Let $\mathbb{X}$ and $\mathbb{Y}$ be sequences as above. We say that $\mathbb{X}$ $R$-embeds or $R$-maps into $\mathbb{Y}$, denoted $\mathbb{X}\hookrightarrow_R \mathbb{Y}$ if there exists $0=i_0<i_1<i_2<\ldots$ and $0=i'_0<i'_1<i'_2< \ldots$ satisfying the following conditions.

\begin{enumerate}
\item For each $r\geq 0$, either $i_{r+1}-i_{r}=i'_{r+1}-i'_{r}=1$ or $i_{r+1}-i_{r}=R_0$ or $i'_{r+1}-i'_{r}=R_0$.
\item If $i_{r+1}-i_{r}=i'_{r+1}-i'_{r}=1$, then $X_{i_{r}+1}\hookrightarrow Y_{i'_{r}+1}$.
\item If $i_{r+1}-i_{r}=R_0$, then $R_0^{-}\leq i'_{r+1}-i'_{r}\leq R_0^{+}$, and both $\mathbb{X}^{[i_r+1,i_{r+1}]}$ and $\mathbb{Y}^{[i'_{r}+1,i'_{r+1}]}$ are good segments.
\item If $i'_{r+1}-i'_{r}=R_0$, then $R_0^{-}\leq i_{r+1}-i_{r}\leq R_0^{+}$, and both $\mathbb{X}^{[i_r+1,i_{r+1}]}$ and $\mathbb{Y}^{[i'_{r}+1,i'_{r+1}]}$ are good segments.
\end{enumerate}
\end{definition}

Throughout we will use a fixed $R$ defined in Theorem~\ref{metatheorem} and  will simply refer to mappings and write that $\mathbb{X}\hookrightarrow \mathbb{Y}$ except where it is ambiguous. The above definition can be modified in an obvious way to define $X\hookrightarrow Y$ when $X$ and $Y$ are finite $\mathbb{X}$ and $\mathbb{Y}$ subsequences respectively.  A key element in our proof is tail estimates on the probability that we can map a block $X$ into a random block $Y$ and so we make the following definition.

\begin{definition} For $X \in\mathcal{C^{\mathbb{X}}}$, we define the \emph{embedding probability} of $X$ as $S_0^{\mathbb{X}}(X)=\mathbb{P}(X\hookrightarrow Y|X)$ where $Y\sim\mu^{\mathbb{Y}}$.  We define $S_0^{\mathbb{Y}}(Y)$ similarly and suppress the notation $\mathbb{X},\mathbb{Y}$ when the context is clear.
\end{definition}

\subsubsection{General Theorem}
We can now state our general theorem which will imply the main results of the paper as shown in \S~\ref{s:applicaiton}.

\begin{maintheorem}[General Theorem]
\label{metatheorem}
There exist positive constants $\beta$, $\delta$, $m$, $R$ such that for all large enough $L_0$ the following hold.
Let $X\sim \mu^{\mathbb{X}}$ and $Y\sim \mu^{\mathbb{Y}}$ be distributions on alphabets such that for all $k\geq L_0$,
\begin{equation}
\label{mux}
\mu^{\mathbb{X}}(\{C_{k+1},C_{k+2},\ldots \})\leq \frac{1}{k}, \quad \mu^{\mathbb{Y}}(\{C'_{k+1},C'_{k+2},\ldots \})\leq \frac{1}{k}
\end{equation}
Suppose the following conditions are satisfied
\begin{enumerate}
\item
For all $0< p \leq 1-L_0^{-1}$,
\begin{equation}
\label{tailx}
\P(S_0^{\mathbb{X}}(X)\leq p)\leq p^{m+1}L_0^{-\beta},\quad
\P(S_0^{\mathbb{Y}}(Y)\leq p)\leq p^{m+1}L_0^{-\beta}.
\end{equation}
\item Most blocks are good,
\begin{equation}
\label{goodxstatement}
\P(X\in G_0^{\mathbb{X}})\geq 1-L_0^{-\delta}, \quad \P(Y\in G_0^{\mathbb{Y}})\geq 1-L_0^{-\delta}.
\end{equation}
\end{enumerate}
Then for $\mathbb{X}=(X_1,X_2,\ldots)$ and $\mathbb{Y}=(Y_1,Y_2,\ldots)$, two independent sequences with laws $\mu^{\mathbb{X}}$ and $\mu^{\mathbb{Y}}$ respectively, we have
$$\mathbb{P}(\mathbb{X}\hookrightarrow_R \mathbb{Y})>0.$$
\end{maintheorem}

\subsection{Proof Outline}
The proof makes use of a number of parameters, $\alpha,\beta,\delta, m, k_0, R$ and $L_0$ which must satisfy a number of relations described in the next subsection.
Our proof is multi-scale and divides the sequences into blocks on a series of doubly exponentially growing length scales $L_j=L_0^{\alpha^j}$ for $j\geq 0$ and at each of these levels we define a notion of a ``good'' block.  Single characters in the base sequences $\mathbb{X}$ and $\mathbb{Y}$ constitute the level 0 blocks.

Suppose that we have constructed the blocks up to level $j$ denoting the sequence as $(X_1^{(j)},X_2^{(j)}\ldots)$.
In \S~\ref{s:prelim} we give a construction of $(j+1)$-level blocks out of $j$-level sub-blocks in such way that the blocks are independent and apart, from the first block, identically distributed and that the first and last $L_j^3$ sub-blocks of each block are good.  For more details see \S~\ref{s:prelim}.

At each level we distinguish a set of blocks to be good.  In particular this will be done in such a way that at each level \emph{any} good block maps into \emph{any} other good block.  Moreover, any segment of $R_j=4^j (2R)$ good $\mathbb{X}$-blocks will map into any segment of $\mathbb{Y}$-blocks of length between $R_j^-=4^j(2-2^{-j})$ and $R_j^+=4^jR^2(2+2^{-j})$ and vice-versa.
This property of certain mappings will allow us to avoid complicated conditioning issues.  Moreover, being able to map good segments into shorter or longer segments will give us the flexibility to find suitable partners for difficult to embed blocks and to achieve improving estimates of the probability of mapping random $j$-level blocks $X \hookrightarrow Y$.  In \S~\ref{s:prelim} we describe how to define good blocks.

The proof then involves a series of recursive estimate at each level given in \S~\ref{s:recursive}.  We ask that at level $j$ the probability that a block is good is at least $1-L_j^{-\delta}$ so that the vast majority of blocks are good.  Furthermore, we show tail bounds on the embedding probabilities showing that for $0<p\leq 1-L_j^{-1}$,
\[
\P(S_j^{\mathbb{X}}(X)\leq p)\leq p^{m+2^{-j}}L_j^{-\beta}
\]
where $S_j^{\mathbb{X}}(X)$ denotes the $j$-level embedding probability $\P[X \hookrightarrow Y|X]$ for $X,Y$ random independent $j$-level blocks.  We show the analogous bound for $\mathbb{Y}$-blocks as well.  This is essentially the best we can hope for -- we cannot expect a better than power-law bound here because of the probability of occurrences of sequences of repeating symbols in the base 0-level sequence of length $C \log (L_j^{\alpha})$ for large $C$.
We also ask that good blocks have the properties described above and that the length of blocks satisfy an exponential tail estimate.
The full inductive step is given in \S~\ref{induction}.  Proving this constitutes the main work of the paper.

The key quantitative estimate in the paper is Lemma~\ref{l:totalSizeBound} which follows directly from the recursive estimates and bounds the chance of a block having an excessive length, many bad sub-blocks or a particularly difficult collection of sub-blocks measured by the product of their embedding probabilities.  In order to achieve the improving embedding probabilities at each level we need to take advantage of the flexibility in mapping a small collection of very bad blocks to a large number of possible partners by mapping the good blocks around them into longer or shorter segments using the inductive assumptions.   To this effect we define families of mappings between partitions to describe such potential mappings.  Because $m$ is large and we take many independent trials  the estimate at the next level improves significantly.  Our analysis is split  into 5 different cases.

To show that good blocks have the required properties we construct them so that they have at most $k_0$ bad sub-blocks all of which are ``semi-bad" (defined in \S~\ref{s:prelim}) in particular with embedding probability at least $(1-\frac{1}{20k_0R_{j+1}^{+}})$.  We also require that each subsequence of sub-blocks is ``strong'' in that every semi-bad block maps into a large proportion of the sub-blocks.  Under these condition we show that for any good blocks $X$ and $Y$ at least one of our families of mappings gives an embedding.  This holds similarly for embeddings of segments of good blocks.

To complete the proof we note that with positive probability $X_1^{(j)}$ and $Y_1^{(j)}$ are good for all $j$ with positive probability.  This gives a sequence of embeddings of increasing segments of $\mathbb{X}$ and $\mathbb{Y}$ and by taking a converging subsequential limit we can construct an $R$-embedding of the infinite sequences completing the proof.

We can also give deterministic constructions using our results.  In Section~\ref{s:deterministic} we construct a deterministic sequence which has an $M$-Lipshitz embedding into a random binary sequence in the sense of Theorem~\ref{t:embed} with positive probability.  Similarly, this approach  gives a binary sequence with a positive density of ones which is compatible sequence with a random $\mbox{Ber}(q)$ sequence in the sense of Theorem~\ref{t:compatible} for small enough $q>0$ with positive probability.

\subsubsection{Parameters}\label{s:parameters}
Our proof involves a collection of parameters $\alpha,\beta,\delta,k_0,m$ and $R$ which must satisfy a system of constraints.  The required constraints are
\[
\alpha>9, \delta>2\alpha \vee 48, \beta>\alpha(\delta+1), m>9\alpha\beta, k_0> 36\alpha\beta, R> 6(m+1).
\]
To fix on a choice we will set
\begin{equation}\label{e:parameters}
\alpha=10, \delta=50, \beta=600, m=60000, k_0=300000, R=400000.
\end{equation}
Given these choices we then take $L_0$ to be sufficiently large.  We did not make a serious attempt to optimize the parameters or constraints and indeed at times did not in order simplify the exposition.

\subsection{Organization of the paper}
In Section~\ref{s:applicaiton} we show how to derive Theorems~\ref{t:embed}, \ref{t:RIPoisson} and~\ref{t:compatible} from our general Theorem~\ref{metatheorem}.  In Section~\ref{s:prelim} we describe our block constructions and formally define good blocks.  In Section~\ref{s:recursive} we state the main recursive theorem and show that it implies Theorem~\ref{metatheorem}.  In Sections~\ref{s:notation} and \ref{s:construction} we construct a collection of generalized mappings of partitions which we will use to describe our mappings between blocks.  In Section~\ref{s:tailestimate} we prove the main recursive tail estimates on the embedding probabilities.  In Section~\ref{s:length} we prove the recursive length estimates on the blocks.  In Section~\ref{s:good} we show that good blocks have the required inductive properties.  Finally in Section~\ref{s:deterministic} we describe how these results yield deterministic sequences with positive probabilities of $M$-Lipshitz embedding or being a compatible sequence.

\section{Applications to Lipschitz Embeddings, Rough Isometries and Compatible Sequences}\label{s:applicaiton}
In this section we show how Theorem~\ref{metatheorem} can be used to derive our three main results.
\subsection{Lipschitz Embeddings}
\subsubsection{Defining the sequences $\mathbb{X}$ and $\mathbb{Y}$ and the alphabets $\mathcal{C}^{\mathbb{X}}$ and $\mathcal{C}^{\mathbb{Y}}$}
Let $X^{*}=\{X_i^{*}\}_{i\geq 1}$ and $Y^{*}=\{Y_i^{*}\}_{i\geq 1}$ be two independent sequences of i.i.d. $Ber(\frac{1}{2})$ variables. Let $\tilde{Y^{*}}=\{\tilde{Y_i^{*}}\}$ be the sequence given by $\tilde{Y_i^{*}}={Y^*}^{[(i-1)M_0+1,iM_0]}$. Now let us divide the $\{0,1\}$ sequences of length $M_0$ in the following 3 classes.
\begin{enumerate}
\item Class $\mathbf{\star}$. Let $Z=(Z_1,Z_2,\cdots ,Z_{M_0})$ be a sequence of $0$'s and $1$'s. A length 2-subsequence $(Z_i,Z_{i+1})$ is called a ``flip" if $Z_i\neq Z_{i+1}$. We say $Z\in \mathbf{*}$ if the number of flips in $Z$ is at least $2R_0^{+}$.

\item Class $\mathbf{0}$. If $Z=(Z_1,Z_2,\cdots ,Z_{M_0})\notin \mathbf{\star}$ and
$Z$ contains more $0$'s than $1$'s, then $Z\in \mathbf{0}$.

\item Class $\mathbf{1}$. If $Z=(Z_1,Z_2,\cdots ,Z_{M_0})\notin \mathbf{\star}$
and $Z$ contains more $1$'s than $0$'s, then $Z\in \mathbf{1}$. For definiteness, let us also define $Z\in \mathbf{1}$, if $Z$ contains equal number of $0$'s and $1$'s and $Z\notin \mathbf{\star}$.
\end{enumerate}

Now set $\mathbb{X}=(X_1,X_2,\ldots)=X^{*}$ and construct $\mathbb{Y}=(Y_1,Y_2,\ldots)$ from $\tilde{Y^{*}}$ as follows. Set $Y_i=\mathbf{0}, \mathbf{1}$ or $\mathbf{\star}$ according as whether $\tilde{Y_i^*}\in \mathbf{0}, \mathbf{1}$ or $\mathbf{\star}$.

It is clear from this definition that $\mathbb{X}=(X_1,X_2,\ldots)$ and $\mathbb{Y}=(Y_1,Y_2,\ldots)$ are two independent sequences of i.i.d. symbols coming from the alphabets $\mathcal{C}^{\mathbb{X}}$ and $\mathcal{C}^{\mathbb{Y}}$ having distributions $\mu^{\mathbb{X}}$ and $\mu^{\mathbb{Y}}$ respectively where
$$\mathcal{C}^{\mathbb{X}}=\{0,1\}, \mathcal{C}^{\mathbb{Y}}=\{\mathbf{0}, \mathbf{1}, \mathbf{\star}\}.$$
We take $\mu^{\mathbb{X}}$ to be the uniform measure on $\{0,1\}$ and $\mu^{\mathbb{Y}}$
to be the natural measure on $\{\mathbf{0},\mathbf{1},\mathbf{\star}\}$ induced by the independent $Ber(\frac{1}{2})$ variables.

We take the relation $\mathcal{R}\subseteq \mathcal{C}^{\mathbb{X}}\times \mathcal{C}^{\mathbb{Y}}$ to be: \{$0\hookrightarrow \mathbf{0}, 0\hookrightarrow \mathbf{\star}$, $1\hookrightarrow \mathbf{1}, 1\hookrightarrow \mathbf{\star}$\} and the good sets $G_0^{\mathbb{X}}=\{0,1\}$ and $G_0^{\mathbb{Y}}=\{\mathbb{\star}\}$.

It is now very easy to verify that $\mathcal{C}^{\mathbb{X}},\mathcal{C}^{\mathbb{Y}},\mu^{\mathbb{X}},\mu^{\mathbb{Y}}, \mathcal{R},G_0^{\mathbb{X}}, G_0^{\mathbb{Y}}$, as defined above satisfies all the conditions described in our abstract framework.

%Indeed, (\ref{mux}) is trivial since $\mathcal{C}^{\mathbb{X}},\mathcal{C}^{\mathbb{Y}}$ are both finite. That $C_i\in G_0^{\mathbb{X}}$ and $C'_k\in G_0^{\mathbb{Y}}$ implies $C_i\hookrightarrow C'_k$ follows from definitions.

\subsubsection{Constructing the Lipschitz Embedding}
Now we verify that the the sequences $\mathbb{X}$ and $\mathbb{Y}$ constructed from the binary sequences $X^{*}$ and $Y^{*}$ can be used to construct an embedding with positive probability. We start with the following lemma.

\begin{lemma}
\label{embeddinglemma}
Let $X^{*}=\{X_i^{*}\}_{i\geq 1}$ and $Y^{*}=\{Y_i^{*}\}_{i\geq 1}$ be two independent sequences of i.i.d. $Ber(\frac{1}{2})$ variables. Let $\mathbb{X}$ and $\mathbb{Y}$ be the sequences constructed from $X^{*}$ and $Y^{*}$ as above. There is a constant $M$, such that whenever $\mathbb{X}\hookrightarrow \mathbb{Y}$, there exists a strictly increasing map $\phi: \mathbb{N}\rightarrow \mathbb{N}$ such that for all $i,j\in \mathbb{N}$, $X_{i}=Y_{\phi(i)}$ and $|\phi(i)-\phi(j)|\leq M|i-j|$, $\phi(1)\leq M/2$.
\end{lemma}

Before proceeding with the proof, let us make the following notation. We say $X^{*}\hookrightarrow _{*M} Y^{*}$ if a map $\phi$ satisfying the conditions of the lemma exists. Let us also make the following definition for finite subsequences.
\begin{definition}
Let ${X^{*}}^{[i_1,i_2]}$ and ${Y^{*}}^{[i'_1,i'_2]}$ be two segments of $X^{*}$ and $Y^{*}$ respectively. We say that ${X^{*}}^{[i_1,i_2]} \hookrightarrow _{* M} {Y^{*}}^{[i'_1,i'_2]}$ if there exists a strictly increasing $\tilde{\phi}: \{i_1,i_1+1,\ldots ,i_2\}\rightarrow  \{i'_1,i'_1+1,\ldots ,i'_2\}$ such that
\begin{enumerate}
\item[(i)] $X_k=Y_{\tilde{\phi}(k)}$ and $k,l \in \{i_1,i_1+1,\ldots ,i_2\}$ implies $|\phi(k)-\phi(l)|\leq M|k-l|$.
\item[(ii)] $\tilde{\phi}(i_1)-i'_1\leq M/3$ and $i'_2-\tilde{\phi}(i_2) \leq M/3$.
\end{enumerate}
\end{definition}

The following observation is trivial.
\begin{observation}
Let $0=i_0<i_2<\ldots$ and $0=i'_0<i'_2<\ldots$ be two increasing sequences of integers. If ${X^{*}}^{[i_k+1,i_{k+1}]}\hookrightarrow _{*M} {Y^{*}}^{[i'_k+1,i'_{k+1}]}$ for each $k\geq 0$, then $X^{*}\hookrightarrow _{* M} Y^{*}$.
\end{observation}

\begin{proof}(of Lemma \ref{embeddinglemma})
Let $X^{*},Y^{*},\mathbb{X},\mathbb{Y}$ be as in the statement of the Lemma.
Let $\mathbb{X}\hookrightarrow \mathbb{Y}$. Let $0=i_0<i_1<i_2<\ldots$ and $0=i'_0<i'_1<i'_2< \ldots$ be the two sequences obtained from the definition of $\mathbb{X}\hookrightarrow \mathbb{Y}$. The previous observation and the construction of $\mathbb{X}$ and $\mathbb{Y}$ then implies that it suffices to prove that there exists $M$ such that for all $h\geq 0$,
$${X^*}^{[i_h+1,i_{h+1}]}\hookrightarrow _{*M} {Y^*}^{[i_hM_0+1,i_{h+1}M_0]}.$$

Notice that since $\{i_{h+1}-i_h\}$ and $\{i'_{h+1}-i'_h\}$ are bounded sequences, if we can find maps $\phi_h: \{i_{h}+1,\ldots ,i_{h+1}\}\rightarrow \{i'_{h}+1,\ldots ,i'_{h+1}\}$ such that $X^*_i =Y^*_{\phi_h(i)}$, then for sufficiently large $M$ and for all $h$ we shall have ${X^*}^{[i_h+1,i_{h+1}]}\hookrightarrow _{*M} {Y^*}^{[i_hM_0+1,i_{h+1}M_0]}$. We shall call such a $\phi_h$ an embedding.

There are three cases to consider.\\

\textbf{Case 1:} $i_{h+1}-i_h=i'_{h+1}-i'_h=1$. By hypothesis, this implies
$X_{i_h+1}\hookrightarrow Y_{i'_h+1}$.
If $X^*_{i_{h}+1}=0$ and ${Y^*}^{[i_hM_0+1,i_{h}M_0+M_0]}\in \{\mathbf{0}, \mathbf{\star}\}$, then ${Y^*}^{[i_hM_0+1,i_{h}M_0+M_0]}$ must contain at least one $0$ and hence an embedding exists. Similarly if
$X^*_{i_{h}+1}=1$ and ${Y^*}^{[i_hM_0+1,i_{h}M_0+M_0]}\in \{\mathbf{1}, \mathbf{\star}\}$ then also an embedding exists.

\textbf{Case 2:} $i_{h+1}-i_h=R_0, R_0^{-}\leq i'_{h+1}-i'_h \leq R_0^{+}$. In this case, $Y^{[i'_h+1,i'_{h+1}]}$ is a ``good" segment, i.e., ${Y^*}{[(i'_h+k)M_0+1, (i'_h+k+1)M_0]}\in \mathbf{\star}$, for $0\leq k \leq {i'_{h+1}-i'_h-1}$. By what we have already observed it now suffices to only consider the case $i'_{h+1}-i'_h=1$. Now by definition of $\mathbf{\star}$, there exist an alternating sequence of $2R_0$ $0$'s and $2R_0$ $1$'s in ${Y^*}^{[i_hM_0+1,(i_{h}+1)M_0]}$. It follows that there is an embedding in this case as well.\\

\textbf{Case 3:} $i'_{h+1}-i'_h=R_0, R_0^{-}\leq i_{h+1}-i_h \leq R_0^{+}$. In this case, $Y^{[i'_h+1,i'_{h+1}]}$ is a ``good" segment, i.e., ${Y^*}{[(i'_h+k)M_0+1, (i'_h+k+1)M_0]}\in \mathbf{\star}$, for $0\leq k \leq {i'_{h+1}-i'_h-1}$. Again it suffices to only consider the case $i_{h+1}-i_h= R_0^{+}$. Now by definition of $\mathbf{\star}$, there exist an alternating sequence of $2R_0^{+}$ $0$'s and $2R_0^{+}$ $1$'s in ${Y^*}^{[i'_hM_0+1,(i'_{h}+1)M_0]}$ and as before an embedding exists in this case as well. This completes the proof of the Lemma.
\end{proof}

\subsubsection{Verifying the Base Case:}
We now complete the proof of Theorem \ref{t:embed} by using Theorem \ref{metatheorem}.

\begin{proof}(of Theorem \ref{t:embed})
Let $\mathcal{C}^{\mathbb{X}}, \mu^{\mathbb{X}},\mathcal{C}^{\mathbb{Y}}, \mu^{\mathbb{Y}}$ be as described above. Let $X\sim \mu^{\mathbb{X}}$, $Y\sim \mu^{\mathbb{Y}}$. (Notice that $\mu^{\mathbb{Y}}$ implicitly depends on the choice of $M_0$). Notice that (\ref{mux}) holds trivially. Let $\beta, \delta, m, R, L_0$ be given by Theorem~\ref{metatheorem}. First we show that  there exists $M_0$ such that (\ref{tailx}), and (\ref{goodxstatement}) holds.

Let $Z=(Z_1,Z_2,\ldots, Z_{M_0})$ be a sequence of i.i.d. $Ber(\frac{1}{2})$ variables. Observe that $$\mathbb{P}(Z\in \mathbf{\star})\geq (1-2^{1-\lfloor\frac{M_0}{2R_0^{+}}\rfloor})^{2R_0^{+}}\rightarrow 1~\text{as}~M_0\rightarrow \infty.$$
Hence we can choose $M_0$ large enough such that
\begin{equation}
\label{embeddingbasecase2}
\mu^{\mathbb{Y}}(\mathbf{\star})\geq \min \{1-L_0^{-\delta}, 1-2^{-(m+1)}L_0^{-\beta}\},
\end{equation}
which implies (\ref{tailx}) and (\ref{goodxstatement}) each hold.

%$\bullet$ Notice that for each $X\in \mathcal{C}^{\mathbb{X}}$, $s_0(X)\geq \mu^{\mathbb{Y}}(\mathbf{\star})\geq 1-L_0^{-\delta}> 1-L_0^{-1}$. Hence (\ref{tailx}) is vacuously true.
%
%$\bullet$ Clearly, for each $Y\in \mathcal{C}^{\mathbb{Y}}$, $r_0(Y)\geq \frac{1}{2}$, and hence (\ref{taily}) holds trivially for $p<\frac{1}{2}$. Now for $\frac{1}{2}\leq p \leq 1-L_0^{-1}$, it is clear that
%$\mathbb{P}(r_0(Y)\leq p)\leq \mathbb{P}(Y\neq \mathbf{\star})\leq 2^{-(m+1)}L_0^{-\beta}\leq p^{m+1}L_0^{-\beta}$, using (\ref{embeddingbasecase2}). Hence (\ref{taily}) holds.
%
%$\bullet$ Since $G_0^{\mathbb{X}}=\mathcal{C}^{\mathbb{X}}$, (\ref{goodxstatement}) holds trivially.
%
%$\bullet$ That (\ref{goodystatement}) holds follows directly from (\ref{embeddingbasecase2}).
%%\end{proof}
%%
%%\begin{proof}(of Theorem \ref{embeddingtheorem})

Now let $X^{*}=\{X_i^{*}\}_{i\geq 1}$ and $Y^{*}=\{Y_i^{*}\}_{i\geq 1}$ be two independent sequences of i.i.d. $Ber(\frac{1}{2})$ variables. Choosing $M_0$ as above, construct $\mathbb{X}$, $\mathbb{Y}$ as described in the previous subsection.
Then by Theorem \ref{metatheorem}, we have that $\mathbb{P}(\mathbb{X}\hookrightarrow_R \mathbb{Y})>0$. Using Lemma \ref{embeddinglemma} it now follows that for $M$ sufficiently large, we have $\mathbb{P}(X^{*}\hookrightarrow _{*M} Y^{*})>0$.  This gives an embedding for sequences indexed by the natural numbers which can easily be extended to embeddings indexed by the full integers with positive probability.  To see that this has probability 1 we note that the event that there exists an embedding is shift invariant and i.i.d. sequences are ergodic with respect to shifts and hence it has probability 0 or 1 completing the proof.
\end{proof}

\subsection{Rough Isometry}
Proposition 2.1 and 2.2 of~\cite{Peled:10} showed that to show that there exists $(M,D,C)$ such that two Poisson processes on $\R$ are roughly isomorphic almost surely it is sufficient to show that two independent copies site percolation of $\Z$, viewed as subsets of $\R$, are roughly isomorphic for some $(M',D',C')$ with positive probability.  We will solve the site percolation problem and thus infer Theorem~\ref{t:RIPoisson}.

\subsubsection{Defining the sequences $\mathbb{X}$ and $\mathbb{Y}$ and the alphabets $\mathcal{C}^{\mathbb{X}}$ and $\mathcal{C}^{\mathbb{Y}}$}
Let $X^{*}=\{X_i^{*}\}_{i\geq 0}$ and $Y^{*}=\{Y_i^{*}\}_{i\geq 0}$ be two independent sequences of i.i.d. $Ber(\frac{1}{2})$ variables conditioned that $X_0^*=Y_0^*=1$. Now let us define two sequences $k_0<k_1<k_2<\ldots$ and $k'_0<k'_1<k'_2<\ldots$ as follows. Let $k_0=0$ and $k_{i+1}=\min_{r>k_i}X^{*}_r=1$. Similarly let $k'_0=0$ and $k'_{i+1}=\min_{r>k'_i}Y^{*}_r=1$. Let $\tilde{X_i^{*}}=X^{*[k_i,k_{i+1}-1]}$ and $\tilde{Y_i^{*}}=Y^{*[k'_i,k'_{i+1}-1]}$.
The elements of the sequences $\{\tilde{X_i^{*}}\}$ and  $\{\tilde{Y_i^{*}}\}$ are sequences consisting of a single $1$ followed by a number (possibly none) of $0$'s. We now divide such sequences into the following classes.

Let $Z=(Z_0,Z_1,\ldots Z_L), (L\geq 0)$ be a sequence of $0$'s and $1$'s with $Z_0=1$ and $Z_i=0$ for $0<i\leq L$.  We say that $Z\in C_0$ if $L=0$ and for $j\geq 1$, we say $Z\in C_j$ if $2^{j-1}\leq L< 2^j$.

Now construct $\mathbb{X}=(X_1,X_2,\ldots)=X^{*}$ and $\mathbb{Y}=(Y_1,Y_2,\ldots)$ from $\tilde{Y^{*}}$ as follows. Set $X_i=C_j$ if $\tilde{X_i^{*}}\in C_j$.  Similarly set $Y_i=C_j$ if $\tilde{Y_i^{*}}\in C_j$.

It is clear from this definition that $\mathbb{X}=(X_1,X_2,\ldots)$ and $\mathbb{Y}=(Y_1,Y_2,\ldots)$ are two independent sequences of i.i.d. symbols coming from the alphabets $\mathcal{C}^{\mathbb{X}}$ and $\mathcal{C}^{\mathbb{Y}}$ having distributions $\mu^{\mathbb{X}}$ and $\mu^{\mathbb{Y}}$ respectively where
$$\mathcal{C}^{\mathbb{X}}=\mathcal{C}^{\mathbb{Y}}=\{C_0,C_1,C_2,\ldots\}.$$
and $\mu^{\mathbb{X}}=\mu^{\mathbb{Y}}$ is given by
$$\mu^{\mathbb{X}}(\{C_j\})=\mu^{\mathbb{Y}}(\{C_j\})=\mathbb{P}(Z\in C_j)$$
where $Z={Z^{*}}^{[0,i-1]}$, $Z_0^{*}=0$, $Z_t^*$ are of i.i.d. $Ber(\frac{1}{2})$ variables for $t\geq 1$ and $i=\min\{k>0: Z^*_{k}=1\}$.

We take the relation $\mathcal{R}\subseteq \mathcal{C}^{\mathbb{X}}\times \mathcal{C}^{\mathbb{X}}$ to be: $C_k\hookrightarrow C_{k'}$ if $|k-k'|\leq M_0$.  The ``good'' sets are defined to be $G_0^{\mathbb{X}}=G_0^{\mathbb{Y}}=\{C_j:j\leq M_0\}$.
It is now very easy to verify that $\mathcal{C}^{\mathbb{X}},\mathcal{C}^{\mathbb{Y}},\mu^{\mathbb{X}},\mu^{\mathbb{Y}}, \mathcal{R},G_0^{\mathbb{X}}, G_0^{\mathbb{Y}}$, as defined above satisfy all the conditions described in our abstract framework.

%That $C_i\in G_0^{\mathbb{X}}$ and $C'_k\in G_0^{\mathbb{Y}}$ implies $C_i\hookrightarrow C'_k$ follows from definitions of $\mathcal{R}, G_0^{\mathbb{X}}, G_0^{\mathbb{Y}}$. So this fits the description of $\mathbb{X}$ and $\mathbb{Y}$ as above in our abstract framework.

\subsubsection{Existence of the Rough Isometry}
\begin{lemma}
\label{roughisometrylemma}
Let $X^*=\{X_i^{*}\}_{i\geq 0}$ and $Y^*=\{Y_i^{*}\}_{i\geq 0}$ be two independent sequences of i.i.d. $Ber(\frac{1}{2})$ variables conditioned so that $X_0=Y_0=1$. Let $N_{X^{*}}=\{i:X_i^{*}=1\}$ and $N_{Y^{*}}=\{i:Y_i^{*}=1\}$.  Let $\mathbb{X}$ and $\mathbb{Y}$ be the sequences constructed from $X^{*}$ and $Y^{*}$ as above. Then there exist constants $(M,D,C)$, such that whenever $\mathbb{X}\hookrightarrow_R \mathbb{Y}$, there exists $\phi: N_{X^*}\rightarrow N_{Y^{*}}$ such that
\begin{enumerate}
\item[(i)] For all $t,s\in N_{X^{*}}$, $$\frac{1}{M}|t-s|-D\leq |\phi(t)-\phi(s)|\leq M|t-s|+D.$$
\item[(ii)] For all $t\in N_{Y^*}$, $\exists s\in N_{X^*}$ such that $|t-\phi(s)|\leq C$.
\end{enumerate}
\end{lemma}

\begin{proof}(of Lemma \ref{roughisometrylemma})
Suppose that $\mathbb{X}\hookrightarrow \mathbb{Y}$ and let $0=i_0<i_1<i_2<\ldots$ and $0=i'_0<i'_1<i'_2< \ldots$ be the two sequences satisfying the conditions of Definition \ref{mapstodefinition}. Let $0=k_0<k_1<k_2<\ldots$ and $0=k'_0<k'_1<k'_2<\ldots$ be the sequences described in the previous subsection while defining $\mathbb{X}$ and $\mathbb{Y}$.  For $r\geq 1$, define $X_r^{**}={X^*}^{\left[k_{i_{r-1}}, k_{i_r}-1\right]}$ and  $Y_r^{**}={Y^*}^{\left[k'_{i'_{r-1}}, k'_{i'_r}-1\right]}$, i.e., $X_r^{**}$ is the segment of $X^{*}$ corresponding to $\mathbb{X}^{[i_{r-1}+1,i_r]}$ and $Y_r^{**}$ is the segment of $Y^{*}$ corresponding to $\mathbb{Y}^{[i'_{r-1}+1,i'_r]}$  Define $N_{X,r}=N_{X^{*}}\cap [k_{i_{r-1}}, k_{i_r}-1]$ and $N_{Y,r}=N_{Y^{*}}\cap [k'_{i'_{r-1}}, k'_{i'_r}-1]$. Notice that by construction, for each $r$, $X^{*}_{k_{i_{r-1}}}=1$ and $Y^{*}_{k'_{i'_{r-1}}}=1$, i.e., $k_{i_{r-1}}\in N_{X,r}\subseteq N_{X^{*}}$.

Now let us define $\phi : N_{X^{*}} \rightarrow   N_{Y^{*}}$ as follows.  If $s\in N_{X,r}$,  define $\phi(s)=k'_{i'_{r-1}}$.
We show now that for $M=2^{M_0+2}R_0^{+}, C=2^{M_0+1}R_0^{+}$ and $D=2^{M_0+1}R_0^{+}$, the map defined as above satisfies the conditions in the statement of the lemma.

\textit{Proof of (i)}. First consider the case where $s,t\in N_{X,r}$ for some $r$. If $s\neq t$ then clearly $\mathbb{X}^{[i_{r-1}+1,i_r]}$ is a good segment and hence $|s-t|\leq 2^{M_0}R_0^{+}$. Clearly $|\phi(s)-\phi(t)|=0$. It follows that for the specified choice of $M$ and $D$,
$$\frac{1}{M}|t-s|-D\leq |\phi(t)-\phi(s)|\leq M|t-s|+D.$$

Let us now consider the case $s\in N_{X,r_1}$, $t\in N_{X,r_2}$. Clearly then $k_{i_{r_1-1}}\leq s < k_{i_{r_1}}$ and $k_{i_{r_2-1}}\leq t < k_{i_{r_2}}$. Also notice that by choice of $D$, for any good segment $\mathbb{X}^{[i_h+1,i_{h+1}]}$ we must have $|k_{i_{h+1}}-k_{i_h}|\leq 2^{M_0}R_0^{+} \leq \frac{D}{2}$. Further if a $i_{h+1}=i_h+1$, we must have that $|N_{X,h+1}|=1$. It follows that $s\leq k_{i_{r_1-1}}+D$ and
$t\leq k_{i_{r_2-1}}+D$. It is clear from the definitions that $\phi(s)=k'_{i'_{r_1-1}}$ and $\phi(t)=k'_{i'_{r_2-1}}.$

Then we have,
$$|\phi(t)-\phi(s)|=\sum_{h=r_1}^{r_2-1}|k'_{i'_h}-k'_{i'_{h-1}}|$$
and
$$\sum_{h=r_1}^{r_2-1}|k_{i_h}-k_{i_{h-1}}|-D \leq |t-s|\leq \sum_{h=r_1}^{r_2-1}|k_{i_h}-k_{i_{h-1}}|+D.$$
It now follows from the definitions that for each $h$,
$$\frac{1}{M}|k_{i_h}-k_{i_{h-1}}| \leq |k'_{i'_h}-k'_{i'_{h-1}}|\leq M|k_{i_h}-k_{i_{h-1}}|.$$
Adding this over $h=r_1,\ldots ,r_2-1$, we get that
$$\frac{1}{M}|t-s|-D\leq |\phi(t)-\phi(s)|\leq M|t-s|+D$$
in this case as well, which completes the proof of $(i)$.

\textit{Proof of (ii)}. Let $t\in N_{Y_*}$ and let $r$ be such that $k'_{i'_r}\leq t < k'_{i'_{r+1}}$. Now if $i'_{r+1}-i'_r=1$ we must have $t=k'_{i'_r}$ and hence $t=\phi(s)$ where $s=k_{i_r}\in N_{X}$ and hence (ii) holds for $t$. If $i'_{r+1}-i'_r\neq 1$ we must have that  $\mathbb{Y}^{[i_{r}+1,i_r]}$ is a good segment and hence $k'_{i'_{r+1}}-k'_{i'_r}\leq 2^{M_0}R_0^{+}$. Setting $s=k_{i_r}\in N_{X}$ we see that $\phi(s)=k'_{i'_r}$ and hence $|t-\phi(s)|\leq 2^{M_0}R_0^{+}\leq C$, completing the proof of (ii).
\end{proof}

\subsubsection{Verifying the Base case:}

\begin{proof}(of Theorem \ref{t:RIPoisson})
Let $\mathcal{C}^{\mathbb{X}}, \mu^{\mathbb{X}},\mathcal{C}^{\mathbb{Y}}, \mu^{\mathbb{Y}}$ be as described above. Let $X\sim \mu^{\mathbb{X}}$, $Y\sim \mu^{\mathbb{Y}}$. (Notice that $\mu^{\mathbb{X}}$ and $\mu^{\mathbb{Y}}$ both implicitly depend on the choice of $M_0$).
Notice first that
$$\mu^{\mathbb{X}}(C_0)=\mu^{\mathbb{Y}}(C_0)=\frac{1}{2}~\text{and}~\mu^{\mathbb{X}}(C_j)=\mu^{\mathbb{Y}}(C_j)=(\frac{1}{2})^{2^{j-1}}-(\frac{1}{2})^{2^{j}}$$
for $j\geq 1$, hence (\ref{mux}) is satisfied for $L_0$ sufficiently large. Let $\beta, \delta, m, R, L_0$ be given by Theorem \ref{metatheorem}. We first show that there exists $M_0$ such that  (\ref{tailx}) and (\ref{goodxstatement}) hold.

First observe that everything is symmetric in $\mathbb{X}$ and $\mathbb{Y}$. Clearly, we can take $M_0$ sufficiently large so that $S_0^{\mathbb{X}}(X)\geq 1-L_0^{-\delta}$ for $X=C_k$ for all $k\leq M_0$.

Now suppose $X=C_{k}$, where $k>M_0$.
\begin{eqnarray*}
S_0^{\mathbb{X}}(X)&=&\sum_{k':|k'-k|\leq M_0} \mu^{\mathbb{Y}}(C_k')\\
%&=& \sum_{k'=k-M_0}^{k+M_0}\mu^{\mathbb{Y}}(C_{k'})\\
&=& \sum_{k'=k-M_0}^{k+M_0} (\frac{1}{2})^{2^{k'-1}}-(\frac{1}{2})^{2^{k'}}= (\frac{1}{2})^{2^{k-M_0-1}}-(\frac{1}{2})^{2^{k+M_0}}.
\end{eqnarray*}

Observe that if $k>M_0$ and $X=C_k$, then $S_0^{\mathbb{X}}(X)<\frac{1}{2}$, and hence it suffices to prove (\ref{tailx}) for $p<\frac{1}{2}$.
Let us fix $p$, where $p<\frac{1}{2}$. Then we have

\begin{eqnarray*}
\mathbb{P}(S_0^{\mathbb{X}}(X)\leq p) &\leq & \sum_{k>M_0} \mu^{\mathbb{X}}(C_k)I((\frac{1}{2})^{2^{k-M_0-1}}-(\frac{1}{2})^{2^{k+M_0}}\leq p)\\
&\leq & \sum_{k>M_0} \mu^{\mathbb{X}}(C_k)I((\frac{1}{2})^{2^{k-M_0}}\leq p)\\
&\leq &  \sum_{k\geq M_0+\log _2(-\log _2 p)} (\frac{1}{2})^{2^{k'-1}}-(\frac{1}{2})^{2^{k'}}\\
&\leq &  (\frac{1}{2})^{2^{M_0+\log _2(-\log _2 p)-1}}= (\frac{1}{2})^{2^{M_0-1}(-\log _2 p)}= p^{2^{M_0-1}} \leq p^{m+1}L_0^{-\beta}
\end{eqnarray*}
for $M_0$ sufficiently large.
Also, since $\sum_{k} \mu^{\mathbb{X}}(C_k)=1$, by choosing $M_0$ sufficiently large we can make $\sum_{k\leq M_0} \mu^{\mathbb{X}}(C_k)\geq 1-L_0^{-\delta}$.

Hence there exist some constant $M_0$ for which both (\ref{tailx}) and (\ref{goodxstatement}) hold. This together with Lemma \ref{roughisometrylemma} and Theorem \ref{metatheorem} implies a rough isometry with positive probability for site percolation on $\N$ and hence for $\Z$.  The comments at the beginning of the subsection the show that the conditional results of~\cite{Peled:10} extend this result to Poisson processes on $\R$ proving Theorem \ref{t:RIPoisson}.
\end{proof}

\subsection{Compatible Sequences}
\subsubsection{Defining the alphabets $\mathcal{C}^{\mathbb{X}}$ and $\mathcal{C}^{\mathbb{Y}}$}
Let $\mathbb{X}=\{X_i\}_{i\geq 1}$ and $\mathbb{Y}=\{Y_i\}_{i\geq 1}$ be two independent sequences of i.i.d. $Ber(q)$ variables. Let us take $\mathcal{C}^{\mathbb{X}}=\mathcal{C}^{\mathbb{Y}}=\{0,1\}$. The measures $\mu^{\mathbb{X}}$ and $\mu^{\mathbb{Y}}$ are induced by the distribution of $X_i$'s and $Y_i$'s, i.e., $\mu^{\mathbb{X}}(\{1\})=\mu^{\mathbb{Y}}(\{1\})=q$ and $\mu^{\mathbb{X}}(\{0\})=\mu^{\mathbb{Y}}(\{0\})=1-q$. It is then clear that $\mathbb{X}$ and $\mathbb{Y}$ are two independent sequences of i.i.d. symbols coming from the alphabets $\mathcal{C}^{\mathbb{X}}$ and $\mathcal{C}^{\mathbb{Y}}$ having distributions $\mu^{\mathbb{X}}$ and $\mu^{\mathbb{Y}}$ respectively.

We define the relation $\mathcal{R}\subseteq \mathcal{C}^{\mathbb{X}}\times \mathcal{C}^{\mathbb{Y}}$ by
$$\{0\hookrightarrow 0, 0\hookrightarrow 1, 1\hookrightarrow 0\}.$$
Finally the ``good" symbols are defined by $G_0^{\mathbb{X}}=G_0^{\mathbb{Y}}=\{0\}$. It is clear that all the conditions in the definition of our set-up is satisfied by this structure.
%(\ref{tailx}) and (\ref{taily}) are trivial since $\mathcal{C}^{\mathbb{X}}$ and $\mathcal{C}^{\mathbb{Y}}$ are finite. That good symbols map to good symblos follows from the definition.

\subsubsection{Existence of the compatible map}

\begin{lemma}
\label{compatiblelemma}
Let $\mathbb{X}=\{X_i\}_{i\geq 1}$ and $\mathbb{Y}=\{Y_i\}_{i\geq 1}$ be two independent sequences of i.i.d. $Ber(q)$ variables. Suppose $\mathbb{X}\hookrightarrow \mathbb{Y}$. Then there exist $D,D'\subseteq \mathbb{N}$ such that,
\begin{enumerate}
\item[(i)] For all $i\in D$, $X_i=0$, for all $i'\in D'$, $Y_{i'}=0$.
\item[(ii)] Let $\mathbb{N}-D={k_1<k_2<\ldots}$ and $\mathbb{N}-D'={k'_1<k'_2<\ldots}$. Then $X_{k_i}\neq Y_{k'_i}$.
\end{enumerate}
\end{lemma}

\begin{proof}
The sets $D$ and $D'$ denote the set of sites we will delete.
Let $0=i_0<i_1<i_2<\ldots$ and $0=i'_0<i'_1<i'_2< \ldots$ be the sequences satisfying the properties listed in Definition \ref{mapstodefinition}.
Let $H_1^{*}=\{h:i_{h+1}-i_h=R_0\}$, $H_2^{*}=\{h:i'_{h+1}-i'_{h}=R_0\}$, $H_3^{*}=\{h: i_{h+1}-i_h=i'_{h+1}-i'_h=1, X_{i_h+1}=Y_{i'_h+1}=0\}$. Let $H^{*}=\cup_{i=1}^{3}H_i^{*}$. Now define
\begin{equation}
D=\bigcup_{h\in H^{*}}\{i_h+1,i_{h+1}\},~D'=\bigcup_{h\in H^{*}}\{i'_h+1,i'_{h+1}\}.
\end{equation}
It is clear from Definition \ref{mapstodefinition} that $D,D'$ defined as above satisfies the conditions in the statement of the lemma.
\end{proof}

%Proof of Theorem \ref{roughisometrytheorem} now follows from Lemma \ref{roughisometrylemma} and Lemma \ref{roughisometrylemma2}.

\subsubsection{Verifying the Base Case}

\begin{lemma}
\label{compatiblelemma2}
Let $\mathcal{C}^{\mathbb{X}}, \mu^{\mathbb{X}},\mathcal{C}^{\mathbb{Y}}, \mu^{\mathbb{Y}}$ be as described above. Let $X\sim \mu^{\mathbb{X}}$, $Y\sim \mu^{\mathbb{Y}}$. (Notice that $\mu^{\mathbb{X}}$, $\mu^{\mathbb{Y}}$ implicitly depend on $q$) Let $\beta, \delta, m, R, L_0$ be given by Theorem \ref{metatheorem}. Then there exists $q_0$ such that  for all $q\leq q_0$,  (\ref{tailx}) and (\ref{goodxstatement})  hold.
\end{lemma}

\begin{proof}
Take $q_0=L_0^{-\delta}$. Let $q\leq q_0$. Clearly, then, for any $X\in\mathcal{C}^{\mathbb{X}}$ (resp. for any $Y\in\mathcal{C}^{\mathbb{Y}}$) we have $S_0^{\mathbb{X}}(X)\geq 1-q\geq 1-L_0^{-1}$ (resp. $S_0^{\mathbb{Y}}(Y)\geq 1-L_0^{-1}$). Hence, (\ref{tailx}) is trivially satisfied. That (\ref{goodxstatement})  holds follows directly from the definitions.
\end{proof}

Theorem \ref{t:compatible} now follows from Lemma \ref{compatiblelemma} and Lemma \ref{compatiblelemma2}.

\section{Preliminaries}\label{s:prelim}

Let $\mathbb{X}$, $\mathbb{Y}$, $\mathcal{C}^{\mathbb{X}}, \mathcal{C}^{\mathbb{Y}}, G_0^{\mathbb{X}}, G_0^{\mathbb{Y}}$ be as described in the previous section. As we have described in \S~\ref{s:intro} before, our strategy of proof of the Theorem \ref{metatheorem} is to partition the sequences $\mathbb{X}$ and $\mathbb{Y}$ into blocks at each level $j\geq 1$. Because of the symmetry between $\mathbb{X}$ and $\mathbb{Y}$ we only describe the procedure to form the blocks for the sequence $\mathbb{X}$. For each $j\geq 1$, we write $\mathbb{X}=(X_1^{(j)},X_2^{(j)},\ldots)$ where we call each $X_i^{(j)}$ a level $j$ $\mathbb{X}$-block.  Most of the time we would clearly state that something is a level $j$ block and drop the superscript $j$. Each of the level $j$ $\mathbb{X}$-block is a concatenation of a number of level $(j-1)$ $\mathbb{X}$-blocks, where level $0$ blocks are just the characters of the sequence $\mathbb{X}$. At each level, we also have a recursive definition of ``$good$" blocks. Let $G_j^{\mathbb{X}}$ and $G_j^{\mathbb{Y}}$ denote the set of good $\mathbb{X}$-blocks and  good $\mathbb{Y}$-blocks at $j$-th level respectively. Now we are ready to describe the recursive construction of the blocks $X_i^{(j)}$. for $j\geq 1$.

\subsection{Recursive Construction of Blocks}
We only describe the construction for $\mathbb{X}$. Let us suppose we have already constructed the blocks of partition  upto level $j$ for some $j\geq 0$ and we have $X=(X_1^{(j)},X_2^{(j)},\ldots)$. Also assume we have defined the "good" blocks at level $j$, i.e., we know $G_j^{\mathbb{X}}$. We can start off the recursion since both these assumptions hold for $j=0$.
We describe how to partition $\mathbb{X}$ into level $(j+1)$ blocks: $\mathbb{X}=(X_1^{(j+1)},X_2^{(j+2)},\ldots)$.

Suppose the first $k$ blocks $X_1^{(j+1)},\ldots , X_k^{(j+1)}$ at level $(j+1)$ has already been constructed and suppose that the rightmost level $j$-subblock of $X_k^{(j+1)}$ is $X_m^{(j)}$. Then $X_{k+1}^{(j+1)}$ consists of the sub-blocks $X_{m+1}^{(j)},X_{m+2}^{(j)},\ldots ,X_{m+l+L_j^3}^{(j)}$ where $l>L_j^3+L_j^{\alpha-1}$ is selected in the following manner. Let $W_{k+1,j+1}$ be a geometric random variable having $\mbox{Geom}(L_j^{-4})$ distribution and independent of everything else. Then $$l=\min\{s\geq L_j^3+L_j^{\alpha -1}+ W_{k+1,j+1}: X_{m+L_j^3+L_j^{\alpha-1}+W_{k+1,j+1}+s+i}\in G_{j}^{\mathbb{X}} ~\text{for}~1\leq i \leq 2L_j^3 \}.$$
That such an $l$ is finite with probability 1 will follow from our recursive estimates.\\

Put simply, our block construction mechanism at level $(j+1)$ is as follows:\\
\emph{Starting from the right boundary of the previous block, we include $L_j^3$ many sub-blocks, then further $L_j^{\alpha-1}$ many sub-blocks, then a geometric $(1/L_j^4)$ many sub-blocks. Then we wait for the first occurrence of a run of $2L_j^3$ many consecutive good sub-blocks, and end our block at the midpoint of this run.}

This somewhat complex choice of block structure is made for several reasons.  It guarantees stretches of good sub-blocks at both ends of the block thus ensuring these are not problematic when trying to embed one block into another.  The fact that good blocks can be mapped into shorter or longer stretches of good blocks then allows us to line up sub-blocks in a potential embedding in many possible ways which is crucial for the induction.  Our blocks are not of fixed length.  It is potentially problematic to our approach if conditional on a block being long that it contains many bad blocks.  Thus we added the geometric term to the length.  This has the effect that given that the block is long, it is most likely because the geometric random variable is large, not because of the presence of many bad blocks.  Finally, the construction means that block will be independent.

We now record two simple but useful properties of the blocks thus constructed in the following observation. Once again a similar statement holds for $\mathbb{Y}$.

\begin{observation}\label{o:blockStructure}
Let $\mathbb{X}=(X_1^{(j+1)},X_2^{(j+1)},\ldots)=(X_1^{(j)}, X_2^{(j)}, \ldots)$ denote the partition of $\mathbb{X}$ into blocks at levels $(j+1)$ and $j$ respectively. Then the following hold.

\begin{enumerate}
\item Let $X_i^{(j+1)}=(X_{i_1}^{(j)},X_{i_1+1}^{(j)},\ldots X_{i_1+l}^{(j)})$. For $i\geq 1$, $X_{i_1+l+1-k}^{(j)}\in G_j^{\mathbb{X}}$ for each $k$, $1\leq k \leq L_j^3$. Further, if $i>1$, then $X_{i_1+k-1}^{(j)}\in G_j^{\mathbb{X}}$ for each $k$, $1\leq k \leq L_j^3$. That is, all blocks at level $(j+1)$, except possibly the leftmost one ($X_1^{(j+1)}$), are guaranteed to have at least $L_j^3$ ``good" level $j$ sub-blocks at either end. Even $X_1^{(j+1)}$ ends in $L_j^3$ many good sub-blocks.
\item $X_1^{(j+1)},X_2^{(j+1)}, \ldots $ are independently distributed. In fact, $X_2^{(j+1)}, X_3^{(j+1)}, \ldots$ are independently and identically distributed according to some law, say $\mu_{j+1}^{\mathbb{X}}$. Furthermore, conditional on the event $\{X_i^{(j)}\in G_j^{\mathbb{X}} ~\text{for}~ i=1,2,\ldots, L_j^3 \}$, the $(j+1)$-th level blocks $X_1^{(j+1)}, X_2^{(j+1)},\ldots$ are independently and identically distributed according to the law $\mu_{j+1}^{\mathbb{X}}$. Let us denote the corresponding law for the $(j+1)$-th level blocks of $\mathbb{Y}$ by $\mu_{j+1}^{\mathbb{Y}}$.
\end{enumerate}
\end{observation}

From now on whenever we say ``a (random) $\mathbb{X}$-block at level $j$", we would imply that it has law $\mu_{j+1}^{\mathbb{X}}$, unless explicitly stated otherwise. Similarly ``a (random) $\mathbb{Y}$-block at level $j$" would imply that it has distribution $\mu_{j+1}^{\mathbb{X}}$. Also for convenience, we assume $\mu_{0}^{\mathbb{X}}=\mu^{\mathbb{X}}$ and $\mu_0^{\mathbb{Y}}=\mu^{\mathbb{Y}}$.

Also, for $j\geq 0$, let $\mu_{j,G}^{\mathbb{X}}$ denote the conditional law of an $\mathbb{X}$ block at level $j$, given that it is in $G_j^{\mathbb{X}}$. We define $\mu_{j,G}^{\mathbb{Y}}$ similarly.

We observe that we can construct a block with law $\mu_{j+1}^{\mathbb{X}}$ (resp. $\mu_{j+1}^{\mathbb{Y}}$) in the following alternative manner without referring to the the sequence $\mathbb{X}$ (resp. $\mathbb{Y}$).

\begin{observation}\label{o:blockRepresentation}
Let $X_1,X_2,X_3,\ldots$ be a sequence of independent level $j$ $\mathbb{X}$-blocks such that $X_i\sim \mu_{j,G}^{\mathbb{X}}$ for $1\leq i \leq L_j^3$ and $X_i\sim \mu_{j}^{\mathbb{X}}$ for $i> L_j^3$. Now let $W$ be a $Geom(L_j^{-4})$ variable independent of everything else. Define as before
$$l=\min\{i\geq L_j^3+L_j^{\alpha -1}+W: X_{i+k}\in G_j^{\mathbb{X}} ~\text{for}~1\leq k\leq 2L_j^3\}.$$
Then $X=(X_1,X_2,\ldots ,X_{l+L_j^3})$ has law $\mu_{j+1}^{\mathbb{X}}$.
\end{observation}

Whenever we have a sequence $X_1,X_2,...$ satisfying the condition in the observation above, we shall call $X$  the (random) level $(j+1)$ block
constructed from $X_1,X_2,....$ and we shall denote the corresponding geometric variable to be $W_X$ and $T_X=l-L_j^3-L_j^{\alpha -1}$.

\subsection{Definitions}
In this subsection we make some definitions that we are going to use throughout our proof.

\begin{definition}
For $j\geq 0$, let $X$ be a block of $\mathbb{X}$ at level $j$ and let $Y$ be a block of $\mathbb{Y}$ at level $j$. We define the embedding probability of $X$ to be  $S_j^{\mathbb{X}}(X)=\mathbb{P}(X\hookrightarrow Y|X)$. Similarly we define $S_j^{\mathbb{Y}}(Y)=\mathbb{P}(X\hookrightarrow Y|Y)$. As noted above the law of $Y$ is $\mu_j^{\mathbb{Y}}$ in the definition of $S_j^{\mathbb{X}}$ and the law of $X$ is $\mu_j^{\mathbb{X}}$ in the definition of $S_j^{\mathbb{Y}}$.
\end{definition}

Notice that  $j=0$ in the above definitions corresponds to the definition we had in Section~\ref{s:abstract}.

\begin{definition}
Let $X$ be an $\mathbb{X}$-block at level $j$. It is called ``semi-bad'' if $X\notin G_j^X$, $S_j^{\mathbb{X}}(X)\geq 1-\frac{1}{20k_0R_{j+1}^{+}}$,$|X|\leq 10L_j$ and $C_k\notin X$ for any $k>L_j^m$. Here $|X|$ denotes the number of $\mathcal{C}^{\mathbb{X}}$ characters  in $X$.
A ``semi-bad'' $\mathbb{Y}$ block at level $j$ is defined similarly.
\end{definition}
We denote the set of all semi-bad $\mathbb{X}$-blocks (resp. $\mathbb{Y}$-blocks) at level $j$ by $SB_j^{\mathbb{X}}$ (resp. $SB_j^{\mathbb{Y}}$).

\begin{definition}
Let $\tilde{Y}=(Y_1,\ldots,Y_{n})$ be a sequence of consecutive $\mathbb{Y}$ blocks at level $j$. $\tilde{Y}$ is said to be a ``strong sequence" if for every $X\in SB_j^{\mathbb{X}}$

$$\# \{1\leq i \leq n: X\hookrightarrow Y_i\} \geq n(1-\frac{1}{10k_0R_{j+1}^{+}}).$$
Similarly a ``strong" $\mathbb{X}$-sequence can also be defined.
\end{definition}

\subsection{Good blocks}
To complete the description, we need now give the definition of ``good'' blocks at level $j$ which we have alluded to above. With the definitions from the preceding section, we are now ready to give the recursive definition of a ``good" block as follows. Suppose we already have definitions of ``good" blocks upto level $j$ (i.e., characterized $G_k^{\mathbb{X}}$ for $k\leq j$). Good blocks at level $(j+1)$ are then defined in the following manner. As usual we only give the definition for $\mathbb{X}$-blocks, the definition for $\mathbb{Y}$ is exactly similar.\\

Let $X^{(j+1)}=(X_1^{(j)},X_2^{(j)},\ldots, X_n^{(j)})$ be a $\mathbb{X}$ block at level $(j+1)$. Notice that we can form blocks at level $(j+1)$ since we have assumed that we already know $G_j^{\mathbb{X}}$. Then we say $X^{(j+1)}\in G_{j+1}^{\mathbb{X}}$ if the following conditions hold.

\begin{enumerate}
\item[(i)] It contains at most $k_0$ bad sub-blocks. $\#\{1\leq i \leq n: X_i\notin G_{j}^{\mathbb{X}}\}\leq k_0$.

\item[(ii)] For each $1\leq i \leq n$ such that $X_i\notin G_{j}^{\mathbb{X}}$, $X_i\in SB_j^{\mathbb{X}}$.

\item[(iii)] Every  sequence of $L_j^{3/2}$ consecutive level $j$ sub-blocks is ``strong''.

\item[(iv)] The length of the block satisfies
$n\leq L_{j}^{\alpha-1}+L_j^5$.
\end{enumerate}

Finally we define ``segments" of a sequence of consecutive $\mathbb{X}$ or $\mathbb{Y}$ blocks at level $j$.

\begin{definition}
Let $\tilde{X}=(X_1,X_2,\ldots)$ be a sequence of consecutive $\mathbb{X}$-blocks. For $i_2 > i_1 \geq 1$, we call the subsequence $(X_{i_1},X_{i_1+1},\ldots , X_{i_2})$ the ``$[i_1,i_2]$-segment" of $\tilde{X}$ denoted by $\tilde{X}^{[i_1,i_2]}$. The  ``$[i_1,i_2]$-segment" of a sequence of $\mathbb{Y}$ blocks is also defined similarly. Also a segment is called a ``good" segment if it consists of all good blocks.
\end{definition}

For the case $j=0$ this reduces to the definition  given in \S~\ref{s:abstract}.

\section{Recursive estimates}\label{s:recursive}

Our proof of the general theorem depends on a collection of recursive estimates, all of which are proved together by induction. In this section we list these estimates for easy reference. The proof of these estimates are provided in the next section. We recall that for all $j>0$ $L_j=L_{j-1}^{\alpha}=L_0^{\alpha ^j}$ and for all $j\geq 0$, $R_j=4^j(2R)$, $R_j^{-}=4^j(2-2^{-j})$ and $R_j^{+}=4^jR^2(2+2^{-j})$. For $j=0$, this definition of $R_j$, $R_j^{+}$ and $R_j^{-}$ agrees with the definition given in \S~\ref{s:abstract}.

\subsection{Tail Estimate}

\begin{enumerate}
\item[I.]
Let $j\geq 0$. Let $X$ be a $\mathbb{X}$-block at level $j$ and  let $m_j=m+2^{-j}$. Then
\begin{equation}
\label{tailx1}
\mathbb{P}(S_j^{\mathbb{X}}(X)\leq p)\leq p^{m_j}L_j^{-\beta}~~\text{for}~~p\leq 1-L_{j}^{-1}.
\end{equation}
Let $Y$ be a $\mathbb{Y}$-block at level $j$. Then
\begin{equation}
\label{taily1}
\mathbb{P}(S_j^{\mathbb{Y}}(Y)\leq p)\leq p^{m_j}L_j^{-\beta}~~\text{for}~~p\leq 1-L_{j}^{-1}.
\end{equation}

\end{enumerate}

%\subsection{Strong Sequence}
%
%\begin{enumerate}
%\item[II.] Let $\tilde{Y}=(Y_1,\ldots Y_{L_j^{1+\epsilon}})$ be a sequence of $L_{j}^{1+\epsilon}$ sub-blocks at level $j$. Then
%\begin{equation}
%\label{strong}
%\mathbb{P}(\tilde{Y}~\text{is ``strong"})\geq 1-e^{-cL_{j}}.
%\end{equation}
%\end{enumerate}

\subsection{Length Estimate}
\begin{enumerate}
\item[II.]
For $X$ be an $\mathbb{X}$ block at at level $j\geq 0$,
\begin{equation}
\label{lengthx}
\mathbb{E}[\exp (L_{j-1}^{-6}(|X|-(2-2^{-j})L_j))] \leq 1.
\end{equation}

Similarly for a level $j$ $\mathbb{Y}$ block $Y$, we have
\begin{equation}
\label{lengthy}
\mathbb{E}[\exp (L_{j-1}^{-6}(|Y|-(2-2^{-j})L_j))] \leq 1.
\end{equation}
\end{enumerate}
For the case $j=0$ we interpret equations (\ref{lengthx}) and (\ref{lengthy}) by setting $L_{-1}=L_0^{\alpha^{-1}}$.

\subsection{Properties of Good Blocks}
\begin{enumerate}
\item[III.]
``Good" blocks map to good blocks, i.e.,
\begin{equation}
\label{goodvgood}
X\in  G_j^{\mathbb{X}}, Y\in G_j^{\mathbb{Y}} \Rightarrow X\hookrightarrow Y.
\end{equation}

\item[IV.]
Most blocks are ``good".
\begin{equation}
\label{xgood}
\mathbb{P}(X\in G_j^\mathbb{X})\geq 1-L_j^{-\delta}.
\end{equation}
\begin{equation}
\label{ygood}
\mathbb{P}(Y\in G_j^\mathbb{Y})\geq 1-L_j^{-\delta}.
\end{equation}

\item[V.]
Good blocks can be compressed or expanded.

Let $\tilde{X}=(X_1,X_2,\ldots)$ be a sequence of $\mathbb{X}$-blocks at level $j$ and $\tilde{Y}=(Y_1,Y_2,\ldots)$ be a sequence of $\mathbb{Y}$-blocks at level $j$. Further we suppose that $\tilde{X}^{[1,R_j^{+}]}$ and $\tilde{Y}^{[1,R_j^{+}]}$ are ``good segments". Then for every $t$ with $R_j^{-}\leq t \leq R_j^{+}$,

\begin{equation}
\label{compress}
\tilde{X}^{[1,R_j]}\hookrightarrow \tilde{Y}^{[1,t]} ~\text{and}~\tilde{X}^{[1,t]}\hookrightarrow \tilde{Y}^{[1,R_j]}.
\end{equation}
\end{enumerate}

\begin{theorem}[Recursive Theorem]
\label{induction}
There exist positive constants $\alpha$, $\beta$, $\delta$, $m$, $k_0$ and $R$ such that for all large enough $L_0$ the following holds.
If the recursive estimates (\ref{tailx1}), (\ref{taily1}), (\ref{lengthx}), (\ref{lengthy}), (\ref{goodvgood}), (\ref{xgood}), (\ref{ygood})  and (\ref{compress}) hold at level $j$ for some $j\geq 0$ then all the estimates also hold at level $j+1$ as well.
\end{theorem}

We will choose the parameters as in equation~\eqref{e:parameters}.  Before giving a proof of Theorem \ref{induction} we show how using this theorem we can prove the general theorem.

\begin{proof}(of Theorem~\ref{metatheorem})
Let $\mathbb{X}=(X_1,X_2,\ldots)$, $\mathbb{Y}=(Y_1,Y_2,\ldots)$ be as in the statement of the theorem. Let for $j\geq 0$, $\mathbb{X}=(X_1^{(j)}, X_2^{(j)},\ldots)$ denote the partition of $\mathbb{X}$ into level $j$ blocks as described above. Similarly let $\mathbb{Y}=(Y_1^{(j)}, Y_2^{(j)},\ldots)$ denote the partition of $\mathbb{Y}$ into level $j$ blocks. Let $\beta, \delta , m, R, L_0$ be as in Theorem~\ref{induction}. Recall that the characters are the blocks at level $0$, i.e., $X_i^{(0)}=X_i$ and $Y_i^{(0)}$ for all $i\geq 1$. Hence  the hypotheses of Theorem \ref{metatheorem} implies that (\ref{tailx1}), (\ref{taily1}), (\ref{xgood}), (\ref{ygood}) holds for $j=0$. It follows from definition that (\ref{goodvgood}) and (\ref{compress}) also hold at level $0$. That (\ref{lengthx}) and (\ref{lengthy}) holds is trivial. Hence the estimates $I-V$ hold at level $j$ for $j=0$. Using the recursive theorem, it now follows that (\ref{tailx1}), (\ref{taily1}), (\ref{lengthx}), (\ref{lengthy}), (\ref{goodvgood}), (\ref{xgood}), (\ref{ygood})  and (\ref{compress}) hold for each $j\geq 0$.

Let $\mathcal{T}_j^{\mathbb{X}}=\{X_k^{(j)}\in G_{j}^{\mathbb{X}}, 1\leq k \leq L_j^3\}$ be the event that the first $L_j^3$ blocks at level $j$ are good.  Notice that on the event $\cap_{k=0}^{j-1} \mathcal{T}_k^{\mathbb{X}}$,  $X_1^{(j)}$ has distribution $\mu_j^{\mathbb{X}}$ by Observation~\ref{o:blockStructure} and so $\{X_i^{(j)}\}_{i\geq 1}$ is i.i.d. with distribution $\mu_j^{\mathbb{X}}$. Hence it follows from equation (\ref{xgood}) that $\mathbb{P}(\mathcal{T}_j^{\mathbb{X}}|\cap_{k=0}^{j-1} \mathcal{T}_k^{\mathbb{X}})\geq (1-L_j^{-\delta})^{L_j^3}$. Similarly defining $\mathcal{T}_j^{\mathbb{Y}}=\{Y_k^{(j)}\in G_{j}^{\mathbb{Y}}, 1\leq k \leq L_j^3\}$ we get using equation (\ref{ygood}) that $\mathbb{P}(\mathcal{T}_j^{\mathbb{Y}}|\cap_{k=0}^{j-1} \mathcal{T}_k^{\mathbb{Y}})\geq (1-L_j^{-\delta})^{L_j^3}$.

Let $\mathcal{A}=\cap _{j\geq 0}(\mathcal{T}_j^{\mathbb{X}}\cap \mathcal{T}_j^{\mathbb{Y}})$. It follows from above that $\mathbb{P}(\mathcal{A})>0$ since $\delta>3$. Also, notice that, on $\mathcal{A}$,  $X_1^{(j)}\hookrightarrow Y_1^{(j)}$ for each $j\geq 0$. Since $|X_1^{(j)}|, |Y_1^{(j)}|\rightarrow \infty$ as $j\rightarrow \infty$, it follows that there exists a subsequence $j_n \rightarrow \infty$ such that there exist $R$-embeddings of $X_1^{(j_n)}$ into $Y_1^{(j_n)}$ with associated partitions $(i_0^n,i_1^n,\ldots,i_{\ell_n}^n)$ and $({i'}_0^n,{i'}_1^n,\ldots,{i'}_{\ell_n}^n)$ with $\ell_n\to\infty$ satisfying the
conditions of Definition~\ref{mapstodefinition} and such that for all $r\geq 0$ we have that $i_r^n \to i_r^\star$ and ${i'}_r^n\to {i'}_r^\star$ as $n\to\infty$.  These limiting partitions of $\mathbb{N}$, $(i_0^\star,i_1^\star,\ldots)$ and $({i'}_0^\star,{i'}_1^\star,\ldots)$,   satisfy the conditions of Definition~\ref{mapstodefinition} implying that $\mathbb{X}\hookrightarrow_R \mathbb{Y}$. It follows that $\mathbb{P}(\mathbb{X}\hookrightarrow \mathbb{Y})>0$.
\end{proof}

The remainder of the paper is devoted to the proof of the estimates in the induction. Throughout these sections we assume that the estimates $I-V$ hold for some level $j\geq 0$ and then prove the estimates at level $j+1$.  Combined they complete the proof of Theorem~\ref{induction}.

\section{Notation for maps: Generalised Mappings}\label{s:notation}

Since in our estimates we will need to map segments of sub-blocks to segments of sub-blocks we need a notation for constructing such mappings.
Let $A,A' \subseteq \mathbb{N}$, be two sets of consecutive integers. Let $A=\{n_1+1,\ldots ,n_1+n\}$, $A'=\{n_1'+1,\ldots , n_1'+n'\}$.
Let $$\mathcal{P}_A=\{P:P=\{n_1=i_0<i_1<\ldots <i_z=n_1+n\}\}$$ denote the set of partitions  of $A$. For $P= \{n_1=i_0<i_1<\ldots <i_z=n_1+n\}\in \mathcal{P}_{A}$, let us denote the ``length'' of $P$, by $l(P)=z$. Also let the set of all blocks of $P$, be denoted by
$\mathcal{B}(P)=\{[i_r+1,i_{r+1}]\cap \mathbb{Z}: 0\leq r\leq z-1\}$. We define $P'\in \mathcal{P}_{A'}$, $l(P')$ and $\mathcal{B}(P')$  similarly. \\

\subsection{Generalised Mappings:}
Now let $\Upsilon$ denote a ``generalised mapping" which assigns to the tuple $(A,A')$, a triplet $(P,P',\tau)$, where $P\in \mathcal{P}_{A}$, $P'\in \mathcal{P}_{A'}$, with $l(P)=l(P')$, and $\tau: \mathcal{B}(P) \mapsto \mathcal{B}(P')$ be the unique increasing bijection from the blocks of $P$ to the blocks of $P'$. Let $P=\{n_1=i_0<i_1<\ldots <i_{l(P)}=n_1+n\}$ and $P'=\{n_1'=i_0<i'_1<\ldots <i'_{l(P')}=n_1'+n'\}$. Then by $\tau$ is an ``increasing" bijection we mean that $l(P)=l(P')=z$ (say), and $\tau([i_r+1,i_{r+1}]\cap \mathbb{Z})=[i'_r+1,i'_{r+1}]\cap \mathbb{Z}$.

A generalised mapping $\Upsilon$ of $(A,A')$ (say, $\Upsilon(A,A')=(P,P',\tau)$) is called ``admissible" if the following holds.

\emph{Let $\{x\}\in \mathcal{B}(P')$ is a singleton. Then $\tau(\{x\})=\{y\}$ (say) is also a singleton. In this case we simply denote $\tau(x)=y$. Similarly, if $\{y\}\in \mathcal{B}(P')$ is a singleton, then $\tau^{-1}(\{y\})$ is also a singleton. Note that since we already requite $\tau$ to be a bijection, it makes sense to talk about $\tau^{-1}$ as a function here. Also if $\tau^{-1}(\{y\})={x}$, we write $\tau^{-1}(y)=x$.}

Let $B\subseteq  A$ and $B'\subseteq A'$ be two subsets of $A,A'$ respectively. An admissible generalized mapping  $\Upsilon^j$ of $(A,A,)$ is called of class $G^j$ with respect to $(B,B')$ (we denote this by saying $\Upsilon^j(A,A',B,B')$ is admissible of class $G^j$)  if it satisfies the following conditions:

\begin{enumerate}
\item[(i)] If $x\in B$, then the singleton $\{x\}\in \mathcal{B}(P)$.  Similarly if $y\in B'$, then $\{y\}\in \mathcal{B}(P')$.

\item[(ii)] If $i_{r+1}>i_r+1$ (equivalently, $i'_{r+1}>i'_r+1$), then $(i_{r+1}-i_{r})\wedge (i'_{r+1}-i'_{r})> L_j$ and $\frac{1-2^{-(j+5/4)}}{R}<\frac{i'_{r+1}-i'_{r}}{i_{r+1}-i_{r}}< R(1+2^{-(j+5/4)})$.

\item[(iv)] For all $x\in B$, $\tau(x)\notin B'$, and for all $y\in B'$, $\tau^{-1}(y)\notin B$.
\end{enumerate}

Similarly, an admissible generalised mappings $\Upsilon(A,A')=(P,P',\tau)$ is called of $Class~H_1^j$ with respect to $B$ if it satisfies the following conditions:

\begin{enumerate}
\item[(i)] If $x\in B$, then $\{x\}\in \mathcal{B}(P)$.
\item[(ii)] If $i_{r+1}>i_r+1$ (equivalently, $i'_{r+1}>i'_r+1$), then $(i_{r+1}-i_{r})\wedge (i'_{r+1}-i'_{r})> L_j$ and $\frac{1-2^{-(j+5/4)}}{R}<\frac{i'_{r+1}-i'_{r}}{i_{r+1}-i_{r}}< R(1+2^{-(j+5/4)})$.
\item[(iii)] For all $x\in B$, $n'-L_j^3> \tau(x)>L_j^3$.
\end{enumerate}

Finally, an admissible generalised mapping $\Upsilon^j(A,A')=(P,P',\tau)$ is called of $Class~H_2^j$ with respect to $B$ if it satisfies the following conditions:

\begin{enumerate}
\item[(i)] If $x\in B$, then $\{x\}\in \mathcal{B}(P)$.
\item[(ii)] $L_j^3<\tau(l_i)< n'-L_j^3$ for $1\leq i\leq K_X$.
\item[(iii)] If $[i_{h}+1, i_h]\cap \Z\in \mathcal{B}(P)$ and $i_{h}+1\neq i_{h+1}$ then $i_{h+1}-i_{h}=R_j$ and $R_j^{-}\leq i'_{h+1}-i'_{h}\leq R_{j}^{+}$.
\end{enumerate}
%
%$\Upsilon^j(A,A',B,B')=(P,P',\tau)$ is called an admissible generalized mapping of $Class~G_2$ if it is of class $G$ and
%for all $x\in B$, $n'-L_j^3> \tau(x)>L_j^3$.

\subsection{Generalised mapping induced by a pair of partitions:}

Let $A,A',B,B'$ be as above. By a ``marked partition pair" of $(A,A')$ we mean a triplet $(P_{*},P'_{*},Z)$ where  $P_{*}=\{n_1=i_0<i_1<\ldots <i_{l(P_{*})}=n_1+n\}\in \mathcal{P}(A)$ and $P'_{*}=\{n_1'=i_0<i'_1<\ldots <i'_{l(P'_{*})}=n_1'+n'\}\in \mathcal{P_{A'}}$, $l(P_{*})=l(P'_{*})$ and $Z\subset [l(P)-1]$ is such that $r\in Z\Rightarrow  i_{r+1}-i_r=i'_{r+1}-i'_r$.

It is easy to see that a ``marked partition pair" induces a Generalised mapping $\Upsilon$ of $(A,A',B,B')$ in the following natural way.

Let $P$ be the partition of $A$ whose blocks are given by $$\mathcal{B}(P)= \cup_{r\in Z} \{\{i\}:i\in [i_{r}+1,i_{r+1}]\cap \Z\}\cup_{r \notin Z}\{[i_{r}+1,i_{r+1}]\cap \mathbb{Z}\}.$$
Similarly let $P'$ be the partition of $A'$ whose blocks are given by  $$\mathcal{B}(P')= \cup_{r\in Z} \{\{i'\}:i'\in [i'_{r}+1,i'_{r+1}]\cap \Z\}\cup_{r \notin Z}\{[i'_{r}+1,i'_{r+1}]\cap \mathbb{Z}\}.$$ Clearly, $l(P_{*})=l(P'_{*})$ and the condition in the definition of $Z$ implies that $l(P)=l(P')$. Let $\tau$ denote the increasing bijection from $\mathcal{B}$ to $\mathcal{B'}$. Clearly in this case $\Upsilon(A,A')=(P,P',\tau)$ is a generalised mapping and is called the generalised mapping induced by the marked partition pair ($P_{*},P'_{*},Z$).

The following lemma gives condition under which an induced generalised mapping is admissible.

\begin{lemma}
\label{marked}
Let $A,A',B,B'$ be as above. Let $P_{*}=\{n_1=i_0<i_1<\ldots <i_{l(P_{*})}=n_1+n\}\in \mathcal{P}(A)$ and $P'_{*}=\{n_1'=i_0<i'_1<\ldots <i'_{l(P'_{*})}=n_1'+n'\}\in \mathcal{P'_{*}}$ be partitions of $A$ and $A'$ respectively of equal length. Let $B_{P_*}=\{r: B\cap [i_r+1,i_{r+1}]\neq \emptyset\}$ and $B_{P'_{*}}=\{r: B' \cap [i'_r+1,i_{r+1}]\neq \emptyset\}$. Let us suppose the following conditions hold.

\begin{enumerate}
\item[(i)] $(P_{*}, P'_{*}, B_{P_*}\cup B_{P'_{*}})$ is a marked partition pair.
\item[(ii)] For $r\notin B_{P_*}\cup B_{P'_{*}})$, $(i_{r+1}-i_{r})\wedge (i'_{r+1}-i'_r)>L_j^2$ and $\frac{1-2^{-(j+5/4)}}{R}<\frac{i'_{r+1}-i'_{r}}{i_{r+1}-i_{r}}< R(1+2^{-(j+5/4)})$.
\item[(iii)] $B_{P_*} \cap B_{P'_{*}}=\emptyset$,
\end{enumerate}
then the induced generalised mapping $\Upsilon$ is admissible of $Class~G^j$.
\end{lemma}
\begin{proof}
Straightforward.
\end{proof}

The usefulness of making these abstract definitions follow from the next couple of lemmas.

\begin{lemma}
\label{compressrounding}
Let $X=(X_1,X_2,\ldots )$ be a sequence of $\mathbb{X}$ blocks at level $j$ and $Y=(Y_1,Y_2,\ldots )$ be a sequence of $\mathbb{Y}$ blocks at level $j$.
Further suppose that $n,n'$ are such that $X^{[1,n]}$ and $Y^{[1,n']}$ are both ``good" segments, $n>L_j$ and $\frac{1-2^{-(j+5/4)}}{R}\leq \frac{n'}{n}\leq R(1+2^{-(j+5/4)})$. Then $X^{[1,n]}\hookrightarrow Y^{[1,n']}$.
\end{lemma}

\begin{proof}
Let us write $n=kR_j+r$ where $0\leq r<R_j$ and $k\in \mathbb{N}\cup \{0\}$.
Now let $s=[\frac{n'-r}{k}]$. Define $0=t_0<t_1<t_2<\ldots t_k=n'-r\leq t_{k+1}=n'$ such that for all $i\leq k$, $t_{i}-t_{i-1}=s$ or $s+1$.
Now from $\frac{1-2^{-(j+5/4)}}{R} \leq \frac{n'}{n}\leq R(1+2^{-(j+5/4)})$, it follows that for large enough $L_j$,
$R_j^-+1\leq s \leq R_j^{+}-1$. From the inductive hypothesis on compression, it follows that, $X^{[iR_j+1,(i+1)R_j]}\hookrightarrow Y^{[t_i+1,t_{i+1}]}$ for $0\leq i\leq k-1$. The lemma follows.
\end{proof}

Let $X=(X_1,X_2,X_3,\ldots X_n)$ be an $\mathbb{X}$-block (or a segment of $\mathbb{X}$-blocks) at level $(j+1)$ where the $X_i$ denote the $j$-level sub-blocks constituting it. Similarly, let $Y=(Y_1,Y_2,Y_3,\ldots Y_{n'})$ be an $\mathbb{Y}$-block (or a segment of $\mathbb{Y}$-blocks) at level $(j+1)$. Let $B_X=\{i: X_i\notin G_j^{\mathbb{X}}\}=\{l_1<l_2<\ldots <l_{K_X}\}$ denote the positions of ``bad" $X$-blocks. Similarly let $B_Y=\{i: X_i\notin G_j^{\mathbb{Y}}\}=\{l'_1<l'_2<\ldots <l'_{K_Y}\}$ be the positions of ``bad" $Y$-blocks.

We next state an easy proposition.

\begin{proposition}
\label{admissiblemap}
%\label{map}
Suppose there exist a generalised mapping $\Upsilon$ given by $\Upsilon([n],[n'])=(P,P',\tau)$ which is admissible and is of $Class~G^j$ with respect to $(B_X, B_Y)$. Further, suppose, for $1\leq i\leq K_Y$, $X_{l_i}\hookrightarrow Y_{\tau(l_i)}$ and for each $1\leq i\leq K_Y$, $X_{\tau^{-1}(l'_i)}\hookrightarrow Y_{l'_i}$. Then $X\hookrightarrow Y$.
\end{proposition}

\begin{proof}
Let $P,P'$ be as above with $l(P)=l(P')=z$. Let us fix $0\leq r\leq z-1$. From the definition of an admissible mapping, it follows that there are 3 cases to consider.

\begin{enumerate}
\item[$\bullet$] $i_{r+1}-i_r=i'_{r+1}-i'_{r}=1$ and either $i_{r}+1\in B_X$ or $i'_{r}+1\in B_Y$. In either case it follows from the hypothesis that $X_{i_{r}+1}\hookrightarrow Y_{i'_{r}+1}$.

\item[$\bullet$] $i_{r+1}-i_r=i'_{r+1}-i'_{r}=1$, $i_{r}+1\notin B_X$, $i'_{r}+1\notin B_Y$. In this case $X_{i_{r}+1}\hookrightarrow Y_{i'_{r}+1}$ follows from the inductive hypothesis that good level $j$ $\mathbb{X}$-blocks map into good level $j$ $\mathbb{Y}$-blocks.

\item[$\bullet$] $i_{r+1}-i_r\neq 1$. In this case both $X^{[i_r+1,i_{r+1}]}$ and $Y^{[i'_r+1,i'_{r+1}]}$ are good segments, and it follows from Lemma \ref{compressrounding} that $X^{[i_r+1,i_{r+1}]}\hookrightarrow Y^{[i'_r+1,i'_{r+1}]}$.
\end{enumerate}

Hence for all $r$, $0\leq r\leq z-1$, $X^{[i_{r}+1,i_{r+1}]}\hookrightarrow Y^{[i'_r+1,i'_{r+1}]}$. It follows that $X\hookrightarrow Y$, as claimed.
\end{proof}

In the same vein, we state the following Proposition whose proof is essentially similar and hence omitted.

\begin{proposition}
\label{admissiblemap2}
%\label{map}
Suppose there exist a generalised mapping $\Upsilon^j$ given by $\Upsilon([n],[n'])=(P,P',\tau)$ which is admissible and is of $Class~H_1^j$ or $H_2^j$ with respect to $B_X$. Further, suppose, for $1\leq i\leq K_X$, $X_{l_i}\hookrightarrow Y_{\tau(l_i)}$ and for each $i'\in [n'] \setminus \{\tau(l_i): 1\leq i \leq K_X\}$, $Y_{i'}\in G_j^{\mathbb{Y}}$. Then $X\hookrightarrow Y$.
\end{proposition}

\section{Constructions}\label{s:construction}
In this section we provide the necessary constructions which we would use in later sections to prove different estimates on probabilities that certain blocks can be mapped to certain other blocks.

\begin{proposition}
\label{compresspartition}
Let  $j\geq 1$ and $n,n'>L_j^{\alpha-1}$  such that
\begin{equation}
\label{lengthratiocp}
\frac{1-2^{-(j+7/4)}}{R}\leq \frac{n'}{n} \leq R(1+2^{-(j+7/4)}).
\end{equation}

Let $B=\{l_1<l_2<\ldots <l_{k_x}\}\subseteq [n]$ and $B'=\{l'_1<l'_2<\ldots l'_{k_y}\} \subseteq [n']$ be such that $l_1,l'_1, (n-l_{k_x}),(n'-l'_{k_y})>L_j^3$, $k_x, k_y \leq k_0R_{j+1}^{+}$. Then there exist a family of admissible generalised mappings $\Upsilon_h$ for $1\leq h \leq L_j^2$, such $\Upsilon_h([n],[n'],B,B')=(P_h,P'_h,\tau_h)$ is of class $G^j$ and such that for all $h$, $1\leq h\leq L_j^2$, $1\leq i\leq k_x$, $1\leq r \leq k_y$, $\tau_h(l_i)=\tau_{1}(l_i)+h-1$ and $\tau_h^{-1}(l'_r)=\tau_1^{-1}(l'_r)-h+1$.
\end{proposition}

To prove Proposition \ref{compresspartition} we need the following two lemmas.

\begin{lemma}
\label{compresspartition1}
Assume the hypotheses of Proposition \ref{compresspartition}. Then there exists a piecewise linear increasing bijection $\psi: [0,n]\mapsto [0,n']$ and two partitions $Q$ and $Q'$ of $[0,n]$ and $[0,n']$ respectively of equal length ($=q$, say), given by $Q=\{0=t_0<t_1<\ldots t_{q-1}<t_q=n\}$ and $Q=\{0=\psi(t_0)<\psi(t_1)<\ldots \psi(t_{q-1})<\psi(t_q)=n'\}$ satisfying the following properties:
\begin{enumerate}
\item $q\leq 2k_0R_{j+1}^{+}+2$.
\item $t_1,\psi(t_1)> L_j^3/2$, $(n-t_{q-1}),(n'-\psi(t_{q-1})>L_j^3/2$.
\item  None of the intervals $[t_{r-1}, t_r]$ intersect both $B$ and $\psi^{-1}(B')$.
\item None of the intervals $[\psi(t_{r-1}),\psi(t_r)]$ intersect both $B'$, and $\psi(B)$.
\item $\frac{1-2^{-(j+3/2)}}{R}\leq \frac{\psi(t_r)-\psi(t_{r-1}))}{t_r-t_{r-1}} \leq R(1+2^{-(j+3/2)})$ for each $r\leq q$.
\item Suppose $i\in (B \cup \psi^{-1}(B'))\cap [t_{r-1},t_r]$. Then $|i-t_{r-1}|\wedge |t_{r}-i| \geq L_j^{9/4}$. Similarly if $i'\in (B' \cup \psi_(B))\cap [\psi(t_{r-1}),\psi(t_r)]$, then $|i'-\psi(t_{r-1})|\wedge |\psi(t_{r})-i| \geq L_j^{9/4}$.
\end{enumerate}
\end{lemma}

\begin{proof}
Let us define a family of maps $\psi_s: [0,n]\rightarrow [0,n']$, $0\leq s \leq L_j^{5/2}$ as follows:

\begin{equation}
\psi_s(x)=\left\{
\begin{array}{l l l}
 x \frac{L_j^3/2+s}{L_j^3/2} ~~~&\text{if }& x\leq L_j^3/2\\
L_j^3/2+s+ \frac{n'-L_j^3}{n-L_j^3}(x-L_j^3/2)~~~&\text{if }& L_j^3/2 \leq x \leq n-L_j^3/2\\
n' - (n-x)(\frac{L_j^3/2-s}{L_j^3/2})~~~&\text{if }&  n-L_j^3/2\leq x\leq n'.
\end{array}\right.
\end{equation}
It is easy to see that $\psi_s$ is a piecewise linear bijection for each $s$ with the piecewise linear inverse being given by
\begin{equation}
\psi_s^{-1}(y)=\left\{
\begin{array}{l l l}
 y \frac{L_j^3/2}{L_j^3/2+s} ~~~&\text{if }& y\leq L_j^3/2+s\\
L_j^3/2+ \frac{n-L_j^3}{n'-L_j^3}(y-L_j^3/2-s)~~~&\text{if }& L_j^3/2+s \leq y \leq n'-L_j^3/2+s\\
n - (n'-x)(\frac{L_j^3/2}{L_j^3/2-s})~~~&\text{if }&  n'-L_j^3/2+s\leq y\leq n'.
\end{array}\right.
\end{equation}

Let $S$ be distributed uniformly on $[0,L_j^{5/2}]$, and  consider the random map $\psi _S$.   Let
$$E=\{|\psi_S(i)-i'|\geq 2L_j^{9/4}, |i-\psi_S^{-1}(i')|\geq 2L_j^{9/4} \forall i\in B, i'\in B'\}.$$

Also notice that for $L_0$ sufficiently large, we get from (\ref{lengthratiocp}) that $\frac{1-2^{-(j+13/8)}}{R}\leq \frac{n'-L_j^3}{n-L_j^3} \leq R(1-2^{-(j+13/8)})$. It follows that for $i\in B, i'\in B'$, $\mathbb{P}(|\psi_S(i)-i'|< 2L_j^{9/4})\leq \frac{6RL_j^{9/4}}{L_j^{5/2}}=\frac{6R}{L_j^{1/4}}$. Similarly
$\mathbb{P}(|i-\psi_S^{-1}(i')|< 2L_j^{9/4})\leq \frac{6R}{L_j^{1/4}}$.
It now follows from $k_x,k_y\leq k_0R_{j+1}^{+}$ that
for $L_0$ large enough
$$\mathbb{P}(E)\geq 1- \frac{6Rk_0^2(R_{j+1}^{+})^2}{L_j^{1/4}}>0.$$
So it follows that there exists $s_0\in [0, L_j^{5/2}]$ such that $|\psi_{s_0}(i)-i'|\geq 2L_j^{9/4}, |i-\psi_{s_0}^{-1}(i')|\geq 2L_j^{9/4} \forall i\in B, i'\in B'$.

Setting $\psi=\psi_{s_0}$ it is now easy to see that for sufficiently large $L_0$ there exists $0=t_0<t_1<\ldots <t_q=n\in [0,n]$ satisfying the properties in the statement of the lemma. One way to do this is to choose $t_k$'s at the points $\frac{l_i+\psi^{-1}(l'_{i'})}{2}$ where $i,i'$ are such that there does not exist any point in the set $B\cup \psi ^{-1}(B')$ in between $l_i$ and $\psi^{-1}(l'_{i'})$ with the obvious modifications for $t_1$ and $t_{q-1}$ satisfying condition $(2)$ of the lemma. That such a choice satisfies the properties $(1)-(6)$ listed in the lemma is easy to verify.
\end{proof}

\begin{figure*}[h]\label{f:mapcase1}
\begin{center}
\includegraphics[height=12.35cm,width=17.1cm]{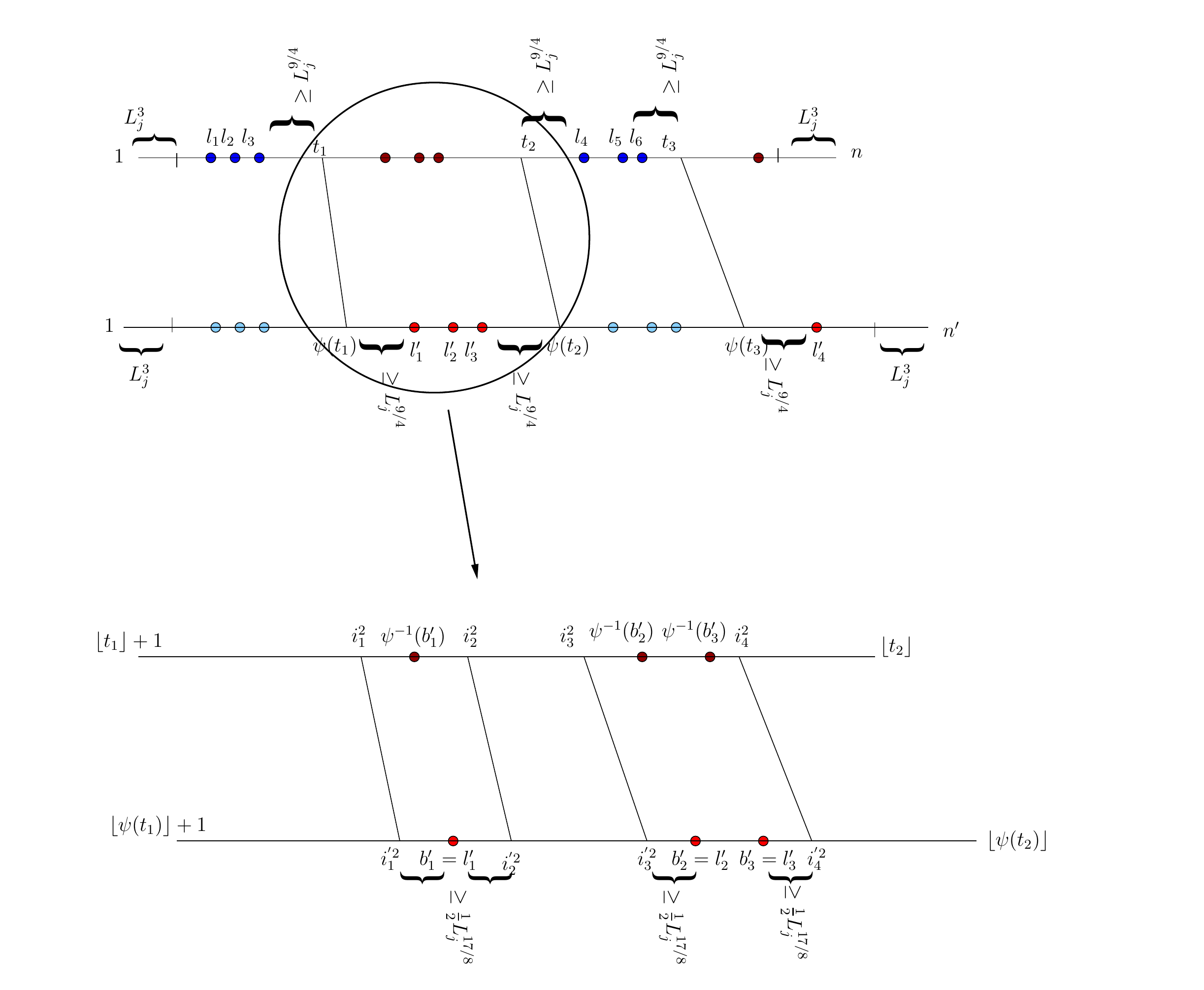}
\caption{The upper figure  illustrates a function $\psi$ and partitions $0=t_0<t_1<\ldots <t_q=n'$ and $0=\psi(t_0)<\psi(t_1)<\ldots <\psi(t_q)=n'$ as described in Lemma \ref{compresspartition1}. The lower figure illustrates  the further sub-division of an  interval $[t_1,t_2]$ as described in Lemma \ref{compresspartition2}. The neighborhood of elements of $b$ are mapped rigidly so above we have $i_2^2-i_1^2=i_2^{'2}-i_1^{'2}$ and $i_4^2-i_3^2=i_4^{'2}-i_3^{'2}$.}
\end{center}
\end{figure*}

\begin{lemma}
\label{compresspartition2}
Assume the hypotheses of Proposition \ref{compresspartition}. Then there exist partitions $P_*$ and $P'_*$ of $[n]$ and $[n']$ of equal length ($=z$, say) given by $P_*=\{0=i_0<i_1<\ldots <i_{z}=n\}$ and $P'_*=\{0=i'_0<i'_1<\ldots <i'_{z}=n'\}$ satisfying the following properties:
\begin{enumerate}
\item $(P_*,P'_*,B_{P_*}\cup B'_{P'_*})$ is a marked partition pair.
\item For $r\notin B_{P_*}\cup B'_{P'_{*}}$, $(i_{r+1}-i_{r})\wedge (i'_{r+1}-i'_r)\geq \lfloor 2L_j^{17/8} \rfloor$ and $\frac{1-2^{-(j+3/2)}}{R} \leq \frac{i_{r}-i_{r-1}}{i'_{r}-i'_{r-1}} \leq R(1+2^{-(j+3/2)})$.
\item $B_{P_*}\cap B'_{P'_*}=\emptyset$.
\item $0, z-1\notin B_{P_{*}} \cup B_{P'_{*}}$.
\item If $l_i\in [i_{r}+1,i_{r+1}]$, then $|l_i-i_{r}|\wedge |l_i-i_{r+1}|>\frac{1}{2}L_j^{17/8}$. Similarly if $l'_i\in [i'_{r}+1,i'_{r+1}]$, then $|l'_i-i'_{r}|\wedge |l'_i-i'_{r+1}|>\frac{1}{2}L_j^{17/8}$.
\end{enumerate}
Here $B_{P_*}=\{r: B\cap [i_r+1,i_{r+1}]\neq \emptyset\}$ and $B'_{P'_*}=\{r: B' \cap [i'_r+1,i'_{r+1}]\neq \emptyset\}$.
\end{lemma}

\begin{proof}
Choose the map $\psi$ and the partitions $Q,Q'$ as constructed in the proof of Lemma~\ref{compresspartition1}. For $1\leq r\leq q$, now let us fix an interval $[t_{r-1},t_{r}]$. We need to consider two cases.\\

\textbf{Case 1:} $B_r:=B\cap [t_{r-1},t_{r}]=\{b_1<b_2<\ldots< b_{k_r}\}\neq \emptyset$.\\

Clearly $k_r\leq k_0R_{j+1}^{+}$.  We now define a partition $P^r=\{\lfloor t_{r-1} \rfloor =i_0^r<i_1^r<\ldots <i_{z_r}^r=\lfloor t_{r} \rfloor \}$ of $[\lfloor t_{r-1} \rfloor +1,\lfloor t_{r}\rfloor]$ as follows.

\begin{enumerate}
\item[$\bullet$] $i_1^r=b_1-\lfloor L_j^{17/8}\rfloor$.
\item[$\bullet$] For $h\geq 1$, if $[i_{h-1}^r,i_{h}^r]\cap B_r=\emptyset$, then define, $i_{h+1}^r=\min\{i\geq i_h^r+\lfloor L_j^{17/8}\rfloor: B_r\cap [i-\lfloor L_j^{17/8}\rfloor, i+\lfloor 3L_j^{17/8}\rfloor]=\emptyset\}$.
\item[$\bullet$] For $h\geq 1$, if $[i_{h-1}^r,i_{h}^r]\cap B_r \neq \emptyset$, define $i_{h+1}^r=\min\{i\geq i_h^r+\lfloor 2L_j^{17/8}\rfloor: i+\lfloor L_j^{17/8}\rfloor+1\in B_r\}$ or $\lfloor t_r\rfloor$ if no such $i$ exists which ends the partition.
\end{enumerate}

Notice that the construction implies that $i^r_{z_r-1}=b_{k_r}+L_j^{17/8}$. Also $i^r_{h+1}-i^r_{h}\geq 2L_j^{17/8}$ for all $h$. Also notice that this implies that alternate blocks of this partition intersect $B_r$ and hence $z_r\leq 2k_0R_{j+1}^{+}+2$. It also follows that the total length of the blocks intersecting $B_{r}$ is at most $4k_0R_{j+1}^{+}L_j^{17/8}$.

Now we construct a corresponding partition $P'^{r}=\{\lfloor \psi(t_{r-1}) \rfloor  =i_0^{'r}<i_1^{'r}<\ldots <i_{z_r}^{'r}=\lfloor \psi(t_{r}) \rfloor \}$ of $[\lfloor \psi(t_{r-1}) \rfloor +1,\lfloor \psi(t_{r})\rfloor]$ as follows.

\begin{enumerate}
\item[$\bullet$] $i_1^{'r}=\lfloor \psi(b_1)\rfloor -\lfloor L_j^{17/8}\rfloor$.
\item[$\bullet$] For $1\leq h \leq z_r-2$, $i_{h+1}^{'r}= i_{h}^{'r}+ (i_{h+1}^r-i_{h}^r)$, when $B_r\cap [i_h^r, i_{h+1}^r]\neq \emptyset$, and  $i_{h+1}^{'r}= i_{h}^{'r}+ \lfloor (i_{h+1}^r-i_{h}^r)\frac{[\lfloor \psi(t_{r}) \rfloor-\lfloor \psi(t_{r-1})\rfloor]}{[\lfloor t_{r} \rfloor-\lfloor t_{r-1}\rfloor]}\rfloor$ otherwise.
\end{enumerate}

Notice that condition (5) of Lemma \ref{compresspartition1} implies that
$$ |(i^{'r}_{z_r}-i^{'r}_{z_r-1})-\frac{[\lfloor \psi(t_{r}) \rfloor-\lfloor \psi(t_{r-1})\rfloor]}{[\lfloor t_{r} \rfloor-\lfloor t_{r-1}\rfloor]}(i^r_{z_r}-i^r_{z_r-1})|\leq 2R(4k_0R_{j+1}^{+}L_j^{17/8}+2k_0R_{j+1}^{+}+2),$$
which together with condition (6) of  Lemma \ref{compresspartition1} implies that for $L_0$ sufficiently large this gives a valid partition of $[\lfloor \psi(t_{r-1}) \rfloor +1,\lfloor \psi(t_{r})\rfloor]$. Moreover, since $\frac{2R(4k_0R_{j+1}^{+}L_j^{17/8}+2k_0R_{j+1}^{+}+2)}{L_j^{9/4}}$ is small for $L_0$ sufficiently large it follows that for $L_0$ large enough we have
$$\frac{1+2^{-(j+3/2)}}{R}\leq \frac{i_{h+1}^{'r}-i_{h}^{'r}}{i_{h+1}^r-i_h^{r}}\leq R(1+2^{-(j+3/2)}).$$

\textbf{Case 2:} $B'_r:=B'\cap [\lfloor \psi(t_{r-1}) \rfloor, \lfloor \psi(t_r) \rfloor]=\{b'_1<b'_2<\ldots< b'_{k'_r}\}\neq \emptyset$.\\

Clearly $k'_r\leq k_0R_{j+1}^{+}$. In this case we start with defining a partition $P^{'r}=\{\lfloor \psi (t_{r-1}) \rfloor =i_0^{'r}<i_1^{'r}<\ldots <i_{z_r}^{'r}=\lfloor \psi(t_{r}) \rfloor \}$ of $[\lfloor \psi(t_{r-1}) \rfloor +1,\lfloor \psi(t_{r})\rfloor]$ as follows.

\begin{enumerate}
\item[$\bullet$] $i_1^{'r}=b'_1-\lfloor L_j^{17/8}\rfloor$.
\item[$\bullet$] For $h\geq 1$, if $[i_{h-1}^{'r},i_{h}^{'r}]\cap B'_r=\emptyset$, then define, $i_{h+1}^{'r}=\min\{i\geq i_h^{'r}+\lfloor L_j^{17/8}\rfloor: B'_r\cap [i-L_j^{17/8}, i+\lfloor 3L_j^{17/8}\rfloor]=\emptyset\}$.
\item[$\bullet$] For $h\geq 1$, if $[i_{h-1}^{'r},i_{h}^{'r}]\cap B'_r\neq \emptyset$, define $i_{h+1}^{'r}=\min\{i\geq i_h^{'r}+\lfloor 2L_j^{17/8}\rfloor: i+\lfloor L_j^{17/8}\rfloor+1\in B'_r\}$  or $\lfloor \psi(t_r)\rfloor$ if no such $i$ exists which ends the partition.
\end{enumerate}

%Notice that the construction implies that $i^r_{z_r-1}=b_{K_r}+L_j^{17/8}$. \textsf{(provided $L_j^{9/4}\geq 8Lj^{17/8}$)}. Also $i^r_{h+1}-i^r_{h}\geq 2L_j^{17/8}$ for all $h$. Also notice that the total length of all the "bad" partition-blocks is at most $4k_0L_j^{17/8}$. More over $z_r\leq 2k_0+2$.
%
%Now we construct a corresponding partition $P'^{r}=\{\lfloor \psi(t_{r}) \rfloor  =i'_0^r<i'_1^r<\ldots <i'_{z_r}^r=\lfloor \psi(t_{r+1}) \rfloor \}$ of $[\lfloor \psi(t_{r-1})+1 \rfloor,\lfloor \psi(t_{r})\rfloor]$ as follows.

%
%First let us observe that the procedure described above gives a valid partition of $[\lfloor \psi(t_{r}) \rfloor +1,\lfloor (t_{r+1})\rfloor]$. Since $i'_{z_r-1}^r \leq \psi(i^r_{z_r-2})+ 4k_0L_j^{17/8} +2k_0+2 \leq \psi_{s_0}(i^r_{l_r})$ from item (4) above for $L_j(L_0)$ sufficiently large.\\

%\begin{enumerate}
%\item[$\bullet$] $i'_1^r=b'_1-2L_j^{17/8}$.
%\item[$\bullet$] For $h\geq 1$, $i'_{h+1}^r=\min\{i\geq i'_h^r+2L_j^{17/8}: B_Y^r\cap [i-2L_j^{17/8}, i+2L_j^{17/8}]=\emptyset, i+2L_j^{17/8}+1\in B_Y^r}$.
%\end{enumerate}

%As before this construction implies that $i^r_{z_r-1}=b'_{k'_r}+L_j^{17/8}$. Also $i^r_{h+1}-i^r_{h}\geq 2L_j^{17/8}$ for all $h$, alternate blocks of this partition intersect $B'_r$ and hence $z_r\leq 2k_0R_{j+1}^{+}+2$. It also follows that the total length of the blocks intersecting $B_r$ is at most $4k_0R_{j+1}^{+}L_j^{17/8}$.

As before, next we construct a corresponding partition $P^{r}=\{\lfloor t_{r-1} \rfloor  =i_0^r<i_1^r<\ldots <i_{z_r}^r=\lfloor t_{r} \rfloor \}$ of $[\lfloor (t_{r-1}) \rfloor + 1,\lfloor t_{r})\rfloor]$ as follows.

\begin{enumerate}
\item[$\bullet$] $i_1^r=\lfloor \psi^{-1}(b'_1)\rfloor -\lfloor L_j^{17/8}\rfloor$.
\item[$\bullet$] For $z_r-2\geq h\geq 1$, $i_{h+1}^r= i_{h}^r+ (i_{h+1}^{'r}-i_{h}^{'r})$, provided $B'_r\cap [i_h^{'r}, i_{h+1}^{'r}]\neq \emptyset$, and  $i_{h+1}^r= i_{h}^r+ \lfloor (i_{h+1}^{'r}-i_{h}^{'r})\frac{[\lfloor t_{r} \rfloor-\lfloor t_{r-1}\rfloor]}{[\lfloor \psi(t_{r}) \rfloor-\lfloor \psi(t_{r-1})\rfloor]}\rfloor$, otherwise.
\end{enumerate}

As before it can be verified that the procedure described above gives a valid partition of $[\lfloor \psi(t_{r-1})+1 \rfloor,\lfloor (t_{r})\rfloor]$ such that for $L_0$ large enough we have for every $h$
$$\frac{1+2^{-(j+3/2)}}{R}\leq \frac{i_{h+1}^{'r}-i_{h}^{'r}}{i_{h+1}^r-i_h^{r}}\leq R(1+2^{-(j+3/2)}).$$

Let us define, $P_*=\cup_{r}P^r$ and $P'_*=\cup_{r}{P'}^{r}$ where $\cup P^r$ simply means the partition containing the points of all $P^r$'s (or alternatively, $\mathcal{B}(\cup P^r)=\cup _r \mathcal{B}(P^r)$). It is easy to check that $(P_*, P'_{*})$ satisfies all the properties listed in the lemma.

\end{proof}

The procedure for constructing $(P_*,P'_{*})$ as described in Lemma \ref{compresspartition1} and \ref{compresspartition2} is illustrated in Figure \ref{f:mapcase1}.

\begin{proof}(of Proposition \ref{compresspartition})
Construct the partitions $P_{*}$ and $P'_{*}$ of $[n]$ and $[n']$ respectively as in Lemma \ref{compresspartition2}. Let
$P_*=\{0=i_0<i_1<\ldots <i_{z-1}<i_z=n\}$ and $P'_{*}=\{0=i'_0<i'_1<\ldots <i'_{z-1}<i'_z=n'\}$. Also notice that it follows from the construction described in Lemma \ref{compresspartition2} that $B_{P_*}\cup B_{P'_{*}}$ does not contain two consecutive integers. For $1\leq h \leq L_j^2$ we let $i_r^h=i_r$ and so $P^h_{*}=P_*$ while we define $i_r^{'h}=i_r+h-1$ for $1\leq r \leq z-1$ so that $P_{*}^{'h}=\{0=i_0^{'h}<i_1^{'h}<\ldots <i_{z-1}^{'h}<i_z^{'h}=n'\}$.

First we observe that the above definition is consistent and gives rise to a valid partition pair $(P_{*}^h,P_{*}^{'h})$ for each $h$, $1\leq h \leq L_j^2$. From item $(5)$, in the statement of Lemma~\ref{compresspartition2} it follows that for each $h$,  $(P_{*}^h,P_{*}^{'h}, B_{P_*}\cup B_{P'_*})$ forms a marked partition pair. Furthermore, for each $h$, this marked partition pair satisfies

\begin{enumerate}
\item For $r\notin B_{P^h_*}\cup B_{P_{*}^{'h}}$, $(i_{r+1}-i_{r})\wedge (i'_{r+1}-i'_r)>L_j^2$ and $\frac{1+2^{-(j+3/2)}}{R}<\frac{i'_{r+1}-i'_{r}}{i_{r+1}-i_{r}}< R(1+2^{-(j+3/2)})$ for $L_j$ sufficiently large.
\item $B_{P_{*}^h}\cap B_{P_{*}^{'h}} =\emptyset$.
\end{enumerate}

Hence it follows from Lemma \ref{marked} that for each $h$, the generalized mapping $\Upsilon_h([n],[n'],B,B')=(P_h,P'_h,\tau_h)$ induced by the marked partition pair $(P_{*}^h,P_{*}^{'h}, B_{P_*}\cup B_{P'_{*}})$ is an admissible mapping of class $G^j$. It follows easily from definitions that for all $h$, $1\leq h\leq L_j^2$, $1\leq i\leq k_x$, $1\leq r \leq k_y$, $\tau_h(l_i)=\tau_{1}(l_i)+h-1$ and $\tau_h^{-1}(l'_r)=\tau_1^{-1}(l'_r)-h+1$. This procedure is illustrated in Figure \ref{f:admissible}.
\end{proof}

\begin{figure*}[h]\label{f:admissible}
\begin{center}
\includegraphics[height=8cm,width=18cm]{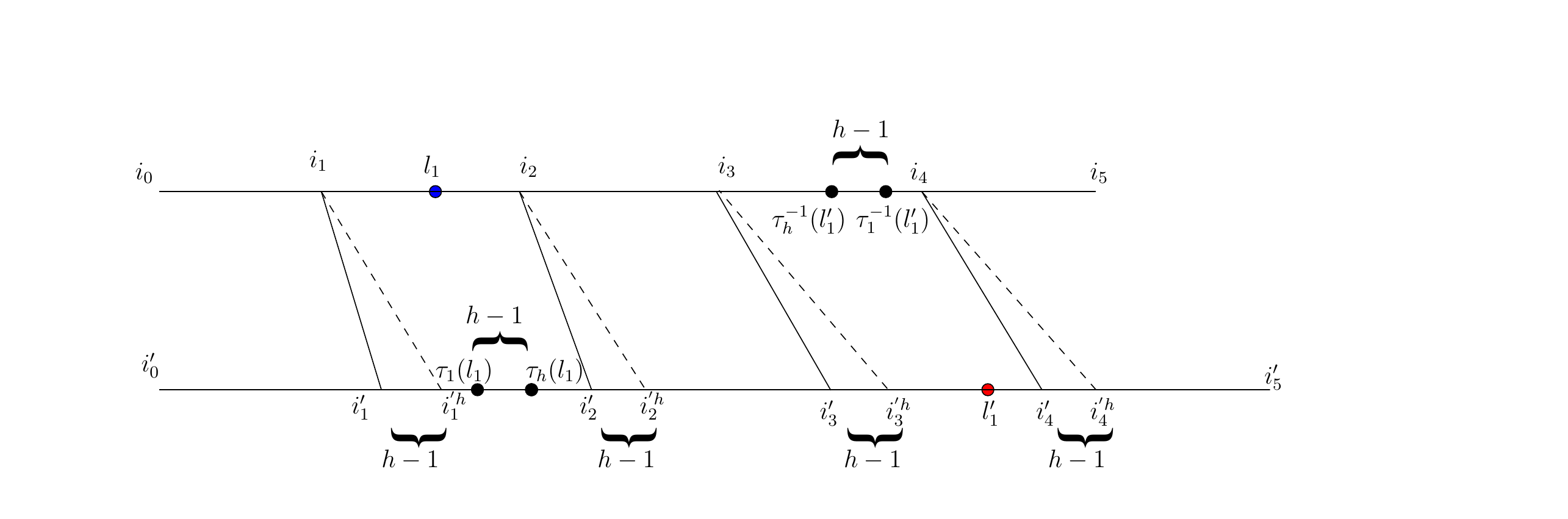}
\caption{Constructing  generalized mappings $(P_h,P'_h,\tau_h)$ from ($P_{*},P_{*})$ as described in the proof of Proposition \ref{compresspartition}.}
\end{center}
\end{figure*}

\begin{proposition}
\label{mapcase2}
For $j\geq 1$ let $n,n'>L_j^{\alpha-1}$ be such that
\begin{equation}
\label{lengthratiocpA}
\frac{1}{R}\leq \frac{n'}{n} \leq R.
\end{equation}
Let $B=\{l_1<l_2<\ldots <l_{k_x}\}\subseteq [n]$ be such that $l_1, (n-l_{k_x})>L_j^3$, $k_x \leq k_0$. Then there exist a family of admissible mappings $\Upsilon_h$ for $1\leq h \leq L_j^2$, $\Upsilon^j_h([n],[n'])=(P_h,P'_h,\tau_h)$ which are of $Class~H_1^j$ with respect to $B$ such that for all $h$, $1\leq h\leq L_j^2$, $1\leq i\leq k_x$ we have that $\tau_h(l_i)=\tau_{1}(l_i)+h-1$.
\end{proposition}

\begin{proof}
This proof is a minor modification of the proof of Proposition \ref{compresspartition}. Clearly, as in the proof of Proposition \ref{compresspartition}, we can construct $L_j^2$ admissible mappings $\Upsilon^{*}_h(A,A')=(P^*_h,P^{'*}_{h})$ which  are of Class $G^j$ with respect to $(B,\emptyset)$ where $A=\{L_j^3/2+1,L_j^3/2+2,...,n-L_j^3/2\}$ and $A'=\{L_j^3+1,L_j^3+2,...,n'-L_j^3\}$.  Denote $P^*_h=\{L_j^3+1=i_0^h<i^h_1<\ldots <i^h_{z-1}<i^h_z=n-L_j^3\}$. Define the partition $P_h$ of $[n]$ as $P_h=\{0< i^h_0 <i^h_1<\ldots <i^h_{z-1}<i^h_z< n\}$, that is with segments of length $L_j^3/2$ added to each end of $P_h^*$.  Define $P'_h$ similarly by adding segments of length $L_j^3$ to each end of $P^{'*}_h$. It can be easily checked that for each $h$, $1\leq h \leq L_j^2$, $(P_h,P'_h,\tau_h)$  is an admissible mapping which is of $Class~H_1^j$ with respect to $B$.
\end{proof}

\begin{proposition}
\label{mapcase3}
Let For $j\geq 1$, $n,n'>L_j^{\alpha-1}$ be such that
\begin{equation}
\label{lengthratiocpB}
\frac{3}{2R}\leq \frac{n'}{n} \leq \frac{2R}{3}.
\end{equation}

Let $B=\{l_1<l_2<\ldots <l_{k_x}\}\subseteq [n]$ be such that $l_1, (n-l_{k_x})>L_j^3$, $k_x \leq \frac{n-2L_j^3}{10R_j^{+}}$. Then there exist an admissible mapping $\Upsilon([n],[n'])=(P,P',\tau)$ which is of $Class~H_2^j$ with respect to $B$.
\end{proposition}

To prove this proposition we need the following lemma.

\begin{lemma}
\label{mapcase3lemma}
Assume the hypotheses of Proposition \ref{mapcase3}. Then there exists partitions $P$ and $P'$ of $[n]$ and $[n']$ of equal length ($=z$, say) given by $P_{*}=\{0=i_0<i_1=L_j^3<\ldots i_{z-1}=n-L_j^3<i_{z}=n\}$ and $P'_{*}=\{0=i'_0<i'_1=L_j^3<\ldots <i'_{z-1}=n'-L_j^3<i'_{z}=n'\}$ satisfying the following properties:

\begin{enumerate}
\item $(P_{*},P'_{*},B^*)$ is a marked partition pair for some $B^*\supseteq B_P\cup \{0,z-1\}$ where $B_P=\{r:[i_{r}+1,i_r]\cap B\neq \emptyset\}$.
\item For $r\notin B$, $(i_{r+1}-i_{r})=R_{j}$ and $R_j^{-}\leq i'_{h+1}-i'_{h}\leq R_{j}^{+}$.
\end{enumerate}
\end{lemma}

\begin{figure*}[h]\label{f:mapcase3}
\begin{center}
\includegraphics[height=8cm, width=16cm]{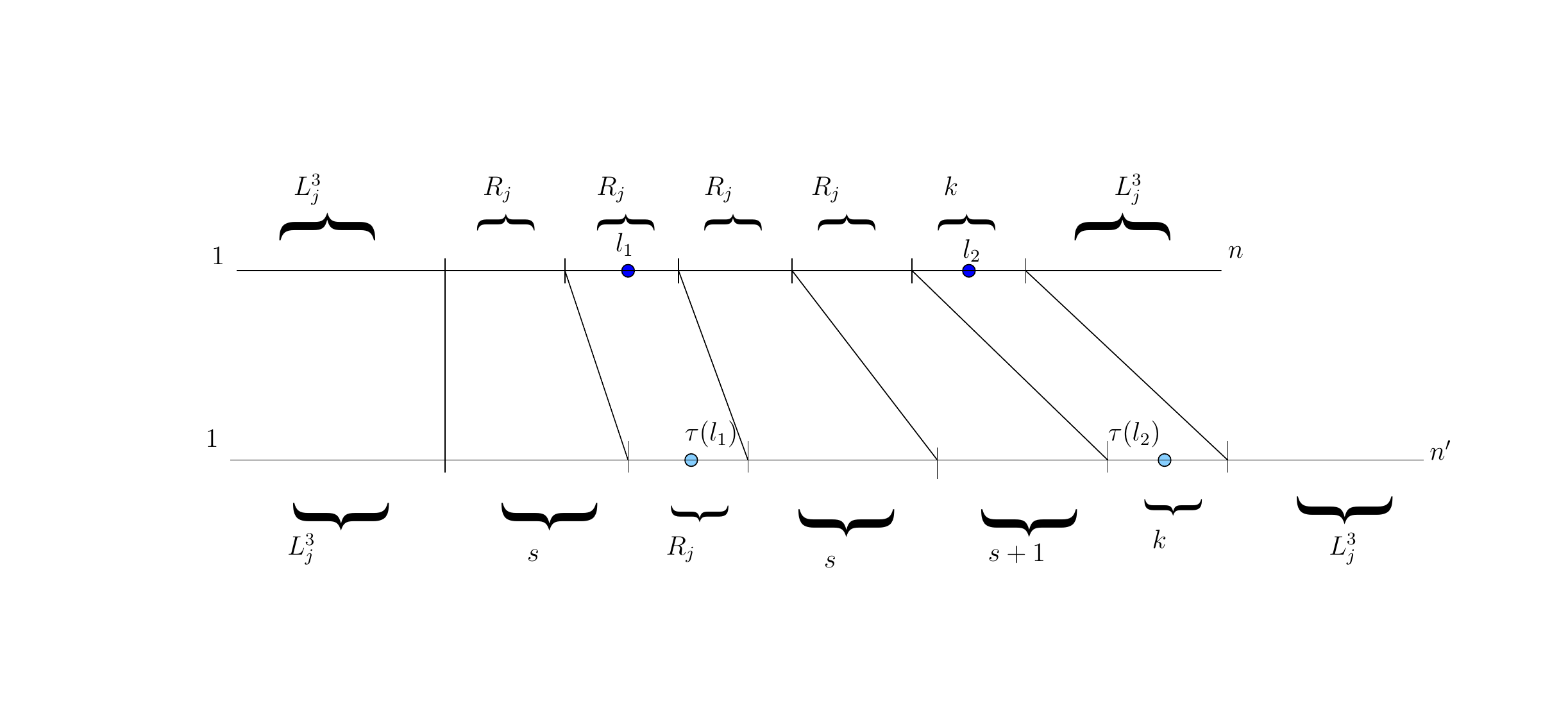}
\caption{Marked Partition pair of $[n]$ and $[n']$ as described in Lemma~\ref{mapcase3lemma} and the induced generalised mapping.}
\end{center}
\end{figure*}

\begin{proof}
Let us write $n=2L_j^3+kR_j+r$ where $0\leq r <R_j$ and $k\in \mathbb{N}$. Construct the partition $P_{*}=\{0=i_0<i_1=L_j^3<\ldots i_{z-1}=n-L_j^3<i_{z}=n\}$  where we set $i_h=L_j^{3}+(h-1)R_j$ for $h=2,3,\ldots , (k+1)$.

Notice that this implies $z=(k+2)$ or $(k+3)$ depending on whether $r=0$ or not. For the remainder of this proof we assume that $r>0$. In the case $r=0$, the same proof works with the obvious modifications.

Now define $B^*=B_P\cup \{0,z-1\}\cup \{k+1\}$.

%Now it is clear from hypotheses that $|B_P|\leq \frac{L_j^{\alpha -1}+T_X}{10R_j^{+}}=\frac{n-2L_j^3}{10R_{j}^{+}}$. Hence,
Clearly
\begin{equation}\label{mapcase3eqn}
\sum_{h\in B_P\cup \{k+1\}} (i_{h+1}-i_h) \leq R_j\left(\frac{n-2L_j^3}{10R_{j}^{+}}+1 \right)\leq \frac{n-2L_j^3}{9R}
\end{equation}
for $L_0$ sufficiently large.

Also notice that for $L_0$ sufficiently large $\frac{3}{2R}\leq \frac{n-2L_j^3}{n'-2L_j^3}\leq \frac{2R}{3}$.

Now let $$s=\left\lfloor\dfrac{(n'-2L_j^3)-\sum_{h\in B_P\cup \{k+1\}}(i_{h+1}-i_h)}{k+1-|B_P\cup \{k+1\}|}\right \rfloor $$ and let

$$(n'-2L_j^3)-\sum_{h\in B_P\cup \{k+1\}}(i_{h+1}-i_h)= s(k+1-|B_P\cup \{k+1\}|) +r';~~~~0\leq r'< k+1-|B_P\cup \{k+1\}|.$$

\textbf{Claim:} $R_j^{-}\leq s\leq R_j^{+}-1.$\\

\textit{Proof of Claim.}
Clearly, $|B_P\cup \{k+1\}|\leq \frac{(n-2L_j^3)}{10R_j^{+}}-1\leq \frac{(n-2L_j^3)}{9R_j^{+}}\leq \frac{(k+1)R_j}{9R_j^{+}}$. Hence $k+1-|B_P\cup \{k+1\}|\geq (k+1)(1-\frac{R_j}{9R_{j}^{+}})\geq (k+1)\frac{8}{9}$ for $R>10$. It follows that
\begin{eqnarray*}
s &\leq & \frac{n'-2L_j^3}{(k+1)\frac{8}{9}}\\
&=& \dfrac{(n-2L_j^3)\frac{n'-2L_j^3}{n-2L_j^3}}{(k+1)\frac{8}{9}}\\
&\leq & \frac{18(k+1)RR_j}{24(k+1)}=\frac{3}{4}RR_{j}\leq R_{j}^{+}-1.
\end{eqnarray*}

To prove the other inequality let us observe using (\ref{mapcase3eqn}),
\begin{eqnarray*}
s &\geq & \frac{(n'-2L_j^3)-\frac{(n-2L_j^3)}{9R}}{(k+1)}-1\\
&\geq & \frac{(n-2L_j^3)\frac{3}{2R}-\frac{(n-2L_j^3)}{9R}}{(k+1)}-1\\
&\geq & \frac{25kR_j}{18(k+1)R}-1\geq R_j^{-},
\end{eqnarray*}
for $L_0$ sufficiently large.

This completes the proof of the claim. Coming back to the proof of the lemma
let us denote the set $\{1,2,\ldots , k+1\}\setminus (B_P\cup \{k+1\}=\{w_1<w_2<...<w_{d}\}$ where $d=k+1-|B_P\cup \{k+1\}|$. Now we define $P'_{*}=\{0=i'_0<i'_1=L_j^3<\ldots <i'_{z-1}=n'-L_j^3<i'_{z}=n'\}$. We define $i_h'$ inductively as follows
\begin{itemize}
\item We set $i'_1=L_j^3$.
\item For $h\in B_{P}\cup \{k+1\}$, define $i'_{h+1}=i'_h+(i_{h+1}-i_h)$.
\item If $h=w_t$, for some $t$ then define $i'_{h+1}=i'_h+(s+1)$ if $t\leq r'$, and $i'_{h+1}=i'_h+s$.
\end{itemize}
Now from the definition of $s$, it is clear that $i'_{k+2}=n'-L_j^3$, as asserted. It now clearly follows that $(P_{*},P'_{*})$ is a pair of partitions of $([n],[n'])$ as asserted in the statement of the Lemma. That $(P_{*},P'_{*},B^*)$ is a marked partition pair is clear. It follows from the claim just proved that $(P_{*},P'_{*})$ satisfies the second condition in the statement of the Lemma. This procedure for forming the marked partition pair $(P_{*},P'_{*})$ is illustrated in Figure \ref{f:mapcase3}.
\end{proof}

\begin{proof}(of Proposition (\ref{mapcase3}))
Construct the partitions $(P_{*},P'_{*})$ as given by Lemma \ref{mapcase3lemma}. Consider the generalized mapping $\Upsilon([n],[n'])=(P,P',\tau)$ induced by the marked partition pair $(P_{*},P,B^*)$.
It follows from the construction and fact that $B^*\supseteq \{0,z-1\}$ that $\Upsilon$ is an admissible mapping which of class $H_2^j$ with respect to $B$.
\end{proof}

\section{Tail estimate}\label{s:tailestimate}

The most important of our inductive hypothesis is the following recursive estimate.
\begin{theorem}\label{t:tail}
Assume that the inductive hypothesis holds up to level $j$. Let $X$ and $Y$ be random  $(j+1)$-level blocks according to $\mu^\mathbb{X}_{j+1}$ and $\mu^\mathbb{Y}_{j+1}$. Then
\begin{align*}
\mathbb{P}(S^\mathbb{X}_{j+1}(X)\leq p)\leq p^{m_{j+1}} L_{j+1}^{-\beta},\quad
\mathbb{P}(S^\mathbb{Y}_{j+1}(Y)\leq p)\leq p^{m_{j+1}} L_{j+1}^{-\beta}
\end{align*}
for $p\leq 1-L_{j+1}^{-1}$ and $m_{j+1}=m+2^{-(j+1)}$.
\end{theorem}

There is of course a symmetry between our $X$ and $Y$ bounds and for conciseness all our bounds will be stated in terms of $X$ and $S^\mathbb{X}_{j+1}$ but will similarly hold for $Y$ and $S^\mathbb{Y}_{j+1}$. For the rest of this section we shall drop the superscript $\mathbb{X}$ and denote $S_{j+1}^{\mathbb{X}}$ (resp. $S_j^{\mathbb{X}}$) simply by $S_{j+1}$ (resp. $S_j$).

The block $X$ is constructed from an i.i.d sequence of $j$-level blocks $X_1,X_2,\ldots$ conditioned on the event $X_i\in G_j^{\mathbb{X}}$ for $1\leq i \leq L_j^3$ as described in Section~\ref{s:prelim}.  The construction also involves a random variable $W_X\sim \mathrm{Geom}(L_j^{-4})$ and let $T_X$ denote the number of extra sub-blocks of $X$, that is the length of $X$ is $L_j^{\alpha-1}+2L_j^3+T_X$.  Let $K_X$ denote the number of bad sub-blocks of $X$ and denote their positions by $\ell_1,\ldots,\ell_{K_X}$.  We define $Y_1,\ldots,W_Y,T_Y$ and $K_Y$ similarly and denote the positions of the bad blocks by $\ell_1',\ldots,\ell_{K_Y}'$.
The proof of Theorem~\ref{t:tail} is divided into 5 cases depending on the number of bad sub-blocks, the total number of sub-blocks of $X$ and how ``bad'' the sub-blocks are.

Our inductive bounds allow the following stochastic domination description of $K_X$ and the $S_j(X_{\ell_i})$.

\begin{lemma}\label{l:badBlockStocDomination}
Let $\tk=\tk(t)$ be distributed as a $\mathrm{Bin}(L_{j}^{\alpha-1}+t,L_j^{-\delta})$ and let $\fs=\fs(t)=\sum_{i=1}^{\tk(t)}( 1 + U_i)$ where $U_i$ are i.i.d. rate $m_{j}$ exponentials.  Then,
\[
\left(K_X,-\log \prod_{i=1}^{K_X}S_j(X_{\ell_i})\right)\cdot I(T_X \leq t) \preceq (\tk,\fs)
\]
where $\preceq$ denotes stochastic domination.
\end{lemma}

\begin{proof}
If $V_i$ are i.i.d. Bernoulli with probability $L_j^{-\delta}$, by the inductive assumption and the fact that $\beta>\delta$ we have that for all $i$, $I(X_i\not\in G_j^X) \preceq V_i$ and hence
\[
\left(I(X_i\not\in G_j^X),-I(X_i\not\in G_j^X)\log S(X_i)\right) \preceq (V_i , V_i(1+U_i))
\]
since for $x>1$
\[
\P[-\log S(X_i) > x] \leq L_j^{-\beta} e^{-xm_{j}} < L_j^{-\delta} e^{-(x-1) m_{j}} =\P[ V_i(1+U_i)>x].
\]
Summing over $L_j^3 + 1\leq i \leq L_j^3 + L_j^{\alpha-1}+ t$ completes the result.
\end{proof}

Using Lemma~\ref{l:badBlockStocDomination} we can bound the probability of blocks having large length, number of bad sub-blocks or small  $\prod_{i=1}^{K_X}S_j(X_{\ell_i})$.  This is the key estimate of the paper.

\begin{lemma}\label{l:totalSizeBound}
For all $t',k',x\geq 0$ we have that
\[
\P\left[T_X \geq t', K_X \geq k', -\log \prod_{i=1}^{K_X}S_j(X_{\ell_i}) > x\right] \leq 2 L_j^{-\delta k' /4}\exp\left(-x m_{j+1} - \frac12 t' L_j^{-4} \right).
\]
\end{lemma}
\begin{proof}
If $T_X = t$ and $K_X=k$ then $W_X \geq (t - 2k L_j^3)\vee 0$.  Hence when $K_X=0$
\begin{align}\label{e:kZero}
\P\left[T_X \geq t', K_X =0\right]\leq \P[W_X \geq t']= (1-L_j^{-4})^{t'}\leq \exp[- \frac23 t' L_j^{-4}].
\end{align}
provided $L_j$ is large enough and of course $-\log \prod_{i=1}^{K_X}S_j(X_{\ell_i})=0$.
By Lemma~\ref{l:badBlockStocDomination} we have that
\begin{align}\label{e:TSBdominationBound}
&\P\left[T_X \geq t', K_X \geq k', -\log \prod_{i=1}^{K_X}S_j(X_{\ell_i}) > x\right]\nonumber\\
&\qquad\leq \sum_{k=k'}^\infty \sum_{t=t'}^\infty \P[W_X \geq (t - 2k L_j^3)] \P[\tk(t) = k, \fs(t) > x]\nonumber\\
&\qquad = \sum_{k=k'}^\infty \sum_{t=t'}^\infty \exp[- \frac23 (t - 2k L_j^3) L_j^{-4}] \P[\tk(t) = k, \fs(t) > x].
\end{align}
Since $\tk$ is binomially distributed,
\begin{align}\label{e:tkBinomial}
\P[\tk(t) = k]= { L_j^{\alpha-1} + t \choose k} L_j^{-\delta k} \left(1-L_j^{-\delta}\right)^{L_j^{\alpha-1} + t - k}.
\end{align}
Conditional on $\tk$ we have that $\fs-\tk$ has distribution $\Gamma(k,1/m_{j})$ and so
\begin{align}\label{e:fsGamma}
\P[\fs>x \mid \tk(t) = k]= \int_{(x-k)\vee 0} \frac{m_j^{k-1}}{(k-1)!} y^{k-1} \exp(-y m_j) dy.
\end{align}
Observe that $\frac{m_j^{k-1}}{m_{j+1}(k-1)!} y^{k-1} \exp(-y 2^{-(j+1)})$ is proportional to the density of a $\Gamma(k,2^{j+1})$ which is maximized at $2^{j+1}(k-1)$.  Hence
\begin{align}\label{e:MaxGammaDensity}
\max_{y\geq 0} \frac{m_j^{k-1}}{m_{j+1}(k-1)!} y^{k-1} \exp(-y 2^{-(j+1)})
&\leq \frac{m_j^{k-1}}{m_{j+1}(k-1)!} (2^{j+1}(k-1))^{k-1} \exp(-(k-1)) \nonumber\\
&\leq (2^{j+1}m_j)^{k},
\end{align}
since by Stirling's approximation $\frac{(k-1))^{k-1}}{(k-1)! \exp(k-1)}<1$.  Since $m_{j+1} = m_j-2^{-(j+1)}$, substituting~\eqref{e:MaxGammaDensity} into~\eqref{e:fsGamma} we get that
\begin{align}\label{e:fsGamma2}
\P[\fs>x \mid \tk(t) = k]&\leq (2^{j+1}m_j)^{k} \int_{(x-k)\vee 0} m_{j+1}  \exp(-y m_{j+1}) dy\nonumber\\
&= (m_j 2^{j+1}e^{m_{j+1}})^{k} \exp(-x m_{j+1}).
\end{align}
Combining~\eqref{e:tkBinomial} and~\eqref{e:fsGamma2} we get that
\begin{align}\label{e:TKfsBoundA}
&\P[\tk(t) = k, \fs(t) > x]\nonumber\\
&\qquad\leq { L_j^{\alpha-1} + t \choose k} L_j^{-\delta k} \left(1-L_j^{-\delta}\right)^{L_j^{\alpha-1} + t - k} (m_j 2^{j+1}e^{m_{j+1}})^{k} \exp(-x m_{j+1})\nonumber\\
&\qquad\leq \frac{\left(1-L_j^{-\delta}\right)^{L_j^{\alpha-1} + t - k}}{\left(1-L_j^{-\delta/2}\right)^{L_j^{\alpha-1} + t - k}} (L_j^{-\delta/2} m_j 2^{j+1}e^{m_{j+1}})^{k} \exp(-x m_{j+1})
\end{align}
since
\[{ L_j^{\alpha-1} + t \choose k} L_j^{-\delta/2 k} \left(1-L_j^{-\delta/2}\right)^{L_j^{\alpha-1} + t - k}=\P[\mathrm{Bin}(L_j^{\alpha-1} + t,L_j^{-\delta/2 k})=k] < 1.
\]
Now for large enough $L_0$,
\begin{align}\label{e:TKfsBoundB}
\frac{\left(1-L_j^{-\delta}\right)^{L_j^{\alpha-1} + t - k}}{\left(1-L_j^{-\delta/2}\right)^{L_j^{\alpha-1} + t - k}} \leq \exp(2 (L_j^{\alpha-1} + t)L_j^{-\delta/2}) \leq 2 \exp(2 t L_j^{-\delta/2}),
\end{align}
since $\delta/2>\alpha$.  As $L_j=L_0^{\alpha^j}$, for large enough $L_0$ we have that $L_j^{-\delta/2} m_j 2^{j+1}e^{m_{j+1}} \leq \frac1{10} L_j^{-\delta/3}$ and so combining~\eqref{e:TKfsBoundA} and~\eqref{e:TKfsBoundB} we have that
\begin{align}\label{e:TKfsBoundC}
&\P[\tk(t) = k, \fs(t) > x] \leq  \frac{2}{10^k} \exp(2 t L_j^{-\delta/2})   L_j^{-\delta k/3} \exp(-x m_{j+1}).
\end{align}
Finally substituting this into~\eqref{e:TSBdominationBound} we get that
\begin{align}
&\P\left[T_X \geq t', K_X \geq k', -\log \prod_{i=1}^{K_X}S_j(X_{\ell_i}) > x\right]\nonumber\\
&\qquad = \sum_{k=k'}^\infty \sum_{t=t'}^\infty \frac{2}{10^k} \exp[- \frac23 (t - 2k L_j^3) L_j^{-4}+2 t L_j^{-\delta/2}] L_j^{-\delta k/3} \exp(-x m_{j+1})\nonumber\\
&\qquad =  L_j^{4} \exp[- \frac12 t' L_j^{-4}] L_j^{-\delta k'/3} \exp(-x m_{j+1}).
\end{align}
for large enough $L_0$ since $\delta/2>4$. If $k'\geq 1$ then since $\delta/3 - \delta/4 > 4$, we get that
\[
\P\left[T_X \geq t', K_X \geq k', -\log \prod_{i=1}^{K_X}S_j(X_{\ell_i}) > x\right] \leq L_j^{-\delta k'/4} \exp(-x m_{j+1}- \frac12 t' L_j^{-4})
\]
which together with~\eqref{e:kZero} completes the result.
\end{proof}

We now move to our five cases.  In each one we will use a different mapping (or mappings) to get good lower bounds on the probability that $X \hookrightarrow Y$ given $X$.

\subsection{Case 1}

The first case is the generic situation where the blocks are of typical length, have few bad sub-blocks whose embedding probabilities are not too small.  This case holds with high probability.
We define the event $\A{1}_{X,j+1}$ to be the set of $(j+1)$ level blocks such that
\[
\A{1}_{X,j+1} := \left\{X:T_X \leq \frac{R L^{\alpha-1}_j}{2}, K_X\leq k_0,  \prod_{i=1}^{K_X}S_j(X_{\ell_i}) > L_j^{-1/3} \right\}.
\]

\begin{lemma}\label{l:A1Size}
The probability that $X\in \A{1}_{X,j+1}$ is bounded by
\[
\P[X\not\in \A{1}_{X,j+1}] \leq L_{j+1}^{-3\beta}.
\]
\end{lemma}
\begin{proof}

By Lemma~\ref{l:totalSizeBound}
\begin{equation}\label{e:A1SizeA}
\P\left[T_X > \frac{R L^{\alpha-1}_j}{2}\right] \leq 2\exp\left(- \frac{R L^{\alpha-5}_j}{4}\right) \leq \frac13 L_{j+1}^{-3\beta}.
\end{equation}
since $\alpha>5$ and $L_0$ is large.  Again by Lemma~\ref{l:totalSizeBound}
\begin{align}\label{e:A1SizeB}
\P[K_X > k_0] &\leq 2 L_j^{-\delta k_0 /4} =2L_{j+1}^{-\delta k_0 /(4\alpha)} \leq \frac13 L_{j+1}^{-3\beta},
\end{align}
since $k_0>36 \alpha\beta$.
Finally again by Lemma~\ref{l:totalSizeBound},
\begin{align}\label{e:A1SizeC}
&\P[\prod_{i=1}^{K_X}S_j(X_{\ell_i}) \leq  L_j^{-1/3}]\leq 2 L_{j}^{- m_{j+1}/3 } \leq \frac13 L_{j+1}^{-3\beta},
\end{align}
since $\frac13 m_{j+1}> \frac13 m > 3\alpha\beta$.  Combining~\eqref{e:A1SizeA}, \eqref{e:A1SizeB} and~\eqref{e:A1SizeC} completes the result.
\end{proof}

\begin{lemma}\label{l:A1Map}
We have that for all $X\in\A{1}_{X,j+1}$,
\begin{equation}\label{e:A1MapE1}
\P[X \hookrightarrow Y\mid Y\in \A{1}_{Y,j+1}, X] \geq \frac12,
\end{equation}
and that
\begin{equation}\label{e:A1MapE2}
\P[X \hookrightarrow Y\mid X \in\A{1}_{X,j+1} , Y\in \A{1}_{Y,j+1}] \geq 1-  L_{j+1}^{-3\beta}.
\end{equation}
\end{lemma}

\begin{proof}
We first prove equation~\eqref{e:A1MapE2} where we do not condition on $X$.
Suppose that $X \in\A{1}_{X,j+1} , Y\in \A{1}_{Y,j+1}$.  Let us condition on the block lengths $T_X,T_Y$, the number of bad sub-blocks, $K_X,K_Y$, their locations, $\ell_1,\ldots,\ell_{K_X}$ and $\ell_1',\ldots,\ell_{K_Y}'$ and the bad-sub-blocks themselves.  Denote this conditioning by
\begin{align*}
\mathcal{F}=\{X \in\A{1}_{X,j+1} , Y\in \A{1}_{Y,j+1},T_X,T_Y,K_X,K_Y,\ell_1,\ldots,\ell_{K_X}, \ell_1',\ldots,\ell_{K_Y}',\\
X_{\ell_1},\ldots,X_{\ell_{K_X}}, Y_{\ell_1'},\ldots,Y_{\ell_{K_Y}'}\}.
\end{align*}
By Proposition~\ref{compresspartition} we can find $L_j^2$ admissible generalized mappings $\Upsilon_h([L_j^{\alpha -1}+2L_j^3+T_X], L_j^{\alpha-1}+2L_j^3+T_Y)$ with associated $\tau_h$ for $1\leq h \leq L_j^2$ which are of class $G^j$ with respect to $(B=\{\ell_1<\ldots<\ell_{K_X}\}, B'=\{\ell'_1<\ldots<\ell'_{K_Y}\}$.  By construction the we have that $\tau_h(\ell_i)=\tau_1(\ell_i)+h-1$ and in particular each position $\ell_i$ is mapped to $L_j^2$ distinct sub-blocks by the map, none of which is equal to one of the $\ell_{i'}'$.  Similarly for the $\tau_h^{-1}$.  Hence we can construct a subset $\mathcal{H}\subset [L_j^2]$ with $|\mathcal{H}|= L_j<\lfloor L_j^2/ 2 k_0\rfloor$ so that for all $i_1\neq i_2$ and $h_1,h_2\in\mathcal{H}$ we have that $\tau_{h_1}(\ell_{i_1})\neq \tau_{h_2}(\ell_{i_2})$ and $\tau_{h_1}^{-1}(\ell_{i_1}')\neq \tau_{h_2}^{-1}(\ell_{i_2}')$, that is that all the positions bad blocks are mapped to are distinct.

By construction all the $Y_{\tau_{h}(\ell_{i})}$ are uniformly chosen good $j$-blocks conditional on $\cf$ and since $S(X_{\ell_{i}})\geq L_j^{-1/3}$ we have that
\begin{equation}\label{e:A1MapA}
\P[X_{\ell_{i}} \hookrightarrow Y_{\tau_{h}(\ell_{i})}\mid \cf] \geq S(X_{\ell_{i}}) - \P[Y_{\tau_{h}(\ell_{i})}\not\in G_j^{\mathbb{X}}] \geq \frac12 S(X_{\ell_{i}}).
\end{equation}
Let $\cd_h$ denote the event
\[
\cd_h=\left\{X_{\ell_{i}} \hookrightarrow Y_{\tau_{h}(\ell_{i})}\hbox{ for } 1\leq i \leq K_X,
X_{\tau_h^{-1}(\ell_{i}')} \hookrightarrow Y_{\ell_{i}'}\hbox{ for } 1\leq i \leq K_Y \right\}.
\]
By Proposition~\ref{admissiblemap} if one of the $\cd_h$ hold then $X\mapsto Y$.  Conditional on $\cf$ for $h\in\ch$ the $\cd_h$ are independent and by~\eqref{e:A1MapA},
\begin{equation}\label{e:A1MapB}
\P[\cd_h \mid \cf] \geq \prod_{i=1}^{K_X} \frac12 S_j(X_{\ell_{i}})\prod_{i=1}^{K_Y} \frac12 S_j(Y_{\ell_{i}'}) \geq 2^{-2k_0} L_j^{-2/3}.
\end{equation}
Hence
\begin{equation}\label{e:A1MapC}
\P[X  \hookrightarrow Y \mid \cf] \geq \P[\cup_{h\in\ch}\cd_h \mid \cf]  \geq 1-\left(1-2^{-2k_0} L^{-2/3}\right)^{L_j}\geq  1-  L_{j+1}^{-3\beta}.
\end{equation}
Now removing the conditioning we get equation~\eqref{e:A1MapE2}.  To prove equation~\eqref{e:A1MapE1} we proceed in the same way but note that since it involves conditioning on the good sub-blocks the events $X_{\tau_h'(\ell_{i}')} \hookrightarrow Y_{\ell_{i}'}$ are no longer conditionally independent. So we will condition on $Y$ having no bad blocks so
\begin{align*}
\mathcal{F}(X)=\{X \in\A{1}_{X,j+1} , Y\in \A{1}_{Y,j+1}, X, T_X,T_Y,K_X,K_Y = 0,\ell_1,\ldots,\ell_{K_X},
X_{\ell_1},\ldots,X_{\ell_{K_X}}\}.
\end{align*}
By the above argument then
\begin{equation}
\P[X  \hookrightarrow Y \mid \cf(X)] \geq \P[\cup_{h\in\ch}\cd_h \mid \cf(X)]  \geq 1-\left(1-2^{-k_0} L_j^{-1/3}\right)^{L_j}\geq  1-  L_{j+1}^{-3\beta}.
\end{equation}
Hence
\begin{align*}
\P[X \hookrightarrow Y\mid Y\in \A{1}_{Y,j+1}, X] &\geq \P[X  \hookrightarrow Y \mid \cf(X)]\cdot \P[K_X=0\mid \cf(X)]\\
&\geq \left( 1-  L_{j+1}^{-3\beta}\right)\cdot (1-L_j^{-\delta})^{L_j^{\alpha-1}+T_Y}\geq \frac12,
\end{align*}
since $L_j^\delta> 10 L_j(L_j^{\alpha-1}+T_Y)$ completing the lemma.
\end{proof}

\begin{lemma}\label{l:A1final}
When $\frac12\leq p \leq 1-L_{j+1}^{-1}$
\[
\P(S_{j+1}(X)\leq p)\leq p^{m_{j+1}} L_{j+1}^{-\beta}
\]
\end{lemma}

\begin{proof}
By Lemma~\ref{l:A1Size} and~\ref{l:A1Map} we have that
\begin{align*}
\P(\P[X \not\hookrightarrow Y\mid X]\geq L_{j+1}^{-1}) &\leq \P[X \not\hookrightarrow Y]L_{j+1}\\
&\leq \Big( \P[X \not\hookrightarrow Y\mid X \in\A{1}_{X,j+1} , Y\in \A{1}_{Y,j+1}]\\
& \qquad + \P[ X \not\in\A{1}_{X,j+1}] + \P[Y \not\in \A{1}_{Y,j+1}] \Big)L_{j+1}\\
&\leq 3L_{j+1}^{1-3\beta}\leq 2^{-m_{j+1}}L_{j+1}^{-\beta}
\end{align*}
where the first inequality is by Markov's inequality.  This implies the corollary.
\end{proof}

\subsection{Case 2}

The next case involves blocks which are not too long and do not contain too many bad sub-blocks but whose bad sub-blocks may have very small embedding probabilities.
We define the class of blocks $\A{2}_{X,j+1}$  as
\[
\A{2}_{X,j+1} := \left\{X:T_X \leq \frac{R L^{\alpha-1}_j}{2}, K_X\leq k_0,  \prod_{i=1}^{K_X}S_j(X_{\ell_i}) \leq L_j^{-1/3} \right\}.
\]

\begin{lemma}\label{l:A2Map}
For $X\in \A{2}_{X,j+1}$,
\[
S_{j+1}(X) \geq \min\left\{\frac12, \frac1{10} L_j \prod_{i=1}^{K_X}S_j(X_{\ell_i}) \right\}
\]
\end{lemma}
\begin{proof}
Suppose that $X \in\A{2}_{X,j+1}$.  Let $\ce$ denote the event
\[
\ce=\{W_Y\leq  L^{\alpha-1}_j, T_Y=W_Y \}.
\]
Then by definition of $W_Y$, $\P[W_Y\leq L^{\alpha-1}_j] \geq 1 - (1-L_j^{-4})^{L_j^{\alpha-1}}\geq 9/10$ while by the definition of the block boundaries the event $T_Y=W_Y$ is equivalent to their being no bad sub-blocks amongst $Y_{L_j^3+L_j^{\alpha-1}+W_Y+1},\ldots,Y_{L_j^3+L_j^{\alpha-1}+W_Y+2L_j^3}$, that is that we don't need to extend the block because of bad sub-blocks. Hence $\P[T_Y=W_Y] \geq (1-L_j^{-\delta})^{2L_j^3} \geq 9/10$.  Combining these we have that
\begin{equation}\label{e:A2MapCEbound}
\P[\ce]\geq 8/10.
\end{equation}
On the event $T_Y=W_Y$ we have that the blocks $Y_{L_j^3+1},\ldots,Y_{L_j^3+L_j^{\alpha-1}+T_Y}$ are uniform $j$-blocks since the block division did not evaluate whether they are good or bad.

Similarly to Lemma~\ref{l:A1Map}, by Proposition~\ref{mapcase2} we can find $L_j^2$  admissible generalized mappings $\Upsilon_h([L_j^{\alpha-1}+2L_j^3+T_X],[L_j^{\alpha -1}+2L_j^{3}+T_Y])$ for $1\leq h \leq L_j^2$ with associated $\tau_h$ which are of class $H_1^j$ with respect to $B=\{\ell_1<\ldots<\ell_{K_X}\}$.  For all $h$ and $i$, $L_j^3+1\leq \tau_h(\ell_i) \leq L_j^3+L_j^{\alpha-1}+T_Y$.  As in Lemma~\ref{l:A1Map} we can a subset $\mathcal{H}\subset [L_j^2]$ with $|\mathcal{H}|= L_j<\lfloor L_j^2/ 2 k_0\rfloor$ so that for all $i_1\neq i_2$ and $h_1,h_2\in\mathcal{H}$ we have that $\tau_{h_1}(\ell_{i_1})\neq \tau_{h_2}(\ell_{i_2})$, that is that all the positions bad blocks are mapped to are distinct.  We will estimate the probability that one of these maps work.

In trying out these $h$ different mappings there is a subtle conditioning issue since a map failing may imply that $Y_{\tau_{h}}$ is not good.  As such we condition on an  event $\cd_h \cup \cg_h $ which holds with high probability.
Let $\cd_h$ denote the event
\[
\cd_h=\left\{X_{\ell_{i}} \hookrightarrow Y_{\tau_{h}(\ell_{i})}\hbox{ for } 1\leq i \leq K_X\right\}.
\]
and let
\[
\cg_h=\left\{Y_{\tau_h(\ell_{i})} \in G_j^{\mathbb{Y}} \hbox{ for } 1\leq i \leq K_X\right\}.
\]
Then
\[
\P[\cd_h \cup \cg_h \mid X,\ce] \geq \P[ \cg_h \mid X,\ce] \geq (1-L_j^{-\delta})^{k_0}.
\]
and since they are conditionally independent given $X$ and $\ce$,
\begin{equation}\label{e:A2MapA}
\P[\cap_{h\in\ch}(\cd_h \cup \cg_h) \mid X,\ce] \geq  (1-L_j^{-\delta})^{k_0 L_j}\geq 9/10.
\end{equation}
Now
\[
\P[\cd_h\mid X,\ce,( \cd_h \cup \cg_h)] \geq \P[\cd_h\mid X,\ce] = \prod_{i=1}^{K_X} S_j(X_{\ell_i})
\]
and hence
\begin{align}\label{e:A2MapB}
\P[\cup_{h\in\ch} \cd_h \mid X,\ce, \cap_{h\in\ch}(\cd_h \cup \cg_h)] &\geq 1-\left(1- \prod_{i=1}^{K_X} S_j(X_{\ell_i})\right)^{L_j}\nonumber\\
&\geq \frac{9}{10}  \wedge \frac14 L_j \prod_{i=1}^{K_X} S_j(X_{\ell_i})
\end{align}
since $1-e^{-x}\leq (1-x/4)\vee 1/10$ for $x\geq 0$.
Furthermore, if
\[
\cm=\left\{\exists h_1\neq h_2 \in\ch: \cd_{h_1}\setminus \cg_{h_1}, \cd_{h_2}\setminus \cg_{h_2} \right\},
\]
then
\begin{align}\label{e:A2MapC}
\P[\cm \mid X,\ce, \cap_{h\in\ch}(\cd_h \cup \cg_h)]
&\leq {L_j \choose 2} \P[\cd_{h}\setminus \cg_{h}\mid X,\ce, \cap_{h\in\ch}(\cd_h \cup \cg_h)]^2\nonumber\\
&\leq {L_j \choose 2} \left(\prod_{i=1}^{K_X} S_j(X_{\ell_i}) \wedge L_j^{-\delta}  \right)^2\nonumber\\
&\leq L_j^{-(\delta-2)} \prod_{i=1}^{K_X} S_j(X_{\ell_i}).
\end{align}
Finally let $\cj$ denote the event
\[
\cj=\left\{Y_{k} \in G_j^{\mathbb{Y}} \hbox{ for all } k\in \{L_j^3+1,\ldots,L_j^3+L_j^{\alpha-1}+T_X\}\setminus \cup_{h\in\ch, 1\leq i \leq K_X}\{\tau_h(\ell_i)\}\right\}.
\]
Then
\begin{equation}\label{e:A2MapD}
\P[\cj \mid X,\ce] \geq \left(1- L_j^{-\delta}\right)^{2L_j^{\alpha-1}} \geq 9/10.
\end{equation}

If $\cj,\cup_{h\in\ch} \cd_h$ and $\cap_{h\in\ch}(\cd_h \cup \cg_h)$ all hold and $\cm$ does not hold then we can find at least one $h\in\ch$ such that $\cd_h$ holds and $\cg_{h'}$ holds for all $h'\in\ch\setminus\{h\}$ then by Proposition~\ref{admissiblemap2} we have that $X\hookrightarrow Y$. Hence by~\eqref{e:A2MapA}, \eqref{e:A2MapB}, \eqref{e:A2MapC}, and~\eqref{e:A2MapD} and the fact that $\cj$ is conditionally independent of the other events that
\begin{align*}
\P[X\hookrightarrow Y\mid X,\ce]
&\geq \P[\cup_{h\in\ch} \cd_h, \cap_{h\in\ch}(\cd_h \cup \cg_h), \cj, \ \neg \cm \mid X,\ce] \nonumber\\
&= \P[\cj \mid X,\ce]\P[\cup_{h\in\ch} \cd_h, \ \neg \cm \mid X,\ce,\cap_{h\in\ch}(\cd_h \cup \cg_h)]
\P[\cap_{h\in\ch}(\cd_h \cup \cg_h)\mid X,\ce]\\
&\geq \frac{81}{100}\left[ \frac{9}{10}  \wedge \frac14 L_j \prod_{i=1}^{K_X} S_j(X_{\ell_i}) - L_j^{-(\delta-2)} \prod_{i=1}^{K_X} S_j(X_{\ell_i}) \right]\\
&\geq \frac{7}{10} \wedge \frac15 L_j \prod_{i=1}^{K_X} S_j(X_{\ell_i}).
\end{align*}
Combining with~\eqref{e:A2MapCEbound} we have that
\begin{align*}
\P[X\hookrightarrow Y\mid X]
&\geq \frac{1}{2} \wedge \frac1{10} L_j \prod_{i=1}^{K_X} S_j(X_{\ell_i}),
\end{align*}
which completes the proof.
\end{proof}

\begin{lemma}\label{l:A2Bound}
When $0<p< \frac12$,
\[
\mathbb{P}(X\in \A{2}_{X,j+1}, S_{j+1}(X)\leq p)\leq \frac15 p^{m_{j+1}} L_{j+1}^{-\beta}
\]
\end{lemma}

\begin{proof}
We have that
\begin{align}
\mathbb{P}(X\in \A{2}_{X,j+1}, S_{j+1}(X)\leq p) &\leq \P\left[\frac1{10} L_j \prod_{i=1}^{K_X}S_j(X_{\ell_i}) \leq p\right]\nonumber\\
&\leq  2\left(\frac{10 p}{L_j}\right)^{m_{j+1}} \leq \frac15 p^{m_{j+1}} L_{j+1}^{-\beta}
\end{align}
where the first inequality holds  by Lemma~\ref{l:A2Map}, the second by Lemma~\ref{l:totalSizeBound} and the third holds for large enough $L_0$  since $m_{j+1}>m>\alpha\beta$.
\end{proof}

\subsection{Case 3}

The third case allows for a greater number of bad sub-blocks.
The class of blocks $\A{3}_{X,j+1}$ defined as
\[
\A{3}_{X,j+1} := \left\{X:T_X \leq \frac{R L^{\alpha-1}_j}{2}, k_0\leq K_X\leq \frac{L_j^{\alpha-1}+T_X}{10 R_j^+} \right\}.
\]

\begin{lemma}\label{l:A3Map}
For $X\in \A{3}_{X,j+1}$,
\[
S_{j+1}(X) \geq \frac12 \prod_{i=1}^{K_X}S_j(X_{\ell_i})
\]
\end{lemma}

\begin{proof}
The proof is a simpler version of Lemma~\ref{l:A2Map} where this time we only need consider a single map $\Upsilon$.  Suppose that $X \in\A{3}_{X,j+1}$.  Again let $\ce$ denote the event
\[
\ce=\{W_Y\leq  L^{\alpha-1}_j, T_Y=W_Y \}.
\]
Similarly to~\eqref{e:A2MapCEbound} we have that,
\begin{equation}\label{e:A3MapCEbound}
\P[\ce]\geq 8/10.
\end{equation}
On the event $T_Y=W_Y$ we have that the blocks $Y_{L_j^3+1},\ldots,Y_{L_j^3+L_j^{\alpha-1}+T_Y}$ are uniform $j$-blocks since the block division did not evaluate whether they are good or bad.

By Proposition~\ref{mapcase3} we can find an admissible generalized mapping $\Upsilon([L_j^{\alpha-1}+2L_j^3+T_X],[L_j^{\alpha-1}+2L_j^3+T_X])$ with associated $\tau$ which are of class $H^j_2$ with  $B=\{\ell_1<\ldots<\ell_{K_X}\}$ so that for all $i$, $L_j^3+1\leq \tau_h(\ell_i) \leq L_j^3+L_j^{\alpha-1}+T_Y$.  We estimate the probability that this gives an embedding.

Let $\cd$ denote the event
\[
\cd=\left\{X_{\ell_{i}} \hookrightarrow Y_{\tau(\ell_{i})}\hbox{ for } 1\leq i \leq K_X\right\}.
\]
By definition,
\begin{align}\label{e:A3MapA}
\P[\cd\mid X,\ce]  = \prod_{i=1}^{K_X} S_j(X_{\ell_i})
\end{align}
Let $\cj$ denote the event
\[
\cj=\left\{Y_{k} \in G_j^{\mathbb{Y}} \hbox{ for all } k\in \{L_j^3+1,\ldots,L_j^3+L_j^{\alpha-1}+T_X\}\setminus \cup_{ 1\leq i \leq K_X}\{\tau(\ell_i)\}\right\}.
\]
Then for large enough $L_j$,
\begin{equation}\label{e:A3MapB}
\P[\cj \mid X,\ce] \geq \left(1- L_j^{-\delta}\right)^{2L_j^{\alpha-1}} \geq 9/10.
\end{equation}

If $\cd$ and $\cj$ hold then by Proposition~\ref{admissiblemap2} we have that $X\hookrightarrow Y$. Hence by~\eqref{e:A3MapA} and \eqref{e:A3MapB} and the fact that $\cd$ and $\cj$ are conditionally independent we have that,
\begin{align*}
\P[X\hookrightarrow Y\mid X,\ce]
&\geq \P[\cd, \cj \mid X,\ce] \nonumber\\
&= \P[\cd \mid X,\ce] \P[\cj \mid X,\ce] \nonumber\\
&\geq \frac{9}{10}  \prod_{i=1}^{K_X} S_j(X_{\ell_i}).
\end{align*}
Combining with~\eqref{e:A3MapCEbound} we have that
\begin{align*}
\P[X\hookrightarrow Y\mid X]
&\geq \frac{1}{2} \prod_{i=1}^{K_X} S_j(X_{\ell_i}),
\end{align*}
which completes the proof.
\end{proof}

\begin{lemma}\label{l:A3Bound}
When $0<p\leq \frac12$,
\[
\mathbb{P}(X\in \A{3}_{X,j+1}, S_{j+1}(X)\leq p)\leq \frac15 p^{m_{j+1}} L_{j+1}^{-\beta}
\]
\end{lemma}

\begin{proof}
We have that
\begin{align}
\mathbb{P}(X\in \A{3}_{X,j+1}, S_{j+1}(X)\leq p) &\leq \P\left[K_X>k_0, \frac12 \prod_{i=1}^{K_X}S_j(X_{\ell_i}) \leq p\right]\nonumber\\
&\leq  2\left(2p\right)^{m_{j+1}} L_j^{-\delta k_0/4}\leq \frac15 p^{m_{j+1}} L_{j+1}^{-\beta}
\end{align}
where the first inequality holds  by Lemma~\ref{l:A3Map}, the second by Lemma~\ref{l:totalSizeBound} and the third holds for large enough $L_0$  since $\delta k_0 >4\alpha\beta$.
\end{proof}

\subsection{Case 4}

Case 4 is the case of blocks of long length but not too many bad sub-blocks (at least with a density of them smaller than $(10 R_j^+)^{-1}$.
The class of blocks $\A{4}_{X,j+1}$ defined as
\[
\A{4}_{X,j+1} := \left\{X:T_X > \frac{R L^{\alpha-1}_j}{2}, K_X\leq \frac{L_j^{\alpha-1}+T_X}{10 R_j^+} \right\}.
\]

\begin{lemma}\label{l:A4Map}
For $X\in \A{4}_{X,j+1}$,
\[
S_{j+1}(X) \geq \prod_{i=1}^{K_X}S_j(X_{\ell_i}) \exp(-3T_X L_j^{-4} /R )
\]
\end{lemma}

\begin{proof}
The proof is  a modification of Lemma~\ref{l:A3Map} allowing the length of $Y$ to grow at a slower rate than $X$.  Suppose that $X \in\A{4}_{X,j+1}$ and let $\ce(X)$ denote the event
\[
\ce(X)=\{W_Y = \lfloor 2 T_X / R\rfloor, T_Y=W_Y \}.
\]
Then by definition $\P[W_Y = \lfloor 2 T_X / R\rfloor] =L_j^{-4} (1-L_j^{-4})^{\lfloor 2 T_X / R\rfloor}$.  Similarly to Lemma~\ref{l:A2Map}, $\P[T_Y=W_Y\mid W_Y] \geq (1-L_j^{-\delta})^{2L_j^3} \geq 9/10$.  Combining these we have that
\begin{equation}\label{e:A4MapCEbound}
\P[\ce(X)]\geq \frac9{10} L_j^{-4} (1-L_j^{-4})^{\lfloor 2 T_X / R\rfloor}.
\end{equation}

By Proposition~\ref{mapcase3} we can find an admissible generalized mapping $\Upsilon([L_j^{\alpha -1}+2L_j^3+T_X],[L_j^{\alpha-1}+2L_j^{3}+T_Y])$   with associated $\tau$ which is of class $H^j_2$ with respect to $B=\{\ell_1<\ldots<\ell_{K_X}\}$ so that for all $i$, $L_j^3+1\leq \tau_h(\ell_i) \leq L_j^3+L_j^{\alpha-1}+T_Y$.  We again estimate the probability that this gives an embedding.

Defining  $\cd$ and $\cj$ as in Lemma~\ref{l:A3Map} and we again have that~\eqref{e:A3MapA} holds.
Then for large enough $L_j$,
\begin{equation}\label{e:A4MapB}
\P[\cj \mid X,\ce(X)] \geq \left(1- L_j^{-\delta}\right)^{L_j^{\alpha-1} + \lfloor 2 T_X / R\rfloor + 2L_j^3} \geq \frac12\exp\left(- L_j^{-\delta}(L_j^{\alpha-1} + \lfloor 2 T_X / R\rfloor + 2L_j^3)\right).
\end{equation}

If $\cd$ and $\cj$ hold then by Proposition~\ref{admissiblemap2} we have that $X\hookrightarrow Y$. Hence by~\eqref{e:A3MapA} and \eqref{e:A4MapB} and the fact that $\cd$ and $\cj$ are conditionally independent we have that,
\begin{align*}
\P[X\hookrightarrow Y\mid X,\ce]
&\geq \P[\cd \mid X,\ce] \P[\cj \mid X,\ce] \nonumber\\
&\geq \frac12\exp\left(- L_j^{-\delta}(L_j^{\alpha-1} + \lfloor 2 T_X / R\rfloor + 2L_j^3)\right) \prod_{i=1}^{K_X} S(X_{\ell_i}).
\end{align*}
Combining with~\eqref{e:A3MapCEbound} we have that
\begin{align*}
\P[X\hookrightarrow Y\mid X]
&\geq  \exp(-3T_X L_j^{-4} /R ) \prod_{i=1}^{K_X} S(X_{\ell_i}),
\end{align*}
since $T_X L_j^{-4} = \Omega(L^{\alpha-5})$ which completes the proof.
\end{proof}

\begin{lemma}\label{l:A4Bound}
When $0<p\leq \frac12$,
\[
\mathbb{P}(X\in \A{4}_{X,j+1}, S_{j+1}(X)\leq p)\leq \frac15 p^{m_{j+1}} L_{j+1}^{-\beta}
\]
\end{lemma}

\begin{proof}
We have that
\begin{align}
\mathbb{P}(X\in \A{4}_{X,j+1}, S_{j+1}(X)\leq p) &\leq \sum_{t=\frac{R L^{\alpha-1}_j}{2}+1}^\infty \P\left[T_x=t,  \prod_{i=1}^{K_X}S_j(X_{\ell_i}) \exp(-3 t L_j^{-4} /R ) \leq p\right]\nonumber\\
&\leq \sum_{t=\frac{R L^{\alpha-1}_j}{2}+1}^\infty 2\left(p \exp(3 t L_j^{-4} /R )\right)^{m_{j+1}} \exp\left(- \frac12 t L_j^{-4}\right)\nonumber\\
&\leq \frac15 p^{m_{j+1}} L_{j+1}^{-\beta}
\end{align}
where the first inequality holds  by Lemma~\ref{l:A4Map}, the second by Lemma~\ref{l:totalSizeBound} and the third holds for large enough $L_0$  since $3m_{j+1}/R<\frac12$ and so for large enough $L_0$,
\[
\sum_{t=R L^{\alpha-1}_j/2+1}^\infty \exp\left(- t L_j^{-4} \left(\frac12-\frac{3m_{j+1}}{R} \right)\right) < \frac1{10} L_{j+1}^{-\beta}.
\]
\end{proof}

\subsection{Case 5}

The final case involves blocks with a large density of sub-blocks.
The class of blocks $\A{5}_{X,j+1}$ defined as
\[
\A{5}_{X,j+1} := \left\{X: K_X > \frac{L_j^{\alpha-1}+T_X}{10 R_j^+} \right\}.
\]

\begin{lemma}\label{l:A5Map}
For $X\in \A{5}_{X,j+1}$,
\[
S_{j+1}(X) \geq \exp(-2 T_X L_j^{-4}) \prod_{i=1}^{K_X} S_j(X_{\ell_i})
\]
\end{lemma}
\begin{proof}
The proof follows by minor modifications of Lemma~\ref{l:A4Map}.  We take $\ce(X)$ to denote the event
\[
\ce(X)=\{W_Y = T_X, T_Y=W_Y \}.
\]
and get a bound of
\begin{equation}\label{e:A5MapCEbound}
\P[\ce(X)]\geq \frac9{10} L_j^{-4} (1-L_j^{-4})^{T_X}.
\end{equation}
We take as our mapping the complete partitions $\{0\leq 1\leq 2\leq\ldots\leq 2L_j^3+L_j^{\alpha-1}+T_X\}$ and $\{0\leq 1\leq 2\leq\ldots\leq 2L_j^3+L_j^{\alpha-1}+T_Y\}$ and so are simply mapping sub-blocks to sub-blocks.
The new bound for $\cj$ becomes
\begin{equation}\label{e:A5MapB}
\P[\cj \mid X,\ce(X)] \geq \left(1- L_j^{-\delta}\right)^{L_j^{\alpha-1} + T_X + 2L_j^3} \geq \frac12\exp\left(- L_j^{-\delta}(L_j^{\alpha-1} + T_X + 2L_j^3)\right).
\end{equation}
Proceeding as in Lemma~\ref{l:A4Map} the yields the result.
\end{proof}

\begin{lemma}\label{l:A5Bound}
When $0<p\leq \frac12$,
\[
\mathbb{P}(X\in \A{5}_{X,j+1}, S_{j+1}(X)\leq p)\leq \frac15 p^{m_{j+1}} L_{j+1}^{-\beta}
\]
\end{lemma}

\begin{proof}
First note that since $\alpha>9$, that $$L_j^{-\frac{\delta}{40 R_j^+}}=L_0^{-\frac{\delta \alpha^j}{40 R_j^+}}\to 0$$ as $j\to\infty$.  Hence for large enough $L_0$,
\begin{equation}\label{e:A5BoundA}
\sum_{t=0}^\infty \left(\exp(2 m_{j+1} L_j^{-4}) L_j^{-\frac{\delta}{40 R_j^+}} \right)^t < 2.
\end{equation}
We have that
\begin{align}
&\mathbb{P}(X\in \A{5}_{X,j+1}, S_{j+1}(X)\leq p) \nonumber\\
&\qquad\leq \sum_{t=0}^\infty \P\left[T_x=t, K_X> \frac{L_j^{\alpha-1}+t}{10 R_j^+}, \prod_{i=1}^{K_X}S_j(X_{\ell_i}) \exp(-2 t L_j^{-4}) \leq p\right]\nonumber\\
&\leq p^{m_{j+1}} \sum_{t=0}^\infty 2\left( \exp(2 m_{j+1} t L_j^{-4})\right) L_j^{-\frac{\delta(L_j^{\alpha-1}+t)}{40 R_j^+}}\nonumber\\
&\leq \frac15 p^{m_{j+1}} L_{j+1}^{-\beta}
\end{align}
where the first inequality holds be by Lemma~\ref{l:A5Map}, the second by Lemma~\ref{l:totalSizeBound} and the third follows by~\eqref{e:A5BoundA} and the fact that
\[
L_j^{-\frac{\delta L_j^{\alpha-1}}{40 R_j^+}} \leq \frac1{20} L_{j+1}^{-\beta},
\]
for large enough $L_0$.
\end{proof}

\subsection{Proof of Theorem~\ref{t:tail}}
We now put together the five cases to establish the tail bounds.

\begin{proof}[Proof of Theorem~\ref{t:tail}]
The case of $\frac12\leq p \leq 1-L_{j+1}^{-1}$ is established in Lemma~\ref{l:A1final}.  By Lemma~\ref{l:A1Map} we have that $S_j(X) \geq \frac12$ for all $X\in \A{1}_{X,j+1}$.  Hence we need only consider $0<p<\frac12$ and cases 2 to 5.  By Lemmas~\ref{l:A2Bound}, \ref{l:A3Bound}, \ref{l:A4Bound} and~\ref{l:A5Bound} then
\begin{align*}
\mathbb{P}(S_{j+1}(X)\leq p) \leq \sum_{l=2}^5 \mathbb{P}(X\in \A{l}_{X,j+1}, S_{j+1}(X)\leq p)\leq  p^{m_{j+1}} L_{j+1}^{-\beta}.
\end{align*}
The bound for $S_{j+1}^{\mathbb{Y}}$ follows similarly.
\end{proof}

\section{Length estimate}\label{s:length}

%\subsection{Length Estimate}
\begin{theorem}
\label{lengthestimate}
Given the inductive assumptions at level $j$, for $X$ be an $\mathbb{X}$ block at at level $(j+1)$ we have that
\begin{equation}
\label{length2a}
\mathbb{E}[\exp (L_{j}^{-6}(|X|-(2-2^{-(j+1)})L_{j+1}))] \leq 1.
\end{equation}
and hence for $x \geq 0$,
\begin{equation}
\label{length2}
\mathbb{P}(|X|> ((2-2^{-(j+1)})L_{j+1}+xL_{j}^6))\leq e^{-x}.
\end{equation}
\end{theorem}

\begin{proof}
By the inductive hypothesis we have for a random $j$-level block $X_i$ that
\begin{equation}
\label{lengthalternateinduction}
\mathbb{E}[\exp (L_{j-1}^{-6}(|X|-(2-2^{-j})L_j))] \leq 1.
\end{equation}
It follows that
\begin{equation}
\label{lengthalternate2}
\mathbb{E}[\exp (L_{j-1}^{-6}(|X|-(2-2^{-j})L_j))|X\in G_j^{\mathbb{X}}] \leq \P[X\in G_j^{\mathbb{X}}]^{-1}\leq \frac1{1 - L_j^{-\delta}} \leq 1+2 L_j^{-\delta},
\end{equation}
since $L_0$ is large enough.
Since $0 \leq 2 L_j^{-6} \leq L_{j-1}^{-6}$ Jensen's inequality and equation~\eqref{lengthalternateinduction}  imply that
\begin{equation}
\label{lengthalternate3}
\mathbb{E}[\exp ( 2 L_{j}^{-6}(|X|-(2-2^{-j})L_j))] \leq 1
\end{equation}
and similarly
\begin{equation}
\label{lengthalternate4}
\mathbb{E}[\exp ( 2 L_{j}^{-6}(|X|-(2-2^{-j})L_j))|X\in G_j^{\mathbb{X}}] \leq 1+2 L_j^{-\delta}
\end{equation}
Let $\tilde{X}=(X_1,X_2,\ldots)$ be a sequence of independent $\mathbb{X}$-blocks at level $j$ with the distribution specified by $X_i\sim  \mu_{j,G}^{\mathbb{X}}$ for $i=1,\ldots ,L_j^3$ and $X_i\sim \mu_j^{\mathbb{X}}$ for $j>L_j^3$. Let $X=(X_1,X_2,\ldots , X_{L_j^{\alpha -1}+2L_j^3+T_X})$ be the $(j+1)$ level $\mathbb{X}$-block obtained from $\tilde{X}$. Then since $T_X$ is independent of the first $L_j^3$ sub-blocks we have
\begin{eqnarray*}
\mathbb{E}[\exp(L_j^{-6}|X|)] &=& \mathbb{E}[\sum_{t=1}^{\infty} \exp (L_j^{-6}\sum_{i=1}^{2L_j^3+L_j^{\alpha -1}+t}|X_i|)I[T_X=t]]\\
&=& \mathbb{E}\bigg[ \exp \Big(L_j^{-6}\sum_{i=1}^{L_j^3} |X_i|\Big)\bigg]  \\
&&\quad \cdot  \sum_{t=1}^{\infty}  \mathbb{P}[T_X=t]^{\frac12}  \mathbb{E}\bigg[\exp \Big(2 L_j^{-6}\sum_{i=L_j^3+1}^{2L_j^3+L_j^{\alpha -1}+t}|X_i|\Big)\bigg]^{\frac{1}{2}},
\end{eqnarray*}
using H\"older's Inequality. Now using (\ref{lengthalternate3}), (\ref{lengthalternate4}) and Lemma~\ref{l:totalSizeBound} it follows from the above equation that
\begin{align*}
\mathbb{E}[\exp(L_j^{-6}|X|)] &\leq  \left(1+2 L_j^{-\delta}\right)^{L_j^3} \sum_{t=1}^{\infty}  \exp \left (L_j^{-5}(2-2^{-j}) (L_j^{\alpha -1}+L_j^3+t)-\frac14 tL_j^{-4} \right)\\
&\leq 2 \exp \left((2-2^{-j})(L_j^{\alpha -6}+L_j^{-2})+L_j^{-5}(2-2^{-j})\right)\\
&\leq  \exp \left((2-2^{-(j+1)})L_j^{\alpha -6}\right),
\end{align*}
since $\alpha>6$.  It follows that
\begin{equation}
\label{lengthalternateinduction2}
\mathbb{E}[\exp (L_{j}^{-6}(|X|-(2-2^{-(j+1)})L_{j+1}))] \leq 1
\end{equation}
while equation~\eqref{length2} follows by Markov's inequality which completes the proof of the theorem.

\end{proof}

\section{Estimates for good blocks}\label{s:good}

\subsection{Most Blocks are good}

\begin{theorem}
\label{goodprobability}
Let $X$ be a $\mathbb{X}$-block at level $(j+1)$. Then $\mathbb{P}(X\in G_{j+1}^{\mathbb{X}})\geq 1-L_{j+1}^{-\delta}$. Similarly for $\mathbb{Y}$-block $Y$ at level $(j+1)$, $\mathbb{P}(Y\in G_{j+1}^{\mathbb{Y}})\geq 1-L_{j+1}^{-\delta}$.
\end{theorem}

Before proving the theorem we need the following lemma to show that a sequence of $L_j^{3/2}$ independent level $j$ subblocks  is with high probability strong.

\begin{lemma}
\label{strong}
$X=(X_1,\ldots X_{L_{j}^{3/2}})$ be a sequence of $L_j^{3/2}$ independent subblocks at level $j$. Then
\begin{enumerate}
\item[(a)]
 $$\mathbb{P}(X~\text{is ``strong"})\geq 1- e^{-\frac{L_j^{5/4}}{2}}.$$
\item[(b)]
Let, for $i=1,2,\ldots L_j^{3/2}$,  $\mathcal{E}_i=\{X^{[1,i]}~\text{is ``good"}~\}$. Then for each $i$,
$$\mathbb{P}(X~\text{is ``strong"}|\mathcal{E}_i)\geq 1- e^{-\frac{L_j^{5/4}}{2}}.$$
\end{enumerate}
\end{lemma}

\begin{proof}
We only prove part (b). Part (a) is similar.
Let $Y$ be a fixed semi-bad block at level $j$. Each of the events $\{X_k\hookrightarrow Y\}$ are independent, (they are independent conditional on $\mathcal{E}_i$) as well. Now, for $k>i$
\begin{equation}
\mathbb{P}(X_k\hookrightarrow Y|\mathcal{E}_i)\geq 1-1/20k_0 R_{j+1}^{+}
\end{equation}
and for $k\leq i$
\begin{equation}
\mathbb{P}(X_k\hookrightarrow Y|\mathcal{E}_i)\geq (1-1/20k_0 R_{j+1}^{+}-L_j^{-\delta}).
\end{equation}
We have $L_j^{\delta}>60 k_0 R_{j+1}^{+}$.
It then follows that, conditional on $\mathcal{E}_i$,
$$\#\{k:X_k\hookrightarrow Y\}\succeq V$$
where $V$ has a $Bin(L_j^{3/2}, (1-1/15k_0 R_{j+1}^{+}))$ distribution.
Using Chernoff's bound, we get
\begin{align*}
&\mathbb{P}(\#\{k:X_k\hookrightarrow Y\} \geq  L_j^{3/2}(1-1/10k_0 R_{j+1}^{+})|\mathcal{E}_i)\\
& \qquad \geq  \mathbb{P}(V\geq L_j^{3/2}(1-1/10k_0 R_{j+1}^{+})) \geq  1- e^{-\frac{L_j^{3/2}}{1800k_0^2(R_{j+1}^{+})^2}} \geq 1-e^{-L_j^{5/4}}
\end{align*}
since $L_j^{1/4}> 1800 k_0^2 (R_{j+1}^{+})^2$ for $L_0$ sufficiently large.
Since the length of a semi-bad block at level $j$ can be at most $10L_j$, and semi-bad blocks can contain only the first $L_j^m$ many characters, there can be at most $L_j^{10mL_j}$ many semi-bad blocks at level $j$.

Hence, using a union bound we get, for each $i$,

$$\mathbb{P}(X~\text{is ``strong"}|\mathcal{E}_i)\geq 1- e^{10mL_j\log L_j}e^{-L_j^{5/4}}\geq 1- e^{-\frac{L_j^{5/4}}{2}},$$
for large enough $L_0$, completing the proof of the lemma.
\end{proof}

\begin{proof}(of Theorem \ref{goodprobability})

To avoid repetition, we only prove the theorem for $\mathbb{X}$-blocks.

Recall the notation of Observation~\ref{o:blockRepresentation} with $(X_1,X_2,X_3,.....)$  a sequence of independent $\mathbb{X}$-blocks at level $j$ with the first $L_j^3$ conditioned to be good and  $X\sim \mu_{j+1}^{\mathbb{X}}$ be the $(j+1)$-th level block constructed from them. Let $W_X$ be the $Geom(L_j^{-4})$ variable associated with $X$ and $T_X$ be the number of excess blocks. Let us define the following events.

$$A_1=\{(X_{i},X_{i+1},\ldots X_{i+L_j^{3/2}}) \text{is a strong sequence for}~1\leq i\leq 2L_j^{\alpha-1}\}.$$

$$A_2=\{\#\{1\leq i\leq L_j^{\alpha -1}+2L_j^3 + T_X: X_i \notin G_j^{\mathbb{X}}\}\leq k_0 \}.$$

$$A_3=\{\#\{1\leq i\leq 2L_j^{\alpha -1}: X_i \notin G_j^{\mathbb{X}}\cup SB_j^{\mathbb{X}}\}=0\}.$$

$$A_4=\{T_X\leq L_j^5-2L_j^3\}.$$

From the definition of good blocks that, $A_1\cap A_2\cap A_3\cap A_4 \subseteq \{X\in G_{j+1}^{\mathbb{X}}\}$ and hence,
\begin{equation}
\label{goodprobability1}
\mathbb{P}(X\in G_{j+1}^{\mathbb{X}})\geq \mathbb{P}(A_1\cap A_2\cap A_3\cap A_4).
\end{equation}
Now, to complete the proof we need suitable estimates for each of the quantities $\mathbb{P}(A_i)$, $i=1,2,3,4$, each of which we now compute.

$\bullet$
Let $\tilde{X_i}=(X_{i+1},X_{i+2},\ldots ,X_{i+L_j^{3/2}})$. From Lemma \ref{strong}, it follows that for each $i$, $$\mathbb{P}(\tilde{X_i}~\text{``is strong''}~)\geq 1-e^{-\frac{L_j^{5/4}}{2}}.$$ It follows that
\begin{equation}
\label{strongest}
\mathbb{P}[A_1^c]\leq 2L_j^{\alpha-1} e^{-\frac{L_j^{5/4}}{2}},
\end{equation}

$\bullet$ By Lemma~\ref{l:totalSizeBound} we have that
\begin{equation}\label{goodprobability2}
\P[A_2] \geq 1-L_j^{-\delta k_0/4}.
\end{equation}

$\bullet$ From the definition of semi-bad blocks, we know that for $i>L_j^3$,
\begin{align*}
\mathbb{P}(X_i \notin G_j^{\mathbb{X}}\cup SB_j^{\mathbb{X}})&\leq \mathbb{P}(S_j^{\mathbb{X}}(X_i)\leq 1-\frac{1}{20k_0 R_{j+1}^{+}})+ \mathbb{P}(|X_i|>10L_j)\\
&\qquad+\mathbb{P}(C_k\in X_i ~\text{for some}~k>L_j^m).
\end{align*}
The Strong Law of Large number implies that $$\mathbb{P}(C_k\in X_i ~\text{for some}~k>L_j^m)\leq \mu^{\mathbb{X}}(\{C_{L_j^m+1}, C_{L_j^{m}+2},\ldots\})\mathbb{E}(|X_i|).$$ Using (\ref{mux}) and Lemma \ref{lengthestimate}, it follows that  $\mathbb{P}(C_k\in X_i ~\text{for some}~k>L_j^m)\leq 3L_j^{1-m} \leq L_j^{-\beta}$ for $L_0$ large enough and since $m> 2+\beta$.

Now using the recursive tail estimates and~\ref{lengthestimate}, we see that
$$\mathbb{P}(X_i \notin G_j^{\mathbb{X}}\cup SB_j^{\mathbb{X}})\leq (1-\frac{1}{20k_0 R_{j+1}^{+}})^{m}L_j^{-\beta}+ \mathbb{P}(|X_i|>10L_j)+L_j^{-\beta}\leq 3L_j^{-\beta}$$
for $L_0$ sufficiently large.
Hence it follows that
\begin{equation}
\label{goodprobability3}
\mathbb{P}[A_3^c]\leq  6L_j^{\alpha-\beta-1} \leq 1/10 L_{j+1}^{-\delta}
\end{equation}
for sufficiently large $L_0$ since $\beta> \alpha \delta +\alpha -1$.

$\bullet$By Lemma~\ref{l:totalSizeBound} we have that
\begin{equation}\label{goodprobability4}
\P[A_4] \geq 1-2\exp(-\frac14 L_j) \geq 1-\frac{1}{10}L_{j+1}^{-\delta}.
\end{equation}

Now from (\ref{goodprobability1}), (\ref{strongest}), (\ref{goodprobability2}), (\ref{goodprobability3}), (\ref{goodprobability4})
it follows that,
\[
\mathbb{P}(X\in G_{j+1}^{\mathbb{X}})\geq 1-\sum_{i=1}^4 \mathbb{P}[A_i^c]\geq  1-L_{j+1}^{-\delta},
\]
completing the proof of the theorem.
\end{proof}

\subsection{Mappings of good segments:}

%\begin{lemma}
%Let $X=(X_1,X_2,\ldots )$ be a sequence of $\mathbb{X}$ blocks at level $j$ and $Y=(Y_1,Y_2,\ldots )$ be a sequence of $\mathbb{Y}$ blocks at level $j$.
%Further suppose that $n,n'$ are such that $X^{[1,n]}$ and $Y^{[1,n']}$ are both ``good" segments, $n>L_j$ and $\frac{1-2^{-(j+1/2)}}{R}\leq \frac{n'}{n}\leq R(1+2^{-(j+1/2)})$. Then $X^{[1,n]}\hookrightarrow Y^{[1,n']}$.
%\end{lemma}
%
%\begin{proof}
%Let us write $n=kR_j+r$ where $0\leq r<R_j$ and $k\in \mathbb{N}\cup \{0\}$.
%Now let $s=[\frac{n'-r}{k}]$. Define $0=t_0<t_1<t_2<\ldots t_k=n'-r\leq t_{k+1}=n'$ such that for all $i\leq k$, $t_{i}-t_{i-1}=s$ or $s+1$.
%Now from $\frac{1-2^{-(j+1/2)}}{R} \frac{n'}{n}\leq R(1+2^{-(j+1/2)})$, it follows that for large enough $L_j$,
%$R_j^-+1\leq s \leq R_j^{+}-1$. From the inductive hypothesis on compression, it follows that, $X^[iR_j+1,(i+1)R_j]\hookrightarrow Y^{[t_i,t_{i+1}]}$ for $0\leq i\leq k$. The lemma follows.
%\end{proof}

\begin{theorem}
\label{compresslemma}
Let $\tilde{X}=(\tilde{X}_1,\tilde{X}_2,\ldots)$ be a sequence of $\mathbb{X}$-blocks at level $(j+1)$ and $\tilde{Y}=(Y_1,Y_2,\ldots)$ be a sequence of $\mathbb{Y}$-blocks at level $j$. Further we suppose that $\tilde{X}^{[1,R_{j+1}]}$ and $\tilde{Y}^{[1,R_{j+1}^{+}]}$ are ``good segments". Then for every $t$ with $R_{j+1}^{-}\leq t \leq R_{j+1}^{+}$,
\begin{equation}
\label{compress2}
\tilde{X}^{[1,R_{j+1}]}\hookrightarrow \tilde{Y}^{[1,t]}.
\end{equation}
%\end{enumerate}
\end{theorem}
\begin{proof}
Let us fix $t$ with $R_{j+1}^{-}\leq t \leq R_{j+1}^{+}$.
Let $\tilde{X}^{1,R_{j+1}}=(X_1,X_2,\ldots,X_n)$ be the decomposition of $\tilde{X}^{[1,R_{j+1}]}$ into level $j$ blocks. Similarly let $\tilde{Y}^{[1,t]}=(Y_1,Y_2,\ldots,Y_{n'})$ denote the decomposition of $\tilde{Y}^{[1,t]}$, into level $j$ blocks.

Before proceeding with the proof, we make the following useful observations.
Since both $\tilde{X}^{[1,R_{j+1}]}$ and $\tilde{Y}^{[1,t]}$ are good segments, it follows that $R_{j+1}L_j^{\alpha-1}\leq n \leq R_{j+1}(L_j^{\alpha-1}+L_j^{5})$, and $tLj^{\alpha-1}\leq n'\leq t(L_j^{\alpha-1}+L_j^{5})$. Hence for $L_0$ large enough, we have

\begin{equation}
\label{lengthratio1}
\frac{1-2^{-(j+7/4)}}{R}\leq \frac{n'}{n} \leq R(1+2^{-(j+7/4)}).
\end{equation}

Let $B_X=\{1\leq i \leq n: X_i\notin G_j^{\mathbb{X}}\}=\{l_1<l_2<\ldots <l_{K_X}\}$ denote the positions of ``bad" $\mathbb{X}$-blocks. Similarly let $B_Y=\{1\leq i \leq n': Y_i\notin G_j^{\mathbb{Y}}\}=\{l'_1<l'_2<\ldots <l'_{K_Y}\}$ denote the positions of ``bad" $\mathbb{Y}$-blocks. Notice that $K_X,K_Y \leq k_0R_{j+1}^{+}$. Using Proposition \ref{compresspartition} we can find
a family of admissible generalised mappings of $Class~G^j$, $\Upsilon_h$ for $0\leq h \leq L_j^2$, $\Upsilon_h([n],[n'],B,B')=(P_h,P'_h,\tau_h)$ such that for all $h$, $1\leq h\leq L_j^2$, $1\leq i\leq K_X$, $1\leq r \leq K_X$, $\tau_h(l_i)=\tau_{1}(l_i)+h-1$ and $\tau_h^{-1}(l'_r)=\tau_1^{-1}(l'_r)-h+1$.
At this point we need the following Lemma.

\begin{lemma}
\label{compresssemibad}
Let $\Upsilon_h=(P_h,P'_h,\tau_h), 1\leq h\leq L_j^2$ be the set of generalised mappings as described above. Then there exists $1\leq h_0 \leq L_j^2$, such that $X_{l_i}\hookrightarrow Y_{\tau_{h_0}(l_i)}$ for all $1\leq i\leq K_X$ and $X_{\tau_{h_0}^{-1}(l'_i)}\hookrightarrow Y_{l'_i}$ for all $1\leq i \leq K_Y$.
\end{lemma}

\begin{proof}
Once again we appeal to the probabilistic method. First observe, for any fixed $i$, $\{\tau_h(l_i):h=1,2,\ldots L_j^2\}$ is a set of $L_j^2$ consecutive integers. Notice that the $j$-th level sub-blocks corresponding to these indices need not belong to the same $(j+1)$-th level block. However, also notice that they can belong to at most 2 consecutive $(j+1)$-level blocks (both of which are good). Suppose the number of sub-blocks belonging to the two different blocks are $a$ and $b$, where $a+b=L_j^2$. Now, by the strong sequence assumption, these $L_j^2$ blocks the must contain at least $\lfloor \frac{a}{L_j^{3/2}}\rfloor +\lfloor \frac{b}{L_j^{3/2}}\rfloor \geq \lfloor L_j^{1/2}\rfloor -2$ many disjoint strong sequences of length $L_j^{3/2}$. By definition of strong sequences then, there exist, among these $L_j^2$ sub-blocks, at least $(L_j^{1/2} -3)L_j^{3/2}(1-\frac{1}{10k_0R_{j+1}^{+}})$ many to which $X_{l_i}$ can be successfully mapped, i.e.,

\begin{equation}
\label{probmethodcompression1}
\#\{h: X_{l_i}\hookrightarrow Y_{\tau_h(l_i)}\}\geq (L_j^{1/2} -3)L_j^{3/2}(1-\frac{1}{10k_0R_{j+1}^{+}}).
\end{equation}

Now, choosing $H$ uniformly at random from $\{1,2,\ldots ,L_j^2\}$, it follows from (\ref{probmethodcompression1}) that

\begin{equation}
\label{probmethodcompression2}
\mathbb{P}(X_{l_i}\hookrightarrow Y_{\tau_H(l_i)})\geq (1-3/L_j^{1/2})(1-\frac{1}{10k_0R_{j+1}^{+}})\geq 1-\frac{1}{10k_0R_{j+1}^{+}}-3/L_j^{1/2}.
\end{equation}
Similar arguments show that for all $i\in \{1,2,\ldots ,k_Y\}$,
\begin{equation}
\label{probmethodcompression3}
\mathbb{P}(X_{\tau_H^{-1}(l'_i)}\hookrightarrow Y_{l'_i})\geq 1-\frac{1}{10k_0R_{j+1}^{+}}-3/L_j^{1/2}.
\end{equation}
A union bound then gives,
\begin{equation}
\mathbb{P}(X_{l_i}\hookrightarrow Y_{\tau_H(l_i)}: 1\leq i \leq K_X,~X_{\tau_H^{-1}(l'_i)}\hookrightarrow Y_{(l'_i)}: 1\leq i \leq K_Y)>1- 2k_0R_{j+1}^{+}(\frac{1}{10k_0R_{j+1}^{+}}+3/L_j^{1/2}),
\end{equation}
and the right hand side is always positive for $L_0$ sufficiently large. The lemma immediately follows from this.
\end{proof}

The proof of the theorem can now be completed using Proposition \ref{admissiblemap}.
\end{proof}

\subsection{Good Blocks Map to Good Blocks}

\begin{theorem}
\label{goodongood}
Let $X\in G_{j+1}^{\mathbb{X}}, Y\in G_{j+1}^{\mathbb{Y}}$, then $X\hookrightarrow Y$.
\end{theorem}
The theorem follows from a simplified version of the proof of  Theorem~\ref{compresslemma} so we omit it.\\

We have now completed all the parts of the inductive step.  Together these establish Theorem~\ref{induction}.

\section{Explicit Constructions}\label{s:deterministic}

Our proof  provides an implicit construction of a deterministic sequence $(X_1,\ldots)$ which embeds into $\mathbb{Y}$ with positive probability.  We will describe a deterministic  algorithm, based on our proof, which for any $n$ will return the first $n$ co-ordinates of the sequence in finite time. Though it is not strictly necessary,  we will restrict discussion to the case of finite alphabets.  It can easily be seen from our construction and proof that any good $j$-level block can be extended into a $(j+1)$-level good block and so the algorithm proceeds by extending one good block to one of the next level.  As such one only needs to show that we can identify all the good blocks at level $j$ in a finite amount of time.

We will also recover all semi-bad blocks.
By our construction all good and semi-bad blocks are of bounded length so there is only a finite space to examine.  To determine if a block is good at level $j+1$ one needs to count how many of its sub-blocks at level $j$ are good and verify that the others are semi-bad.  It also requires that it has ``strong subsequences''.  This can be computed if we have a complete list of semi-bad $j$-level blocks.

Determining if $X$, a $(j+1)$-level block, is semi-bad requires calculating its length and its embedding probability.  For this we need to run over all possible $(j+1)$-level blocks, calculate their probability and then test if $X$ maps into them.  By the definition of an $R$-embedding we need only consider those of length at most $O(R |X|)$ so this can be done in finite time.

With this listing of all good blocks one can then construct in an arbitrary manner a sequence in which the first block is good at all levels which will have a positive probability of an $R$-embedding into a random sequence.  From this construction, and the reduction in \S~\ref{s:applicaiton}, we can construct a deterministic sequences which has an $M$-Lipshitz embedding into a random binary sequence in the sense of Theorem~\ref{t:embed} with positive probability.  Similarly, this approach  gives a binary sequence with a positive density of ones which is compatible sequence with a random $\mbox{Ber}(q)$ sequence in the sense of Theorem~\ref{t:compatible} for small enough $q>0$ with positive probability.

\subsection*{Acknowledgements}

We would like to thank Vladas Sidoravicius for describing his work on these problems and Peter G\'{a}cs for informing us of his new results.  We would also like to thank Geoffrey Grimmett and Ron Peled for introducing us to the Lipschitz embedding and rough isometry problems and to Peter Winkler for very useful discussions.

\bibliography{embed}
\bibliographystyle{plain}

\end{document}